\documentclass[a4paper, reqno]{amsart}

\usepackage[british]{babel}
\usepackage[british]{isodate}
\cleanlookdateon

\usepackage{amssymb, amsfonts, amsthm,amsmath,amsaddr}
\usepackage{setspace}

\usepackage{enumerate}
\usepackage{enumitem}
\usepackage{dsfont}

\newcommand{\bw}{\mathbf{w}}
\newcommand{\eps}{\varepsilon}
\sffamily

\newcommand{\PPo}[1]{\PP \left[\,#1\,\right]}

\newcommand{\Ex}[1]{\mathbb{E} \left[\, #1\,\right]}

\newcommand{\VV}{\operatorname{Var}}
\newcommand{\EE}{\mathbb{E}}
\newcommand{\PP}{\mathbb{P}}

\newcommand{\A}{{\ensuremath{\mathcal{A}}}}
\newcommand{\KernelDense}{\mathcal{E}_d}
\newcommand{\KernelSparse}{\mathcal{E}_{s}}
\newcommand{\LayersEvent}[1]{\mathcal{E}_{#1}}

\newcommand{\VS}[4]{{\ifx&#3&{\ifx&#2&V\else\widehat{V}\fi}\else{\ifx&#2&#3\else\widehat{#3}\fi}\fi}_{#1}{\ifx&#4&\else[#4]\fi}} 
\newcommand{\VertexSetRestricted}[1]{\VS{#1}{}{}{}}
\newcommand{\InfectedVertexSetRestricted}[1]{\VS{#1}{*}{}{}}
\newcommand{\totalWeight}[1]{\VS{}{}{W}{#1}}
\newcommand{\WeightLayers}[1]{W^{(#1)}}
\newcommand{\InfectedWeightLayers}[1]{\totalWeight{\postKernelProcess{#1}{U}\cap\Layer{#1}}}
\newcommand{\InfectedWeightKernel}[1]{\totalWeight{\postKernelProcess{#1}{U}\cap\Kernel{#1}}}
\newcommand{\totalWeightRestricted}[1]{\VS{#1}{}{W}{}}

\newcommand{\totalInfectedWeightRestricted}[1]{\VS{#1}{*}{W}{}}
\newcommand{\infdeg}[2]{\hat{d}^{(#2)}(#1)}
\newcommand{\neighbourhood}[1]{N(#1)}

\newcommand{\VertexSet}{\VertexSetRestricted{}}
\newcommand{\EdgeSet}{E}
\newcommand{\maxWeight}[1]{w_{\operatorname{max}}[#1]}
\newcommand{\minWeight}[1]{w_{\operatorname{min}}[#1]}

\newcommand{\Kernel}[1]{\mathcal{K}_{#1}}

\newcommand{\wcore}[1]{\totalWeightRestricted{\ge\weightBoundLayer{#1}}}

\newcommand{\sbwsmoment}[1]{\kappa_{#1}}

\newcommand{\weightConstant}{\lambda}
\newcommand{\infectionRate}{p_0}

\newcommand{\weightBoundHeavy}{\phi_{H}}
\newcommand{\weightBoundKernel}{\phi_\kernel}
\newcommand{\weightBoundBreedingGround}{\phi_s}
\newcommand{\weightBoundInitialInfection}{\phi_0}
\newcommand{\weightBoundLayer}[1]{\psi_{#1}}
\newcommand{\candidateThresholdSparse}{p_s}
\newcommand{\candidateThresholdDense}{p_d}
\newcommand{\threshold}{p_c}

\newcommand{\infec}[1]{\mathcal{A}_{#1}}
\newcommand{\infecCoupling}[1]{\mathcal{B}_{#1}}
\newcommand{\infecNonHeavy}[1]{\mathring{\mathcal{C}}_{#1}}

\newcommand{\infeclast}[1]{\widetilde{\mathcal{A}}_{#1}}

\newcommand{\cI}[1]{\mathcal{I}_{#1}}

\newcommand{\I}[1]{I_{#1}}
\newcommand{\postKernelProcess}[2]{\A_{#1}^{#2}}

\newcommand{\kernel}{K}
\newcommand{\breedingGround}{S}

\newcommand{\Layer}[1]{\mathcal{S}_{#1}}

\newcommand{\scaledTotalWeightSparse}[1]{\nu_{#1}}
\newcommand{\scaledTotalWeightSparseExpectation}{\overline{\nu}_s}

\newcommand{\scaledTotalWeightDenseExpectation}{\overline{\nu}_d}
\newcommand{\scaledTotalWeightSubcriticalExpectationInverse}{\mu}

\newcommand{\lastLayer}{i^*}

\newcommand{\stopEvent}[1]{\mathcal{T}_{#1}}
\newcommand{\goodEvent}[1]{\mathcal{G}_{#1}}
\newcommand{\inductionLayerEvent}[1]{\mathcal{W}_{#1}}
\newcommand{\largeLayerEvent}[1]{\mathcal{L}_{#1}}
\newcommand{\event}{\mathcal{E}}

\newcommand{\wbrv}{Z_{\bw}}
\newcommand{\cond}{\; \middle\vert \;}
\newcommand{\Ind}[1]{\mathds{1}_{\left\{#1\right\}}}

\newcommand{\offspringWeightRatio}{\chi}

\newcommand{\quotmom}{\Upsilon}

\newcommand{\brokenUpperSBD}{\mathcal{S}}

\newtheorem{theorem}{Theorem}[section]

\newtheorem{lemma}[theorem]{Lemma}

\newtheorem{claim}[theorem]{Claim}
\newtheorem{corollary}[theorem]{Corollary}
\newtheorem{remark}[theorem]{Remark}

\newtheorem{proposition}[theorem]{Proposition}

\newtheorem{definition}[theorem]{Definition}
\newtheorem{example}[theorem]{Example}

\usepackage{color}

\title[A critical phenomenon in bootstrap processes]{A phase transition regarding the evolution of bootstrap processes in inhomogeneous random graphs}

\author{Nikolaos Fountoulakis\textsuperscript{1}}
\email{n.fountoulakis@bham.ac.uk}
\thanks{\textsuperscript{1}supported by the EPSRC Grant No. EP/K019749/1. Part of this research was carried out while this author was visiting TU Graz, supported by TU Graz and Austrian Science Fund (FWF) P26826.}
\address{School of Mathematics,\\ University of Birmingham,\\ B15 2TT, Birmingham, United Kingdom}
\author[N. Fountoulakis, M.~Kang, C.~Koch, and T.~Makai]{Mihyun Kang\textsuperscript{2}, Christoph Koch\textsuperscript{3}, and Tam\'as Makai\textsuperscript{4}}
\email{\{kang, ckoch, makai\}@math.tugraz.at}
\thanks{\textsuperscript{2}supported by Austrian Science Fund (FWF): P26826.}
\thanks{\textsuperscript{3}supported by Austrian Science Fund (FWF): W1230, P26826.}
\thanks{\textsuperscript{4}supported by Austrian Science Fund (FWF): P26826. Part of this research was carried out while this author was a postdoctoral research associate at the School of Mathematics, University of Birmingham supported by the Marie Curie CIG PCIG09-GA2011-293619.}
\address{Institute of Discrete Mathematics,\\Graz University of Technology,\\ Steyrergasse 30, 8010 Graz, Austria}

\date{\today}

\begin{document}

\begin{abstract}
A bootstrap percolation process on a graph with 
infection threshold $r\ge 1$ is a dissemination process that evolves in time steps. 
The process begins with a subset of infected vertices and in each subsequent step every uninfected vertex that has at least $r$ infected neighbours becomes infected 
and remains so forever.

Critical phenomena in bootstrap percolation processes were originally observed by Aizenman and Lebowitz in 
the late 1980s as finite-volume phase transitions in $\mathbb{Z}^d$ that are caused by the accumulation of small local islands of infected 
vertices. They were also observed in the case of dense (homogeneous) random graphs by Janson, \L uczak, Turova and Vallier (2012). 
In this paper, we consider the class of inhomogeneous random graphs known as the \emph{Chung-Lu model}: 
 each vertex is equipped with a positive weight 
and each pair of vertices appears as an edge with probability proportional to the product of the weights. In particular, we focus on the \emph{sparse} regime, where
 the number of edges is proportional to the number of vertices. 

The main results of this paper determine those weight sequences for which a critical phenomenon occurs:
there is a critical density of vertices that are infected at the beginning of the process, above which 
a small (sublinear) set of infected vertices creates an avalanche of infections that in turn leads to an outbreak. 
We show that this occurs essentially only when the tail of the weight distribution dominates a power law with exponent 3 and we determine
the critical density in this case.  
\medskip

\noindent
\texttt{keywords}: bootstrap percolation, inhomogeneous random graphs, critical phenomena.

\noindent
\emph{2010 AMS Subj. Class.}: 05C80, 82B43; 60K37, 82B26
\end{abstract}
\maketitle
\setcounter{footnote}{4}

\section{Introduction}
A \emph{bootstrap percolation process} with \emph{infection threshold}  $r\geq 1$
is a dissemination process on a graph $G$ which evolves in steps.
Each vertex may have one of two possible states: it is either \emph{infected} or \emph{uninfected}.
At the beginning, there is a subset $\infec{0}$ of \emph{initially infected} vertices, and all remaining vertices are uninfected. In each subsequent step, any uninfected vertex with at least $r$ infected neighbours also becomes infected, and never changes its state again. The set of vertices infected 
until step $t\ge 0$ is denoted by $\infec{t}$. The process will stop once there is a step in which no vertices became infected. In particular, if $G$ is a finite graph then the process always stops, and we denote by $\infec{F}$ the set of all vertices which became infected throughout the entire process. 

The bootstrap percolation process was introduced in the context of magnetic disordered systems by Chalupa, Leath, and Reich~\cite{ChLeRe:79} in 1979. Since then it has been used as a model for several phenomena in various areas, from jamming transitions~\cite{tobifi06} and magnetic systems~\cite{sadhsh02} to neuronal activity~\cite{Am-nn, ET09}. Certain variations of this process are related to the dynamics of the Ising model at zero temperature~\cite{Fontes02,GlauberMorris2009}. A short survey regarding applications can be found in~\cite{AdL03}.

Several qualitative characteristics of bootstrap percolation, for instance the dependence of the final set of infected vertices $\infec{F}$ on the set $\infec{0}$ of initially infected vertices, have been studied on a variety of families of graphs, such as trees~\cite{BPP06, BGHJP, FS08}, grids~\cite{BBDM2010, ar:BaloghPete, CM02, holroyd03}, lattices on the hyperbolic plane~\cite{BootHyper2013}, and hypercubes~\cite{BB06}, as well as on many models of random graphs~\cite{Am-bp, balpit07, ar:JLTV10}. 

The most well-studied quantity is the probability that all vertices of the underlying graph are eventually infected. 
In particular, this quantity has been considered as a function of the density $\infectionRate$ of initially
 infected vertices. More specifically, assuming that before the process begins each vertex of the graph 
is independently infected with probability $\infectionRate$, what is the probability that the final set contains every 
vertex? In other words, what is the probability that the process \emph{percolates}? 

In several families of infinite graphs, it turns out that there is a critical value for $\infectionRate$ above which the probability of percolation is positive. 
This is the case for the family of infinite regular trees with degree $d+1$ and $d \geq r$, as it was proved
by Balogh, Peres, and Pete~\cite{BPP06}, as well as for (infinite) 
Galton-Watson trees (this was shown by Bollob\'as, Gunderson, Holmgren, Janson, and Przykucki~\cite{BGHJP}). Fontes and Schonmann~\cite{FS08} showed that infinite regular trees also exhibit two thresholds: 
a critical density $p_f$ (that was proved in~\cite{BPP06}) above which percolation occurs almost surely and 
a critical density $p_c < p_f$ above which infinite infected clusters exist almost surely.

A large part of the literature on bootstrap percolation processes has been devoted to the $d$-dimensional integer lattice $\mathbb{Z}^d$. Schonmann~\cite{Schon92} showed that if the elements of $\infec{0}$ are selected independently with probability $\infectionRate$, then the evolution of the process is in some sense ``trivial'':
if $r\leq d$, then for every $\infectionRate>0$ all vertices of the lattice become infected with probability 1, whereas if 
$r> d$, then this does not happen unless $\infectionRate=1$. The former had already been shown by van Enter~\cite{vanEnter} in 1987 for $d=r=2$. Moreover, for this case,  Aizenman and Lebowitz~\cite{AizLeb88} identified a phase-transition phenomenon when the process is restricted to a box of $\mathbb{Z}^2$ of side-length $n\to \infty$, which 
Holroyd~\cite{holroyd03} later made precise. Let $r=2$ and $G$ be the $2$-dimensional grid with vertex set $V=\{1,\dots, n\}^2$, and let $\infec{0}\subseteq V$ be a random subset
containing every element independently with probability $\infectionRate=\infectionRate(n)$. Holroyd~\cite{holroyd03} showed that the probability $I(n, \infectionRate)$ that the entire 
square is eventually infected satisfies $I(n,\infectionRate) \rightarrow 1$ if $\liminf_{n \rightarrow \infty} \infectionRate(n) \log n > \pi^2/18$, and 
$I(n,\infectionRate) \rightarrow 0$ if $\limsup_{n \rightarrow \infty} \infectionRate (n) \log n < \pi^2/18$. 
This result has been generalised to higher dimensions by Balogh, Bollob\'as, and Morris~\cite{BBM09} (when $G$ is the 3-dimensional grid on $\{1,\dots, n\}^3$ and $r=3$) and Balogh, Bollob\'as, Duminil-Copin, and Morris~\cite{BBDM2010} (in general).
This is an instance of the so-called \emph{metastability phenomenon}. 

Similar thresholds have been identified in the case of the binomial random graph $G(n,p)$, where every edge on a set of $n$ vertices is present independently with probability $p$. Janson, {\L}uczak, Turova, and Vallier~\cite{ar:JLTV10} presented a complete analysis of the bootstrap percolation process for various ranges of $p$. Among other results they showed that when $1/n \ll p \ll n^{-1/r}$, there is a critical function $a_c = a_c(n)$ such that with high probability\footnote{with probability tending to one as $n\to\infty$} the following occurs: if $\infectionRate \ll a_c /n$, then there is very little evolution of the process, whereas if $\infectionRate \gg a_c/n$, then eventually almost every vertex becomes infected. For sparser graphs, this is not the case. When $p= c/n$ (that is,  the average degree is approximately constant) and if $\infectionRate=o(1)$, then only a sub-linear number of vertices will ever be infected with high probability. In fact, no evolution occurs with high probability. This had been observed previously by Balogh and Bollob\'as (see~\cite{balpit07}).

However, this is no longer the case if one considers sparse random graphs which are \emph{inhomogeneous}. We focus on random graphs which are defined through a sequence of weights 
assigned to the vertices: these weights determine the probability that two vertices are adjacent. More specifically, we are interested in the case where this probability is proportional to the product of the weights of these vertices. Hence, pairs of vertices where at least one of them has  high weight are more likely to appear as edges. Of course, $G(n,p)$ is a special case of such a random graph, in which all 
vertices have the same weight. The first author together with Amini~\cite{ar:AmFount2012} showed that such  a threshold does exist when the sequence of weights follows a power law distribution with 
exponent in the interval $(2,3)$. 
They showed that there is a function $a_c=a_c (n)= o(n)$ such that if $\infectionRate \ll a_c/n$, then with high probability no evolution occurs, but if $\infectionRate\gg a_c /n$, then even if $\infectionRate=o(1)$, with high probability a constant fraction of all vertices become infected eventually. 
In addition the first author together with Amini and Panagiotou~\cite{ar:AmFountPan2014}, determined the value of this constant.  
More general results which include those in~\cite{ar:AmFount2012} were obtained by Karbasi, Lengler and Steger~\cite{ar:KarLengSteg15}.
Similar behaviour was also observed in geometric random graph models that 
exhibit a power law degree distribution with such exponent~\cite{ar:CandFount2016, ar:KochLengler2016} as 
well as in several versions of the preferential attachment model~\cite{ar:AbdFount2014, ar:EGGS2014}.

 The aim of this paper is to determine the conditions on the sequence of weights which characterise the existence of such a threshold function. 
We show that a \emph{critical droplet} (in the terminology of~\cite{AizLeb88}) is formed by a certain set of vertices of high weight, which we call the \emph{nucleus} of the infection. 
Informally, this is a set consisting of vertices of high weight which become infected at some stage and from this set the infection spreads to a positive fraction of the vertices 
of the random graph. 
 Effectively, we show that such a nucleus is formed, if the tail of the empirical distribution function of the weight sequence dominates the tail of a power law distribution with exponent equal to 3. 
 Furthermore, we determine the critical density of the initially infected vertices below which this phenomenon does not occur. 

\section{Model and notation}

\subsection{Inhomogeneous random graphs}
The random graph model that we consider is a special yet general enough version of an  
\emph{inhomogeneous random graph} introduced by S\"oderberg~\cite{ar:s02} 
and studied in its full generality by Bollob\'as, Janson, and Riordan in~\cite{ar:bjr07}.
The model is asymptotically equivalent to a model considered by Chung and Lu~\cite{CL02,CL03} and Chung, Lu, and Vu~\cite{ar:cl04}. They analysed several typical properties of the resulting graphs, including the average distance between two randomly chosen vertices that belong to the same component and the distribution of the component sizes. 

Let $n\in\mathbb{N}$ and $[n]:=\{1,\dots,n\}$. We consider a graph $G=(\VertexSet,\EdgeSet)$ with vertex set~$\VertexSet:=[n]$ and a random edge set $\EdgeSet$ defined as follows: each vertex~$v$ is assigned a positive weight~$w_v=w_v(n)\in \mathbb R^+$, and without loss of generality we will assume throughout the paper that $w_1 \leq w_2 \leq \cdots \leq w_n$. We denote this \emph{weight sequence} by $\bw=\mathbf{w}(n) := (w_1, \dots, w_n)$ and the \emph{total weight} by $W := \sum_{v\in \VertexSet} w_v$. Any two distinct vertices $u,v\in V$ form an edge, i.e.,\ $\{u,v\}\in \EdgeSet$, independently with probability 
\begin{equation}
\label{eq:pijCL}
	p_{u,v}=p_{u,v}(\bw) := \min\left\{\frac{w_u w_v}{W},1\right\}.
\end{equation}
We refer to this model as the \emph{Chung-Lu random graph} $CL(\bw)$. A fundamental observation is that the weights (essentially) determine the \emph{expected} degrees of all vertices: if we ignore the minimum in~\eqref{eq:pijCL}, and also allow for a loop at vertex~$u$, then the expected degree of that vertex is~$\sum_{v\in \VertexSet} w_uw_v/W = w_u$.

For the sake of a more concise exposition of the results and proofs, we assume that the minimal weight $w_1$ is at least $1$ and the total weight satisfies 
\begin{align}\label{Wcondition}
 W=\weightConstant n,\quad\text{for some}\quad \weightConstant\ge 1.
\end{align}

Central in our results will be the distribution of the weight of a vertex selected randomly with probability proportional to its weight.
More formally, let $X$ denote a $\VertexSet$-valued random variable whose distribution is given by $\PP[X=u]=\frac{w_u}{W}$ 
for any vertex $u\in\VertexSet$ of weight $w_u$. Then the weight $w_X$ of this randomly chosen vertex $X$ is a 
$\mathbb{R}^+$-valued random variable $w_X$ whose distribution function is given by 
$$
\PP[w_X \le a]=\sum_{u : w_u \leq a}\frac{w_u}{W},
$$
for any $a \in \mathbb{R}$.
This distribution (of $w_X$) is called the \emph{size-biased distribution} and 
we denote a weight chosen randomly according to this distribution by $\wbrv$.

Let $\bw = \bw (n)$ be a sequence of weight sequences. 
We will consider events on the probability space that is the product of the space induced by $CL(\bw)$ and 
the one representing the set of initially infected vertices $\mathcal{A}_0 \subset \VertexSet$, where each vertex is initially infected independently 
with probability $p_0 = p_0 (n)$. 
We let $\Omega_n$ denote the sequence of these spaces.
If $\{\mathcal{E}_n\}_{n \in \mathbb{N}}$ is a sequence of events with $\mathcal{E}_n \subseteq \Omega_n$, 
we say that they occur \emph{with high probability} (whp) 
if the probability of $\mathcal{E}_n$ tends to $1$ as $n\to\infty$. Moreover, any unspecified limits and asymptotics will be as 
$n\to\infty$. 

We will also be using a probabilistic version of the standard Landau notation. 
Let $X_n$ be a sequence of non-negative random variables (where for each $n$ the variable $X_n$ is defined on $\Omega_n$)  
and $y_n$ be a sequence of non-negative real numbers. We say that whp $X_n = O(y_n)$ if there exists $C > 0$ such that $X_n \leq C y_n$ whp, and whp $X_n = \Omega(y_n)$ if there exists $C > 0$ such that $X_n \geq C y_n$ whp. 
If both hold, we say that whp $X_n = \Theta (y_n)$. 
Furthermore we say that whp $X_n = o(y_n)$ if for any $\eps >0$ whp $X_n/ y_n < \eps$ - in other words, 
$X_n/y_n$ converges to 0 \emph{in probability}. 

We sometimes also write $x_n \ll y_n$ to denote $x_n = o(y_n)$ and $x_n\gg y_n$ to denote $x_n=\omega(y_n)$ 
(as in the standard Landau notation), for two sequences of non-negative real numbers. Also, the meaning 
of ``whp $X_n \ll y_n$'', where $\{ X_n \}$ is a sequence of non-negative random variables on $\Omega_n$, is now obvious from 
the above. 

\subsection{Bootstrap processes}\label{sec:bootstrapPercolation}
Consider an integer $r\ge 2$ and a weight sequence $\bw=(w_1,\ldots, w_n)$.
Let $G\sim CL(\bw)$ be a Chung-Lu random graph with weight sequence $\bw$.  
Given an \emph{initial infection rate} $\infectionRate=\infectionRate(n)\in[0,1]$, 
we \emph{initially} infect a random subset $\infec{0}\subseteq \VertexSet$ which contains each element with probability $\infectionRate$ independently. 
As we already mentioned, a bootstrap process is a process evolving in discrete time steps. At any time $t\ge 0$ there is a set $\infec{t}$ of \emph{infected} vertices, which is defined iteratively by
\[
\infec{t+1} := \infec{t} \cup \left\{v \in \VertexSet \cond \text{ $v$ has at least $r$ neighbours in $\infec{t}$}\right\}
\]
for all $t\geq 0$. Furthermore, we set $ \infec{F} := \bigcup_{t\ge 0} \infec{t}$. 

\subsection*{Critical function}
A function $\threshold=\threshold(n)$, $0\le \threshold\le 1$, is called a \emph{critical function} (with respect to the initial infection rate $\infectionRate$) if the following three conditions are satisfied:
\begin{enumerate}
\item $\threshold = o(1)$; 
\item if $\infectionRate \ll \threshold$, then whp $|\infec{F}|=o(n)$;
\item if $\infectionRate \gg \threshold$, then whp $|\infec{F}|=\Theta(n)$.
\end{enumerate}
We refer to the  two latter cases as the \emph{subcritical} and the \emph{supercritical} regime, respectively. 
Of course, the above definition yields a class of functions that have the same order of magnitude. With slight abuse of notation we will be referring to
 \emph{the} threshold $\threshold$, and treat it as if it was uniquely defined.

\subsection{Some more notation}\label{sec:not}
Given an interval $I\subseteq \mathbb{R}^+$, we denote the set of vertices having weight in $I$ by 
$$\VertexSetRestricted{I}:=\{v\in \VertexSet\colon w_v\in I\},$$
and in particular if $I=[f,\infty)$ for some $f\in\mathbb{R}^+$ we write ``$\ge f$'' instead of ``$[f,\infty)$'', and likewise for intervals $(f,\infty)$, $(0,f]$, and $(0,f)$ .

 For a set $U\subseteq V$ of vertices we write $\totalWeight{U}$ for the sum of the weights of the vertices in $U$ and $\maxWeight{U}$, $\minWeight{U}$ for their maximal and minimal weight, respectively. Moreover, we denote by
 $$
  \widehat{U}:=U\cap \infec{F}
  $$
   the subset of all those vertices in $U$ which eventually become infected. In case $U=\VertexSetRestricted{I}$ for some interval $I\subseteq\mathbb{R}^+$, we also use the abbreviations 
  $$
  \totalWeightRestricted{I}:=\totalWeight{\VertexSetRestricted{I}} \qquad\text{and}\qquad \totalInfectedWeightRestricted{I}:=\totalWeight{\InfectedVertexSetRestricted{I}}.
  $$

 We denote by $\binom{U}{\ell}$ the set of all $\ell$-element subsets of a set $U$ and for a vertex $u\in \VertexSet$ we denote by $N(u)$ the set of neighbours of $u$ in $G$. 
Furthermore, we write $X\sim \mathcal{D}$ for a random variable with distribution $\mathcal{D}$.

\section{Main results and proof outline}

\subsection{Main results}
Our main result is to characterise two classes of weight sequences $\bw$: for the first class bootstrap percolation on the Chung-Lu random graph $CL(\bw)$ exhibits a critical phenomenon, while for the second it does not. 
Roughly speaking, the first class stochastically dominates the size-biased distribution associated to a power law of exponent $3$, with a suitable large constant. In contrast, the second class is stochastically dominated by the size-biased distribution associated to a power law of exponent $3$, with a suitable small constant. In this sense, the characterisation only has a constant ``gap''.

Interestingly, there can be different types of nuclei, each leading to an outbreak once infected, depending on some property of the weight sequence $\bw$, each providing its own \emph{candidate threshold}, which if exceeded (by an $\omega(1)$-factor) guarantees an outbreak whp. This behaviour depends sensitively on the following weight bound $\weightBoundHeavy=\weightBoundHeavy(n)$ defined (point-wise) by 
\begin{equation}\label{eq:defHeavy}
\weightBoundHeavy(n):=\min\left\{x\in\mathbb{R}^+\colon |\VertexSetRestricted{\ge x}|\geq \left(\frac{W}{4x^2}\right)^r\right\},
\end{equation}
where we use the convention that $\weightBoundHeavy(n):=w_n+1$ if this set is empty. Vertices whose weight is at least $\weightBoundHeavy$ are called \emph{heavy}. The subgraph spanned by the heavy vertices will be called the \emph{dense} subgraph, while the subgraph spanned by all non-heavy vertices will be called the \emph{sparse} subgraph.

If we consider the bootstrap process restricted to the sparse subgraph, we obtain the first candidate threshold 
\begin{equation}\label{eq:candidateThresholdSparse}
\candidateThresholdSparse=\candidateThresholdSparse(n):=\left(\frac{W}{\sum_{u\in\VertexSetRestricted{<\weightBoundHeavy}}w_u^{r+1}}\right)^{1/(r-1)},
\end{equation}
which always exists. On the other hand, studying the process restricted to the dense subgraph 
yields the second candidate threshold   
\begin{equation}\label{eq:candidateThresholdDense}
\candidateThresholdDense=\candidateThresholdDense(n):=\left(\frac{1}{\sum_{u\in\VertexSetRestricted{\ge \weightBoundHeavy}}w_u^r}\right)^{1/r},
\end{equation}
provided this quantity is finite. This happens if and only if $\weightBoundHeavy\le w_n$. Should both candidate thresholds exist, the threshold will always be the smallest of the two. 

In fact in Example \ref{ex:example} we show that which of the candidates is smaller actually depends on the weight sequence. 
Even though it is a priori not obvious why, considering these two candidate thresholds turns out to be sufficient, due to a matching lower bound on the threshold $\threshold$.

\begin{theorem}\label{threshold}
Let $r\ge 2$ be an integer, let $\alpha>0$, $C\ge64r(\min\{\alpha,1/2\})^{-3}$, and $C_1>0$. Furthermore let $\bw=\bw(n)$ be a sequence of ordered weight sequences satisfying $1\le w_1\le\dots\le w_n\to\infty$, $w_{n-r+1}\geq \alpha w_n$, and 
\begin{equation}\label{eq:supercriticalCondition}
\PP[\wbrv\ge x]\ge \frac{C}{x},
\end{equation}
for all $C_1\le x\le w_n$. Consider the Chung-Lu random graph $G\sim CL(\bw)$. Then the bootstrap process 
on $G$ with infection threshold $r$ and initial infection rate $\infectionRate = \infectionRate (n)$ 
has a critical function $\threshold$.

Furthermore, if $\weightBoundHeavy\le w_n$, then $\min\{\candidateThresholdSparse, \candidateThresholdDense\}=o(1)$ and
$
\threshold=\Theta(\min\{\candidateThresholdSparse, \candidateThresholdDense\});
$
otherwise $\candidateThresholdSparse=o(1)$ and
$
\threshold=\Theta(\candidateThresholdSparse).
$
\end{theorem}

Condition \eqref{eq:supercriticalCondition} on $\wbrv$ essentially states that the distribution of the weights in the weight sequence stochastically 
dominates a distribution that has power law tail with exponent equal to 3. Indeed, recall that a distribution function $F(x)$ 
has a power law tail with exponent equal to $\tau >0$ if there exist constants $\gamma>0$ and $x_0>0$ such that 
$1-F(x)\geq \frac{\gamma}{x^{\tau -1}}$, for any $x >x_0$. 
Assume that $\tau >2$ (otherwise the distribution 
has infinite expected value). If $F^*$ denotes the distribution function of the size-biased version of a random variable whose 
distribution is $F$, then for any $x > x_0$
$$ 1- F^*(x) \geq \gamma \int_x^{\infty} \frac{z}{z^{\tau -1}}dz = \Omega \left( \frac{1}{x^{\tau-2}} \right). $$
 Hence, our claim is verified if $\tau = 3$. 

However, such a critical function does not always exist. In particular, if the size-biased distribution associated to the weight sequence is dominated by 
a power law with exponent $3$ with a suitably small constant, then there is no critical phenomenon. 
\begin{theorem}\label{nothreshold}
Let $r\ge 2$ be an integer and let $\bw=\bw(n)$ be a sequence of ordered weight sequences satisfying $1\le w_1\le\dots\le w_n\to\infty$. Consider a bootstrap process on the random graph $CL(\bw)$ with infection threshold $r$ and initial infection rate 
$\infectionRate=\infectionRate(n)$. 
If there exist constants $0<c<1/30$ and
$c_1>0$ and a function $h=h(n)\to\infty$ such that
the size-biased distribution associated with $\bw$ satisfies 
$$
\PP[\wbrv\geq f]\leq \frac{c}{f}
$$ 
for every
$c_1\leq f \leq h$, then no critical function exists.
\end{theorem}

\subsection{Proof outline}\label{sec:outline}
Theorem~\ref{threshold} states that the following holds whp: if $\infectionRate\gg \threshold$, then a constant fraction of all vertices become infected by the end of the process, and if $\infectionRate\ll \threshold$, then only few additional vertices become infected. 

In the supercritical regime, i.e., $\infectionRate\gg \threshold$, there are two phases. 
In the first phase, we show that if $\infectionRate \gg \threshold$, then there exists a weight-bound $\weightBoundKernel=\weightBoundKernel(n)$ such that the subset of vertices $\VertexSetRestricted{\geq \weightBoundKernel}$ has the property that whp ``almost all''\footnote{with respect to the total weight} of its vertices become infected. The set $\VertexSetRestricted{\ge \weightBoundKernel}$ is called a \emph{nucleus} 
of the process. 
For the construction of a nucleus $\VertexSetRestricted{\geq \weightBoundKernel}$ we first observe that the behaviour of the infection process on the subgraph spanned by the non-heavy vertices (i.e.,\ the vertices with weight less than $\weightBoundHeavy$, defined in \eqref{eq:defHeavy}) is quite different from that on the subgraph spanned by the heavy vertices. This is due to the fact that the vertices of weight less than $\weightBoundHeavy$ span a ``sparse'' subgraph, while the vertices of weight at least $\weightBoundHeavy$ span a ``dense'' subgraph. 
We show that either of these restricted processes creates a nucleus on its own under some suitable condition on $\infectionRate$. Therefore, the actual threshold cannot be (substantially) larger than the minimum of these two candidate thresholds. However, this is already sufficient because we will prove a matching upper bound for the threshold in Section~\ref{subcritical}.

In Section~\ref{sparse} we analyse the \emph{sparse process}. We show that if $\infectionRate\gg \candidateThresholdSparse$, then the sum of the weights of the infected vertices with weight less than $\weightBoundHeavy$ increases until whp ``almost all'' vertices of weight at least $\weightBoundKernel$ become infected.

The \emph{dense process} provides a candidate threshold 
if $\weightBoundHeavy\leq w_n$, i.e.,\ $\candidateThresholdDense$ exists, and $\infectionRate\gg \candidateThresholdDense$, then we use one of the following two approaches. In the first case we use the important observation that the subgraph spanned by $\VertexSetRestricted{\ge\weightBoundHeavy}$ stochastically dominates the binomial random graph $G(|\VertexSetRestricted{\ge\weightBoundHeavy}|,\weightBoundHeavy^2/W)$. 
For binomial random graphs the evolution of bootstrap percolation is well-understood (see~\cite{ar:JLTV10}). In particular, 
it follows (cf. Theorem~\ref{jltv} below) that if we can find $\omega(1)$ infected vertices with weight at least $\weightBoundHeavy$, then whp every vertex of weight $\weightBoundHeavy$ becomes infected. However in the second case, when the number of vertices of weight at least $\weightBoundHeavy$ is small, then this result is not applicable. The condition that $w_{n-r+1} \geq \alpha w_n$ essentially states that the $r$ highest weights are the same up to a multiplicative constant. This is a technical condition which ensures that there are $r$ vertices of approximately the same weight which will become infected whp. 
Since every vertex with weight at least $W/w_{n-r+1}$ is connected to each of these vertices with probability one, they also become infected. The proof appears in Section \ref{dense}.

Thereafter, in the second phase, we show that the condition on the size-biased distribution ensures that whp a set of linear size becomes infected. This is achieved by partitioning the vertices according to their weights and showing that the infection spreads from one part to the next, i.e., the one containing vertices of slightly smaller weight. In particular, we show that most vertices of a given part have at least $r$ neighbours in the previous part. 
The structure we discover there is very reminiscent of the construction of a giant $r$-core, that is, a subgraph of minimum degree at least $r$ that has linear order. This concept has been studied extensively in the random graph
literature -- see for example~\cite{ar:PitSpenWorm96}. The details can be found in Section~\ref{outbreak}.

For the subcritical regime, i.e.,\ when $\infectionRate \ll \threshold$, we show in Section~\ref{subcritical} that the number of infected vertices can be approximated by the total progeny of a subcritical branching process. 
Considering the current generation of newly infected vertices, we expose sequentially their uninfected neighbours. 
If an uninfected vertex 
is adjacent to a newly infected vertex and in addition, it has $r-1$ neighbours within the infected set, then we 
declare this vertex to be an offspring of the newly infected vertex. 
The event that a given vertex becomes infected in a certain step, conditional on the history of the process is the intersection of non-decreasing and 
non-increasing events. At this point, we make use of the FKG inequality (cf. Theorem~\ref{FKG} below) 
to deduce that these events are negatively correlated, whereby we can obtain a simple upper bound on the probability of infection at a certain step. We show that if the initial density of infected vertices is asymptotically below the threshold function, the process has expected progeny per vertex that is less than 1 and thus the infection spreads only to a few additional vertices (Section ~\ref{sec:nooutbreak}). 

The proof of Theorem~\ref{nothreshold} follows a similar argument.  In this case, its assumption on the 
distribution of the weights implies that the process we just described is again subcritical. 
With little more work, we show that this implies that the bootstrap process terminates after a small number of steps and ends with a set of infected vertices that is sublinear with high probability. 
The details can be found in  Section~\ref{sec:nothreshold}. 

\section{Fundamental properties \& tools}

We first perform some fundamental calculations, which we will be using throughout the paper. We continue 
with a collection of concentration inequalities (from the literature) that we will make use of at some point in our arguments.

\subsection{Fundamental properties of the weight sequence}\label{sec:fundamentals}
We start out by observing that the weight-bound $\weightBoundHeavy$ defined in \eqref{eq:defHeavy}, which characterises heavy vertices 
tends to infinity.
\begin{claim}\label{claim:weightBoundHeavyProperties}
Let $\weightBoundHeavy$ be defined as in \eqref{eq:defHeavy}. Then we have $\weightBoundHeavy\to\infty$ and $\weightBoundHeavy +1\le \frac{2}{3}\sqrt{W}$ for any sufficiently large $n$.
\end{claim}
\begin{proof}
Assume that $\weightBoundHeavy\le w_n$, then we have
$$
\weightBoundHeavy^{2r}\ge \frac{W^r}{4^r|\VertexSetRestricted{\ge\weightBoundHeavy}|}\stackrel{\eqref{Wcondition}}{\ge} \frac{\weightConstant^r n^{r-1}}{4^r}\to\infty, $$
since there are only $n$ vertices in total; otherwise we have $\weightBoundHeavy>w_n\to\infty$.

Now for the second statement, note that if $w_n\le \frac{2}{3}\sqrt{W}-2$, then we obtain $\weightBoundHeavy+1\le w_n+2\le \frac{2}{3}\sqrt{W}$; on the other hand, if $w_n>\frac{2}{3}\sqrt{W}-2$ then $|\VertexSetRestricted{\ge \frac{2}{3}\sqrt{W}-2}|\ge 1$, and thus 
$$
\left(\frac{W}{4\left(\frac{2}{3}\sqrt{W}-2\right)^2}\right)^r\le (1+o(1))\left(\frac{9}{16}\right)^r\le 1,
$$
 where the last inequality holds for any sufficiently large $n$. Consequently, by the definition of $\weightBoundHeavy$ we have $\weightBoundHeavy+1\le \frac{2}{3}\sqrt{W}-1<\frac{2}{3}\sqrt{W}$, and thus the claim follows.
\end{proof}
\begin{remark}\label{rem:dropMinimum}
Claim~\ref{claim:weightBoundHeavyProperties} implies that for any non-heavy vertices $u,v\in\VertexSetRestricted{<\weightBoundHeavy}$ we may drop the minimum in~\eqref{eq:pijCL}, i.e., they form an edge with probability $p_{u,v}=w_uw_v/W$.
\end{remark}

Next we relate sums of powers of vertex weights to the size-biased distribution. Recall the following standard formula for the moments of  a non-negative random variable.
\begin{lemma}[e.g.~\cite{MR1155402}]\label{thm:moments}
Let $k\ge 1$ be an integer and $X$ be a non-negative random variable. Then 
$$
\EE\left[X^k\right] = k\int_0^\infty x^{k-1}\PP[X\geq x]dx.
$$
\end{lemma}
\begin{proposition}\label{prop:restrmom}
Let $0< a< b$ and let $\vartheta\ge 2$ be an integer. Then
\begin{align} \label{eq:restrmom}
W^{-1}\sum_{v\in \VertexSetRestricted{[a,b)}} w_v^\vartheta= \PP [a \leq \wbrv < b]a^{\vartheta-1} + (\vartheta-1)\int_a^b x^{\vartheta-2} \PP[x\leq \wbrv < b] dx .
\end{align}
In particular, for $b=\weightBoundHeavy$ defined in \eqref{eq:defHeavy} and $r\le \vartheta\le 2r-1$, we obtain the upper bound
$$
\sum_{v\in \VertexSetRestricted{[a,\weightBoundHeavy)}} w_v^{\vartheta}\le 4r W^{r}a^{\vartheta-2r}.
$$
\end{proposition}
\begin{proof}
We write 
$$ 
\sum_{v\in\VertexSetRestricted{[a,b)}} w_v^\vartheta = W \sum_{v\in\VertexSetRestricted{[a,b)}} w_v^{\vartheta-1} \frac{w_v}{W} = W \EE\left[\wbrv^{\vartheta-1} \Ind{a\le \wbrv< b}\right]. 
$$
Hence, using Lemma~\ref{thm:moments} with $k=\vartheta-1$ for the random variable $\wbrv \Ind{a\le \wbrv< b}$ the first claim~\eqref{eq:restrmom} follows.

For the upper bound we will apply \eqref{eq:restrmom} with $b=\weightBoundHeavy$. As a first step we prove the following bound on the size-biased distribution for any $0<y<\weightBoundHeavy$
\begin{equation}\label{eq:weightBiasedProbabilitiesUpperBound}
\PP[y\le \wbrv<\weightBoundHeavy]< 2 W^{r-1}/y^{2r-1}.
\end{equation}
To prove this we observe that for any $0<y_1<y_2$ we have
\begin{equation}\label{eq:weightBiasedProbabilitiesUpperBoundSimple}
\PP[y_1\le\wbrv\le y_2]= W^{-1}\sum_{u\in\VertexSetRestricted{[y_1,y_2]}}w_u\le W^{-1}\left|\VertexSetRestricted{\ge y_1}\right| y_2.
\end{equation}
First note that if $\weightBoundHeavy/2 \leq y < \weightBoundHeavy$,  then we have something stronger than 
\eqref{eq:weightBiasedProbabilitiesUpperBound}:
\begin{equation}\label{eq:upperBoundSBD}
\PP[y\le \wbrv<\weightBoundHeavy]\stackrel{\eqref{eq:weightBiasedProbabilitiesUpperBoundSimple}}{\le} \frac{|\VertexSetRestricted{\ge y}|\weightBoundHeavy }{W}\stackrel{y <\weightBoundHeavy}{<} \frac{ W^{r-1}\weightBoundHeavy}{2 y^{2r}}
\stackrel{y\ge \weightBoundHeavy/2}{\le} \frac{W^{r-1}}{y^{2r-1}}.
\end{equation}
Now let $\brokenUpperSBD=\{y<\weightBoundHeavy : \PP[y\le \wbrv<\weightBoundHeavy]\geq 2 W^{r-1}/y^{2r-1}\}$ and assume for contradiction that $\brokenUpperSBD$ is not empty. Note that \eqref{eq:upperBoundSBD} implies that for any element $y'\in\brokenUpperSBD$ we have $y'< \weightBoundHeavy/2$. Therefore there exists a $y'\in \brokenUpperSBD$ with $y'<\weightBoundHeavy/2$ such that $2y'\not\in\brokenUpperSBD$ and $2y'<\weightBoundHeavy$.

Since $2y'\not\in\brokenUpperSBD$ and $2y'<\weightBoundHeavy$, we have
$$
\PP[2y'\le \wbrv < \weightBoundHeavy]\le\frac{2W^{r-1}}{(2y')^{2r-1}}.
$$
Furthermore, 
$$
\PP[y'\le \wbrv < 2y']\stackrel{\eqref{eq:weightBiasedProbabilitiesUpperBoundSimple}}{<}\frac{2y'\left|\VertexSetRestricted{\ge y'}\right|}{W}\stackrel{y'<\weightBoundHeavy}{<}\frac{ W^{r-1}}{(y')^{2r-1}}.
$$
Consequently we obtain
\begin{align*}
\PP[y'\le\wbrv<\weightBoundHeavy]<\frac{W^{r-1}}{(y')^{2r-1}}+\frac{2W^{r-1}}{(2y')^{2r-1}}< \frac{2 W^{r-1}}{(y')^{2r-1}},
\end{align*}
resulting in a contradiction, and hence~\eqref{eq:weightBiasedProbabilitiesUpperBound} holds for all $0<y<\weightBoundHeavy$.

Next note that
$$
\int_{a}^{\weightBoundHeavy}\frac{2W^{r-1}}{x^{2r-1}}x^{\vartheta-2}dx\stackrel{\vartheta<2r}{\le} \frac{2W^{r-1}a^{\vartheta-2r}}{2r-\vartheta},
$$
and thus we obtain the upper bound
\begin{align*}
\sum_{v\in \VertexSetRestricted{[a,\weightBoundHeavy)}} w_v^{\vartheta}&\stackrel{\eqref{eq:restrmom},\eqref{eq:weightBiasedProbabilitiesUpperBound}}{\le} W\frac{2W^{r-1}}{a^{2r-1}}a^{\vartheta-1} + W(\vartheta-1)\int_a^{\weightBoundHeavy} x^{\vartheta-2} \frac{2W^{r-1}}{x^{2r-1}} dx\\
&\le  2W^{r}a^{\vartheta-2r}(1+(\vartheta-1)/(2r-\vartheta))\stackrel{\vartheta<2r}{\le} 4r W^{r}a^{\vartheta-2r}.\qedhere
\end{align*}
\end{proof}

Note that a priori it is not clear by the definition of the candidate thresholds $\candidateThresholdSparse$ and $\candidateThresholdDense$ defined in ~\eqref{eq:candidateThresholdSparse} and~\eqref{eq:candidateThresholdDense}, respectively, that the claimed critical function $\threshold$ in Theorem~\ref{threshold} satisfies $\threshold=o(1)$. However, this is an almost direct consequence of Proposition~\ref{prop:restrmom}. 
\begin{corollary}\label{cor:candidatethresholdsSmall}
Under the assumptions of Theorem~\ref{threshold} if in addition $\weightBoundHeavy>w_n$ then
$$
\candidateThresholdSparse=o(1).
$$
If instead of $\weightBoundHeavy>w_n$ we have $\weightBoundHeavy\le w_n$, then 
$$
\min\{\candidateThresholdSparse,\candidateThresholdDense\}=o(1).
$$
\end{corollary}
\begin{proof}
Assume that $\weightBoundHeavy>w_n$. Then we have that $\PP [a \leq \wbrv < \weightBoundHeavy]=\PP [a \leq \wbrv]$. 
Proposition~\ref{prop:restrmom} implies that under the assumptions of Theorem~\ref{threshold} we have
$$
\sum_{u\in\VertexSetRestricted{<\weightBoundHeavy}}w_u^{r+1}\ge rW\int_{C_1}^{\weightBoundHeavy}Cx^{r-2}dx=\Omega(W\weightBoundHeavy^{r-1}),
$$
and thus by the definition in~\eqref{eq:candidateThresholdSparse} and the first statement of Claim \ref{claim:weightBoundHeavyProperties} we have 
$$
\candidateThresholdSparse=O(\weightBoundHeavy^{-1})\stackrel{C.\ref{claim:weightBoundHeavyProperties}}{=}o(1).
$$
Now, if $w_n\geq \weightBoundHeavy$, 
by the definition in~\eqref{eq:candidateThresholdDense} we have
$$
\min\{\candidateThresholdSparse,\candidateThresholdDense\}\le\candidateThresholdDense=O (\weightBoundHeavy^{-1})\stackrel{C.\ref{claim:weightBoundHeavyProperties}}{=}o(1).
$$
\end{proof}

\subsection{Tools}\label{sec:tools}

We will need the following result due to Janson, \L uczak, Turova, and Vallier \cite{ar:JLTV10} for bootstrap percolation on the binomial random graph $G(n,p)$. 
The result holds for any choice (both random and deterministic) of the set $\infec{0}$ of initially infected vertices.
\begin{theorem}[Theorems 5.6 and 5.8 in~\cite{ar:JLTV10}]\label{jltv}

Consider a bootstrap percolation process with threshold $r\geq 2$ on $G(n,p)$, where $p\ge b n^{-1/r}$ for some constant $b>0$, and the number $a$ of initially infected vertices satisfies $a\to\infty$. Then whp all $n$ vertices become infected eventually. 
\end{theorem}

We also apply the FKG inequality several times in our proofs. We consider the following setting. Let $G=(V,E)$ be a random graph where every pair of distinct vertices $u,v\in V$ appears as  an edge in $E$ independently with some probability $p_{u,v} \in [0,1]$.

A graph property is called \emph{non-decreasing} if it is preserved under the addition of edges, and it is \emph{non-increasing} if it is preserved under the removal of edges.
\begin{theorem}[FKG inequality, see for example~\cite{JLR}]\label{FKG}
	Let $A$ be a non-increasing graph property and $B$ be a non-decreasing graph property. Consider the random graph $G$ and denote by $\mathcal{A}$ and $\mathcal{B}$ the events that $G$ has property $A$ or $B$, respectively. Then we have
	$$\PP[\mathcal{A} \cap \mathcal{B}]\leq \PP[\mathcal{A}]\PP[\mathcal{B}].$$
\end{theorem}

Throughout the paper, the proofs will rely on the following collection of concentration inequalities. 
Often we deal with sums of independent random variables which guarantee exponentially small bounds on the probability of non-concentration.
\begin{theorem}[Chernoff inequality, e.g.~\cite{JLR}]\label{Chernoff}
	Let $X_i$, for $1\le i\le m$, be independent Bernoulli random variables with mean $0\le p_i\le 1$, and let $X=\sum_{i=1}^{m}X_i$ denote their sum. Then for any $s>0$ we have
	$$\PP[X\leq \EE[X]-s]\leq \exp\left(-\frac{s^2}{2\EE[X]}\right)$$
	and
	$$\PP[X\geq \EE[X]+s]\leq \exp\left(-\frac{s^2}{2(\EE[X]+s/3)}\right).$$
\end{theorem}
A similar bound due to McDiarmid~\cite{MR1678578} applies also in a more general setting.
\begin{theorem}[\cite{MR1678578}]\label{mcdct}
Let $X_i$, for $1 \leq i \leq m$, be independent random variables satisfying $X_i \leq \EE[X_i] + M$ for some constant $M$, and let $X=\sum_{i=1}^{m}X_i$ denote their sum. Then, for any $s>0$ we have
$$
\PP[X\geq \EE[X]+s]\leq \exp\left(-\frac{s^2}{2(\VV[X]+Ms/3)}\right).
$$
\end{theorem}
When the order of magnitude of each individual random variable varies strongly the following bound may prove to be stronger.
\begin{theorem}[Azuma-Hoeffding inequality, e.g.~\cite{JLR}]\label{AzumaHoeffding}
Let $X_i$, for $1 \leq i \leq m$, be independent random variables satisfying $a_i\leq X_i \leq b_i$, and let $X=\sum_{i=1}^{m}X_i$ denote their sum. Then, for any $s>0$ we have
$$
\PP[X\leq \EE[X]-s]\leq \exp\left(-\frac{s^2}{2\sum_{i=1}^m (b_i-a_i)^2}\right).
$$
\end{theorem}

\section{Nuclei}\label{sec:kernel}
As we described in Section \ref{sec:outline}, the first step towards proving that an outbreak occurs whp is to show that a subset of vertices of high weight becomes almost completely 
infected. We shall call this set a nucleus (see Definition \ref{definition:kernel} below).
The total weight of this set is high enough, so that it functions as a source of the infection of a large part of the 
random graph.  
Moreover, the density within the set is high enough, so that the bootstrap process restricted to the subgraph induced by these vertices results in its almost complete infection. 
\begin{definition}\label{definition:kernel}
Let $\weightBoundKernel=\weightBoundKernel(n)$ satisfy $\weightBoundKernel\le \min\left\{w_{n-r+1},W/w_{n-r+1}\right\}$ and $\weightBoundKernel=o(\sqrt{W})$. 
We call the quantity $\weightBoundKernel$ the  \emph{weight-bound} of the nucleus if
there exists a function $\eps = \eps (n):\mathbb{N} \to [0,1)$, tending to 0 as $n\to \infty$, such that whp
\begin{equation*} 
\totalInfectedWeightRestricted{\ge\weightBoundKernel}\geq(1-\eps )\totalWeightRestricted{\ge\weightBoundKernel}.
\end{equation*}
In addition, the set $\VertexSetRestricted{\ge\weightBoundKernel}$ is called a \emph{nucleus}.

\end{definition}
The goal of this section is to prove the existence of a nucleus in the supercritical regime.
\begin{theorem}\label{thm:kernel}
Suppose that the assumptions of Theorem~\ref{threshold} hold. If additionally 
$$
\weightBoundHeavy> w_n\qquad\text{and}\qquad\infectionRate\gg\candidateThresholdSparse,
 $$
or 
$$
\weightBoundHeavy\le w_n\qquad\text{and}\qquad\infectionRate\gg\min\{\candidateThresholdSparse,\candidateThresholdDense\},
$$
then there exists a nucleus with weight-bound $\weightBoundKernel$.
\end{theorem}
The proof of this theorem is rather sophisticated and requires some preparation. Hence, we defer it to Section~\ref{sec:proofKernel}. Its proof consists of two Theorems (Theorems~\ref{thm:kernelSparse} 
and ~\ref{thm:kernelDense}) depending on certain properties of the weight sequence. 
In fact, it is these two theorems that we use towards the proof of Theorem~\ref{threshold}. However, we also
prove Theorem~\ref{thm:kernel}, as we believe it is of independent interest. 

\subsection{Modified process}
The existence of a nucleus will be shown in several stages. The first stage is a restriction of the bootstrap process to a carefully selected 
set of vertices and its analysis therein. 
 We modify the percolation process by restricting it to some \emph{breeding ground} $\breedingGround$ for some number of steps. The steps when this restriction holds is called 
 the \emph{breeding phase}. The actual definitions of the breeding ground $\breedingGround$ used in the sparse and dense process will differ: in the former it consists of vertices of intermediate weight, while in the latter it is the set of all heavy vertices.

 More formally, let $\breedingGround\subseteq \VertexSet$ be a subset of vertices, and let $\weightBoundInitialInfection=\weightBoundInitialInfection(n)$ be a lower bound on their weights, i.e.\ $\weightBoundInitialInfection\leq \minWeight{\breedingGround}$. Then we initially infect any vertex  with weight less than $\weightBoundInitialInfection$  independently with probability $\infectionRate$ (but no vertices of larger weight). They form the set $\infecCoupling{0}$ and for convenience of notation we define $\infecCoupling{-1}:=\emptyset$. Now, for any $t\ge 0$ we denote by $\infecCoupling{t}$ the set of all vertices which became infected either initially or in some step $1\le t'\le t$ in the following process: in the $t$-th step, $t\ge 1$, we infect all (uninfected) vertices in $\breedingGround$ having at least $r$ infected neighbours in $\infecCoupling{t-1}\setminus\infecCoupling{t-2}$ (but no other vertices are infected). Note that apart from the vertices in $\infecCoupling{0}$ which were infected initially, only vertices within the breeding ground $\breedingGround$ ever become infected in this restricted process. Furthermore, whether a vertex $u\in\VertexSet\setminus\infecCoupling{t-1}$ becomes infected at time $t$ depends only on the edges connecting $u$ to any of the vertices in $\infecCoupling{t-1}\setminus\infecCoupling{t-2}$ (none of which has been revealed so far). Thus all these events are independent.

Because this restriction of the process can only delay the time at which any particular vertex becomes infected (if it becomes infected at all), we have $\infecCoupling{t}\cap \VertexSetRestricted{I}\subseteq \infec{t}\cap \VertexSetRestricted{I}$, for any $t\ge 0$ and any interval $I\subseteq \mathbb{R}^{+}$. In particular, this holds at the end of the breeding phase. At this point, it will be necessary to return to the original process instead (or another subprocess). Considering these modifications is sufficient for showing that a candidate nucleus is indeed a nucleus, since this only depends on a \emph{lower} bound on the  total weight of the vertices within the candidate nucleus which become infected eventually.

We will give a few results that will allow us to control the evolution of the modified process. We will start with a concentration result on the weight of 
$\infecCoupling{0}$. Recall that we consider a breeding ground $S$ of minimum weight that is bounded from below by $\phi_0$. 
\begin{lemma}\label{ini} 
If $\weightBoundInitialInfection\rightarrow\infty$ and $n\infectionRate\rightarrow \infty$, then  whp $\totalWeight{\infecCoupling{0}}\geq n\infectionRate/2$. 
\end{lemma}
\begin{proof}
Since $\weightBoundInitialInfection\gg 1$, we have $|\VertexSetRestricted{\ge \weightBoundInitialInfection}|\leq W/\weightBoundInitialInfection=o(n)$. Therefore $(1+o(1))n$ vertices have weight less than $\weightBoundInitialInfection$ and each of these vertices is infected independently with probability $\infectionRate$. Since $n\infectionRate\gg 1$, the Chernoff bound (Theorem~\ref{Chernoff}) implies that whp at least $n\infectionRate/2$ of these vertices become infected. Since the weight of every vertex is at least $1$ the result follows. 
\end{proof}

Next we show that the probability of a vertex $u$ becoming infected in step $t+1$ is essentially determined by its weight $w_u$ and the total weight of the vertices which became infected in the previous steps, i.e., $\totalWeight{\infecCoupling{t}}$. The upper bound is almost immediate.
\begin{lemma}\label{infecprobu}
For any $t\geq 0$ and any vertex $u\in \breedingGround\setminus \infecCoupling{t}$ we have
$$\PP[u\in \infecCoupling{t+1}|\infecCoupling{t}]\leq \frac{w_u^{r}\, \totalWeight{\infecCoupling{t}}^r}{r! W^r}.$$
\end{lemma}
\begin{proof}
By a union bound we have
$$
\PP[u\in \infecCoupling{t+1}|\infecCoupling{t}]\leq \sum_{\mathcal{R}\in \binom{\infecCoupling{t}}{r}}\prod_{v\in \mathcal{R}}\frac{w_u w_v}{W}\leq \frac{w_u^r}{W^r}\sum_{\mathcal{R}\in \binom{\infecCoupling{t}}{r}}\prod_{v\in \mathcal{R}}w_v \leq \frac{w_u^r}{W^r} \frac{\totalWeight{\infecCoupling{t}}^r}{r!}.\qedhere
$$
\end{proof}
For the lower bound we use the following intuition which turns out to be justified: the total weight of infected vertices should increase substantially in each step, and it should not be dominated by a few vertices of very large weight. So assume for now that this heuristic applies. In this case we prove that if the total weight of infected vertices is already large enough, then $u$ becomes infected whp. Otherwise, the total weight of infected vertices is still sufficiently large and the upper bound in  Lemma~\ref{infecprobu} is actually tight. 
\begin{lemma}\label{infecprobl}
Let $\xi=\xi(n)$, $\xi_1 = \xi_1(n)$, $\xi_2 = \xi_2 (n)$ be three functions such that $\xi, \xi_1, \xi_2 \to \infty$. 
Let $t\ge0$ and assume that
\begin{align}
\totalWeight{\infecCoupling{t}}&\geq \xi_1 \totalWeight{\infecCoupling{t-1}}\qquad\text{and}\label{eq:increasingWeight}\\
\maxWeight{\infecCoupling{t}}&\leq \frac{\totalWeight{\infecCoupling{t}}}{\xi_2}.\label{eq:fewHeavyInfectedVertices}
\end{align}
Furthermore, let $u\in \breedingGround\setminus \infecCoupling{t}$ be such that 
\begin{align}
\totalWeight{\infecCoupling{t}}&\leq (W/w_u)\sqrt{\log{\xi}}\qquad \text{and}\label{eq:smallTotalWeight}\\
\maxWeight{\infecCoupling{t}}&\leq \frac{4}{9} \frac{W}{w_u}.\label{eq:fewHeavyInfectedVertices2}
\end{align}

Then we have 
$$
\PP[u\in \infecCoupling{t+1}|\infecCoupling{t},\infecCoupling{t-1}]\geq \frac{w_u^{r}\,  \totalWeight{\infecCoupling{t}}^r}{r! W^r}\exp 
\left(-1.9\sqrt{\log{\xi}}\right) .
$$
If $u\in \VertexSet\setminus \infecCoupling{t}$ is such that 
\begin{equation}\label{eq:largeTotalWeight}
\totalWeight{\infecCoupling{t}}\ge (W/w_u)\sqrt[4]{\log{\xi}}, 
\end{equation}
instead of \eqref{eq:smallTotalWeight}, then we have
$$
\PP[|\neighbourhood{u}\cap(\infecCoupling{t}\setminus\infecCoupling{t-1})|\ge r\;|\;\infecCoupling{t},\infecCoupling{t-1}] \ge 1-o(1/\sqrt{\log{\xi}}),
$$
and consequently, if $u\in \breedingGround$, then
$$
\PP[u\in \infecCoupling{t+1}|\infecCoupling{t},\infecCoupling{t-1}]\ge 1-o(1/\sqrt{\log{\xi}}).
$$
\end{lemma}

\begin{proof}
In order to prove the first statement consider a vertex $u$ of weight $w_u$. Next observe that for any vertex $v\in \infecCoupling{t}$ we have $w_u w_v/W<1$ by~\eqref{eq:fewHeavyInfectedVertices2}, and thus we may drop the minimum in~\eqref{eq:pijCL} 
i.e.,\ $u$ and $v$ form an edge with probability  $p_{u,v}=w_u w_v/W$, independently.  Furthermore, since $\totalWeight{\infecCoupling{t}\setminus \infecCoupling{t-1}}=(1+o(1))\totalWeight{\infecCoupling{t}}$ by~\eqref{eq:increasingWeight}, and $\totalWeight{\infecCoupling{t}}\gg \maxWeight{\infecCoupling{t}}$ by~\eqref{eq:fewHeavyInfectedVertices}, there are at least $r$ vertices in $\infecCoupling{t}\setminus \infecCoupling{t-1}$. We bound the probability that $u$ has exactly $r$ infected neighbours in  $\infecCoupling{t}\setminus\infecCoupling{t-1}$:
\begin{align}
\PP\left[u\in \infecCoupling{t+1}\cond\infecCoupling{t},\infecCoupling{t-1}\right]&\geq\PP\left[|\neighbourhood{u}\cap (\infecCoupling{t}\setminus\infecCoupling{t-1})|=r \cond \infecCoupling{t},\infecCoupling{t-1}\right].\label{eq:infecprobabilityCoupling}
\end{align}
To compute the right-hand side of \eqref{eq:infecprobabilityCoupling}, we sum over all sets $\mathcal{R}$ of $r$ distinct vertices in $\infecCoupling{t}\setminus\infecCoupling{t-1}$ the probability that the event $\{\neighbourhood{u}\cap (\infecCoupling{t}\setminus\infecCoupling{t-1})=\mathcal{R}\}$ holds, 
because these events are mutually exclusive for different $r$-tuples of vertices. 
This provides a lower bound on \eqref{eq:infecprobabilityCoupling}
\begin{align*}
\PP\left[u\in \infecCoupling{t+1}\cond\infecCoupling{t},\infecCoupling{t-1}\right]&\ge\sum_{\mathcal{R}\in \binom{\infecCoupling{t}\setminus\infecCoupling{t-1}}{r}}\prod_{v\in \mathcal{R}}\frac{w_u w_v}{W}\prod_{v'\in\infecCoupling{t}\setminus (\infecCoupling{t-1}\cup \mathcal{R})} \left(1-\frac{w_u w_{v'}}{W}\right).
\end{align*}
Furthermore, by~\eqref{eq:fewHeavyInfectedVertices2} we have $w_u w_{v'}/W \le 4/9$. Because $1-x\ge \exp(-x/(1-x))$, for any $x<1$, the innermost product is therefore bounded by
\begin{align*}
\prod_{v'\in\infecCoupling{t}\setminus (\infecCoupling{t-1}\cup \mathcal{R})} \left(1-\frac{w_u w_{v'}}{W}\right)&\geq \prod_{v'\in\infecCoupling{t}\setminus (\infecCoupling{t-1}\cup \mathcal{R})} \exp \left(-\frac{w_u w_{v'}}{W}\left/\left(1-\frac{w_u \maxWeight{\infecCoupling{t}}}{W}\right)\right. \right)\\
&\stackrel{\eqref{eq:fewHeavyInfectedVertices2}}{\geq}\exp \left(-\frac{9}{5} w_u \sum_{v'\in\infecCoupling{t}}\frac{w_{v'}}{W}\right)\\
&=\exp \left(-\frac{9}{5} w_u \frac{\totalWeight{\infecCoupling{t}}}{W}\right)\\
&\stackrel{\eqref{eq:smallTotalWeight}}{\geq} \exp \left(-\frac{9}{5}\sqrt{\log{\xi}} \right).
\end{align*}
Note that this lower bound holds uniformly over all choices of $\mathcal{R}$. Hence it suffices to bound 
\begin{align*}
\sum_{\mathcal{R}\in \binom{\infecCoupling{t}\setminus\infecCoupling{t-1}}{r}}\prod_{v\in \mathcal{R}}\frac{w_u w_v}{W}&= \frac{w_u^r}{W^r}\sum_{\mathcal{R}\in \binom{\infecCoupling{t}\setminus\infecCoupling{t-1}}{r}}\prod_{v\in \mathcal{R}}w_v \\
&\geq \frac{w_u^r}{r!W^r } \left(\totalWeight{\infecCoupling{t}\setminus\infecCoupling{t-1}}^r- \totalWeight{\infecCoupling{t}}^{r-2}\sum_{w_v \in \infecCoupling{t}}w_v^2\right)\\
&\stackrel{\eqref{eq:increasingWeight}}{=} \frac{w_u^r}{r!W^r } \left(\totalWeight{\infecCoupling{t}}^r(1-o(1))- \totalWeight{\infecCoupling{t}}^{r-2}\sum_{w_v \in \infecCoupling{t}}w_v^2\right).
\end{align*}
The first claim follows from the fact that
$$
\totalWeight{\infecCoupling{t}}^2 \stackrel{\eqref{eq:fewHeavyInfectedVertices}}{\gg} \maxWeight{\infecCoupling{t}}\totalWeight{\infecCoupling{t}}= \maxWeight{\infecCoupling{t}}\sum_{u \in \infecCoupling{t}}w_u \geq \sum_{u \in \infecCoupling{t}}w_u^2. 
$$

For the second statement, let $u$ be a vertex of weight $w_u$ and note that if there are at least $r$ vertices in $\infecCoupling{t}\setminus \infecCoupling{t-1}$ with weight at least $W/w_u$, then all of them will be neighbours of $u$ with probability $1$. 
Otherwise, consider $\infecCoupling{t}'\subseteq \infecCoupling{t}\setminus \infecCoupling{t-1}$ where all such vertices have been removed. We have 
\begin{equation*}
\totalWeight{\infecCoupling{t}'}\geq \totalWeight{\infecCoupling{t}}-\totalWeight{\infecCoupling{t-1}}-(r-1)\maxWeight{\infecCoupling{t}}\stackrel{\eqref{eq:increasingWeight},\eqref{eq:fewHeavyInfectedVertices}}{=}(1-o(1))\totalWeight{\infecCoupling{t}}.
\end{equation*}
We have 
$$
\EE[|\neighbourhood{u}\cap \infecCoupling{t}'|]=\sum_{v\in \infecCoupling{t}'} \frac{w_u w_v}{W}=\frac{w_u \totalWeight{\infecCoupling{t}'}}{W}=(1+o(1))\frac{w_u \totalWeight{\infecCoupling{t}}}{W}\stackrel{\eqref{eq:largeTotalWeight}}{=}\Omega \left((\log{\xi})^{1/4}\right).
$$
Since $|\neighbourhood{u}\cap \infecCoupling{t}'|$ is the sum of independent indicator random variables, the Chernoff bound (Theorem~\ref{Chernoff}) implies
$$
\PP[|\neighbourhood{u}\cap \infecCoupling{t}'|<r\;|\;\infecCoupling{t},\infecCoupling{t-1}] 
\leq \exp\left(-\frac{(\EE[X]-r)^2}{2\EE[X]}\right)= o(1/\sqrt{\log{\xi}}),
$$
and the second claim follows.
\end{proof}

In the analysis of the sparse process, the breeding ground is chosen to be a specific subset $\breedingGround\subseteq \VertexSetRestricted{(\weightBoundBreedingGround,\weightBoundHeavy]}$ for some suitably chosen function $\weightBoundBreedingGround=\weightBoundBreedingGround(n) \to\infty$. We show that whp the total weight of infected vertices (in $\breedingGround$) increases significantly in every step until it is large enough so that in the following step ``almost all'' vertices in the (candidate) nucleus $\VertexSetRestricted{\ge\weightBoundKernel}$ become infected. We first show that this holds in expectation, and then using Chebyshev's inequality we deduce concentration around the expected value. However, since heavy vertices, i.e.,\ the ones with weight at least $\weightBoundHeavy$, increase the variance significantly we need to exclude them. The details can be found in Section~\ref{sparse}.

Our approach for the dense process is different. Here the breeding ground $\breedingGround$ is formed by all heavy vertices, i.e.,\ the ones with weight at least $\weightBoundHeavy$. Because these induce a dense graph only the number of infected vertices matters, while their total weight becomes irrelevant. By Theorem~\ref{jltv} we have that if at some point there are $\omega(1)$ infected heavy vertices, then every heavy vertex becomes infected eventually, hence the set of all heavy vertices forms a nucleus $\VertexSetRestricted{\ge\weightBoundKernel}$. We show that whp there are $\omega(1)$ infected heavy vertices as long as there are $\omega(1)$ heavy vertices in total. However, this is not guaranteed. If there are only $O(1)$ heavy vertices, then we can still ensure that the $r$ vertices  of largest weight, $n-r+1,\dots,n$, become infected because their weights differ by at most a constant factor. These will in turn automatically infect any vertex $u$ of weight $w_u\ge W/w_{n-r+1}$, because $u$ is connected to $n-r+1,\dots,n$ with probability $1$ by~\eqref{eq:pijCL}, providing a nucleus $\VertexSetRestricted{\ge\weightBoundKernel}$ also in this case.

\subsection{Sparse process}\label{sparse}

Ultimately we want to show that if the initial infection rate $0\le \infectionRate\le 1$ satisfies
\begin{align}\label{eq:sparseSuper}
 \infectionRate\gg \candidateThresholdSparse, 
\end{align}
where $\candidateThresholdSparse$ is the candidate threshold defined in~\eqref{eq:candidateThresholdSparse}, 
then there is an outbreak. As a first step we show that~\eqref{eq:sparseSuper} implies the existence of a nucleus $\VertexSetRestricted{\ge\weightBoundKernel}$. Since the details of the proof are very 
delicate, we split it into three parts: the first contains all relevant definitions, the second addresses some preliminary calculations which allow for a more concise 
exposition of the main argument proving Theorem~\ref{thm:kernelSparse} 
in the final part. How this nucleus triggers an 
outbreak will be shown in Section~\ref{outbreak}.

\subsubsection{Setup} First of all we assume that \eqref{eq:sparseSuper} holds, and introduce the parameter
\begin{equation}\label{def:scaledTotalWeightSparseExpectation}
\scaledTotalWeightSparseExpectation:=\infectionRate/\candidateThresholdSparse\to\infty.
\end{equation}
Furthermore, in order to define the breeding ground $\breedingGround$ we observe that roughly speaking Lemmas~\ref{infecprobu} and~\ref{infecprobl} imply that for some time the probability that a vertex $u$ of some intermediate weight $w_u$ becomes infected is proportional to $w_u^r$, and therefore the expected total weight of infected vertices of this type is essentially given by the sum of the $(r+1)$-st powers of their weights. 

With this in mind we define an auxiliary weight bound $\weightBoundBreedingGround=\weightBoundBreedingGround(n)$ as the following (point-wise) maximum
\begin{align} \label{def:weightBoundBreedingGround}
\weightBoundBreedingGround(n):=\max\left\{x\in \mathbb R^+\, :\, \sum_{\VertexSetRestricted{[x,\weightBoundHeavy)}}w_u^{r+1}\geq \frac{1}{2}\sum_{\VertexSetRestricted{<\weightBoundHeavy}}w_u^{r+1}\right\},
\end{align} 
and note that we have $\weightBoundBreedingGround<\weightBoundHeavy$ and 
$\VertexSetRestricted{[\weightBoundBreedingGround,\weightBoundHeavy)}\neq\emptyset$. 
Now we define $\breedingGround$ to be a non-empty subset of 
$\VertexSetRestricted{[\weightBoundBreedingGround,\weightBoundHeavy)}$.
In particular, assume that the vertices in $\VertexSetRestricted{[\weightBoundBreedingGround,\weightBoundHeavy)}$ 
are $v_{j_s}, v_{j_s+1},\ldots, v_{j_H}$ arranged in non-decreasing order of weight. 
Let 
\begin{align}\label{def:Sr>2}
 \breedingGround':= \{v_{j_s}, v_{j_s+2},\ldots, \} \cup \{v_{j_H} \},
 \end{align}
 that is, we include in $\breedingGround'$ \emph{every other} vertex of 
 $\VertexSetRestricted{[\weightBoundBreedingGround,\weightBoundHeavy)}$ starting from the first one 
 (the lightest), but we always include the last one.
 For $r\ge 3$, we take $\breedingGround = \breedingGround'$.


Recall that the infection of the vertices in a given step is independent and that the probability that a vertex $u$ of some intermediate weight $w_u$ becomes infected is proportional to $w_u^r$. Therefore the variance is proportional to the $(r+2)$-nd powers. Since vertices with large weight give rise to large variance we have to exclude some heavy vertices. In the $r\ge 3$ case we show that it is enough to exclude the vertices of weight at least $\weightBoundHeavy$.
On the other hand, for $r=2$, we need to be a bit more careful and choose $\breedingGround$ according to the following procedure.
We include in $\breedingGround$ each vertex in $\breedingGround'$ following the ordering as it appears in
\eqref{def:Sr>2}, so that
 $\breedingGround$ is maximal with respect to 
\begin{equation}\label{eq:breedingGroundREqualsTwo}
\sum_{u\in\breedingGround}w_u^4\le 9W^2.
\end{equation}
Note that since $\weightBoundHeavy^4\le W^2$ by Claim~\ref{claim:weightBoundHeavyProperties}, this greedily chosen breeding ground $\breedingGround$ contains at least one vertex. Furthermore, by construction we have for any $v \in\breedingGround' \setminus \breedingGround$ 
\begin{equation}\label{eq:breedingGroundREqualsTwoInverse}
w_v^4+\sum_{u\in\breedingGround}w_u^4> 9W^2.
\end{equation}

Now we use $\scaledTotalWeightSparseExpectation$ to define the weight-bound $\weightBoundInitialInfection=\weightBoundInitialInfection(n)$, which provides an upper bound on the weights of initially infected vertices in the modified process, by 
\begin{equation}\label{def:weightBoundInitialInfection}
\weightBoundInitialInfection:=\min\left\{\weightBoundBreedingGround,n\infectionRate \scaledTotalWeightSparseExpectation^{-1/2}\right\}.
\end{equation}

\subsubsection{Preliminaries}
First of all we want to guarantee that we can apply Lemma~\ref{infecprobl} for $t=0$. 
\begin{claim}\label{claim:iniSparse}
We have  $np_0 \rightarrow \infty$ as $n\to \infty$. Furthermore, 
whp $\totalWeight{\infecCoupling{0}}\ge n\infectionRate/2$ and $\maxWeight{\infecCoupling{0}}\leq 2\totalWeight{\infecCoupling{0}}\scaledTotalWeightSparseExpectation^{-1/2}=o(\totalWeight{\infecCoupling{0}})$. 
\end{claim}
\begin{proof}
We first show that $\weightBoundBreedingGround\to\infty$. For this note that for $x>0$, we have
\begin{equation}\label{eq:cruderUpperBound}
\sum_{u\in\VertexSetRestricted{<x}}w_u^{r+1}\leq x^r \sum_{u\in\VertexSetRestricted{<x}} w_u \leq x^r\sum_{u\in\VertexSet} w_u \, \leq x^r\, W.
\end{equation}
Setting $m_0:=\left(\candidateThresholdSparse^{1-r}/2\right)^{1/r}$, we thus obtain 
$$\sum_{u\in\VertexSetRestricted{<m_0}}w_u^{r+1}\stackrel{\eqref{eq:cruderUpperBound}}{\le}  m_0^{r}\, W=\frac{1}{2} \candidateThresholdSparse^{1-r} \, W \stackrel{\eqref{eq:candidateThresholdSparse}}{=} \frac{1}{2}\sum_{u\in\VertexSetRestricted{<\weightBoundHeavy}}w_u^{r+1}$$
and thus
$$\sum_{u\in\VertexSetRestricted{[m_0,\weightBoundHeavy)}}w_u^{r+1}= \sum_{u\in\VertexSetRestricted{<\weightBoundHeavy}}w_u^{r+1}-\sum_{u\in\VertexSetRestricted{<m_0}}w_u^{r+1}\geq \frac{1}{2}\sum_{u\in\VertexSetRestricted{<\weightBoundHeavy}}w_u^{r+1}.$$
Therefore, by the definition~\eqref{def:weightBoundBreedingGround} of $\weightBoundBreedingGround$, we have $\weightBoundBreedingGround\geq m_0$. 
Due to Corollary~\ref{cor:candidatethresholdsSmall} we have $\candidateThresholdSparse=o(1)$ and thus  $m_0=\Theta(\candidateThresholdSparse^{(1-r)/r})\gg 1$ .

Using this we prove that $\weightBoundInitialInfection\gg 1$. By Proposition~\ref{prop:restrmom}, for $\vartheta=r+1$ and $a=1$, 
we have 
$$
\sum_{u\in\VertexSetRestricted{<\weightBoundHeavy}}w_u^{r+1} \le 4rW^{r},
$$
and therefore, by~\eqref{def:scaledTotalWeightSparseExpectation}~and~\eqref{Wcondition} 
we also have
\begin{align*}
n\infectionRate\scaledTotalWeightSparseExpectation^{-1/2}\ge\frac{W}{\weightConstant} \scaledTotalWeightSparseExpectation^{1/2}\candidateThresholdSparse\stackrel{\eqref{eq:candidateThresholdSparse}}{=} \frac{\scaledTotalWeightSparseExpectation^{1/2}}{\weightConstant} \left(\frac{W^r}{\sum_{u\in\VertexSetRestricted{<\weightBoundHeavy}}w_u^{r+1}}\right)^{1/(r-1)}
\geq \frac{\scaledTotalWeightSparseExpectation^{1/2}}{4r\weightConstant}\to\infty, 
\end{align*}
implying $\weightBoundInitialInfection\to\infty$ and $n\infectionRate\to\infty$. 

Now, we apply Lemma~\ref{ini}, which implies that whp 
$$
\totalWeight{\infecCoupling{0}}\stackrel{L.\ref{ini}}{\ge} n\infectionRate/2.
$$
Conditional on $\totalWeight{\infecCoupling{0}}\ge n\infectionRate/2$ by \eqref{def:weightBoundInitialInfection} we have
$$\maxWeight{\infecCoupling{0}}\leq \weightBoundInitialInfection\leq n\infectionRate \scaledTotalWeightSparseExpectation^{-1/2}\leq 2\totalWeight{\infecCoupling{0}}\scaledTotalWeightSparseExpectation^{-1/2}.$$

\end{proof}

However, we need some more preparation for the inductive argument. The following lower bound of the sums of the $(r+1)$-st powers of the vertex weights in $\breedingGround$ is reminiscent of the definition of $\weightBoundBreedingGround$ in~\eqref{def:weightBoundBreedingGround}. In fact for $r\ge 3$ the bound follows from this. However, the proof for $r=2$ is non-trivial due to the slightly different construction of the breeding ground $\breedingGround$.
\begin{lemma}\label{boundonS}
For any integer $r\ge 2$ we have
$$
\sum_{u\in \breedingGround} w_u^{r+1}\geq \frac{1}{8}\sum_{u\in\VertexSetRestricted{<\weightBoundHeavy}} w_u^{r+1}.
$$
\end{lemma}
\begin{proof}
Let $r\ge 3$.  The definition of $\weightBoundBreedingGround$ in~\eqref{def:weightBoundBreedingGround} 
implies that 
$$\sum_{u\in \VertexSetRestricted{[\weightBoundBreedingGround,\weightBoundHeavy]}} w_u^{r+1}\geq \frac{1}{2}\sum_{u\in\VertexSetRestricted{<\weightBoundHeavy}} w_u^{r+1}. $$
Now, note  that the set $\breedingGround$ in~\eqref{def:Sr>2} contains every other vertex (ordered by weight) and moreover the vertex of largest weight within $\VertexSetRestricted{[\weightBoundBreedingGround,\weightBoundHeavy)}$, hence this  implies that 
\begin{equation} \label{eq:half}\sum_{u\in \breedingGround} w_u^{r+1} \geq \frac{1}{2} \sum_{u\in \VertexSetRestricted{[\weightBoundBreedingGround,\weightBoundHeavy]}} w_u^{r+1}.
 \end{equation}
The same is true if $r=2$ and $\breedingGround = \breedingGround'$.

Hence, consider $r=2$ and assume that $\breedingGround \subsetneq \breedingGround'$,
 i.e.,\ there exists a vertex $v\in \breedingGround' \setminus\breedingGround$. For any such vertex we have that $w_v^4< \weightBoundHeavy^4\leq W^2$ by Claim~\ref{claim:weightBoundHeavyProperties} and by~\eqref{eq:breedingGroundREqualsTwoInverse} we have $\sum_{v\in \breedingGround} w_v^4\geq 8 W^2$. Therefore
$$
\sum_{u\in \breedingGround} w_u^3\geq \sum_{u\in \breedingGround} w_u^4/\maxWeight{\breedingGround}\geq 8 W^2/\maxWeight{\breedingGround}.
$$
On the other hand, since we picked the elements of $\breedingGround$ in a non-decreasing order of the weights, we have for any $v\in \breedingGround' \setminus\breedingGround$ that $w_v\ge \maxWeight{\breedingGround}$, and thus 
\begin{align*}
\sum_{v\in \breedingGround' \setminus \breedingGround} w_v^3 \leq 
\sum_{v\in\VertexSetRestricted{[\weightBoundBreedingGround,\weightBoundHeavy)}}w_v^3
\stackrel{P.\ref{prop:restrmom}, \vartheta=3}{\leq} 8W^2/\maxWeight{\breedingGround}.
\end{align*}
But this implies by the pigeon-hole principle, that 
$$
\sum_{u\in \breedingGround} w_u^3\geq \frac{1}{2} \sum_{u\in \breedingGround'} w_u^3
\stackrel{\eqref{eq:half}}{\geq} 
\frac{1}{4}\sum_{u\in\VertexSetRestricted{[\weightBoundBreedingGround,\weightBoundHeavy)}} w_u^3 \ \stackrel{\eqref{def:weightBoundBreedingGround}}{\geq}\ \frac{1}{8}\sum_{ u\in\VertexSetRestricted{<\weightBoundHeavy}} w_u^3$$
as desired.
\end{proof}

The following lemma will be used later in the proof of Lemma~\ref{indstep} as an upper bound on the second moment of the total weight of the newly infected vertices 
during the modified process. The bound will be used along with Chebyshev's inequality in order to show that whp the weight of these vertices is at least 
a certain multiple of the current generation. 
\begin{lemma}\label{lem:conc}
For any integer $r\ge 2$ we have 
$$
\sum_{u\in \breedingGround} w_u^{r+2}\le 2^{r+7}W^2\candidateThresholdSparse^{2-r}.
$$
\end{lemma}
\begin{proof}
If $r=2$, then by~\eqref{eq:breedingGroundREqualsTwo} we have  $\sum_{u\in \breedingGround}w_u^4\le 9W^2<2^9W^2$ and thus the claim is immediate.

If $r\ge 3$, then recall that $\breedingGround=\breedingGround'$. 
We abbreviate
\begin{equation}\label{def:quotmom}
\quotmom:=\candidateThresholdSparse^{r-1}W^{-1}\sum_{u\in\breedingGround} w_u^{r+2}.
\end{equation}
Note that $\quotmom\stackrel{}{=}\sum_{u\in\breedingGround} w_u^{r+2}\Big/\sum_{u\in\VertexSetRestricted{<\weightBoundHeavy}}w_u^{r+1}$ by the definition of $\candidateThresholdSparse$ (cf.~\eqref{eq:candidateThresholdSparse}), and hence $\quotmom\le \weightBoundHeavy$.

Moreover, since $r+2\le 2r-1$ for any $r\ge 3$, applying Proposition~\ref{prop:restrmom} with $a=\quotmom/2$, $\vartheta=r+1$ and $\vartheta=r+2$ respectively implies that
\begin{equation}\label{eq:sum_r+1}
\sum_{u\in\VertexSetRestricted{[\quotmom/2,\weightBoundHeavy)}} w_u^{r+1}\le 4rW^r\quotmom^{1-r} \le 2^{r+2}W^r\quotmom^{1-r}
\end{equation}
and 
\begin{equation}\label{eq:sum_r+2}
\sum_{u\in\VertexSetRestricted{[\quotmom/2,\weightBoundHeavy)}} w_u^{r+2}\le 4r W^r\quotmom^{2-r}\le 2^{r+2}W^r\quotmom^{2-r}.
\end{equation}

Now if
$$
\sum_{u\in\VertexSetRestricted{<\quotmom/2}} w_u^{r+1}\leq 2^{r+6} W^{r}\quotmom^{1-r},
$$
 then~\eqref{eq:sum_r+1} implies
 $$
 \sum_{u\in\VertexSetRestricted{<\weightBoundHeavy}} w_u^{r+1}=\sum_{u\in\VertexSetRestricted{<\quotmom/2}} w_u^{r+1}+\sum_{u\in\VertexSetRestricted{[\quotmom/2,\weightBoundHeavy)}} w_u^{r+1}\leq 2^{r+7}W^{r}\quotmom^{1-r}.
 $$
Moreover, expressing the left-hand side using $\candidateThresholdSparse$ (cf.~\eqref{eq:candidateThresholdSparse}) and the right-hand side using~\eqref{def:quotmom} we obtain
$$
\candidateThresholdSparse^{1-r}W\le 2^{r+7}\candidateThresholdSparse^{-(r-1)^2}W^{2r-1}\left(\sum_{u\in\breedingGround}w_u^{r+2}\right)^{1-r}
$$
and the claim follows by taking the $(r-1)$-st root and solving for the sum. 

Otherwise, we may assume that 
\begin{equation}\label{eq:sum_of_small_vertices_large}
\sum_{u\in\VertexSetRestricted{<\quotmom/2}} w_u^{r+1}> 2^{r+6} W^{r}\quotmom^{1-r}.
\end{equation}
However, we will show that this leads to a contradiction. First note that this implies 
\begin{equation*}
\sum_{u\in\VertexSetRestricted{[\quotmom/2,\weightBoundHeavy)}} w_u^{r+2}\stackrel{\eqref{eq:sum_r+2}}{\le} 2^{r+4}W^r\quotmom^{2-r}\stackrel{\eqref{eq:sum_of_small_vertices_large}}{<}\frac{1}{4}\quotmom\sum_{u\in\VertexSetRestricted{<\quotmom/2}}w_u^{r+1}
\end{equation*}
and we also have
$$
\sum_{u\in\VertexSetRestricted{<\weightBoundHeavy}}w_u^{r+2}\ge \sum_{u\in\breedingGround}w_u^{r+2}\stackrel{\eqref{def:quotmom}}{=}\quotmom\sum_{u\in\VertexSetRestricted{<\weightBoundHeavy}}w_u^{r+1}\ge \quotmom\sum_{u\in\VertexSetRestricted{<\quotmom/2}}w_u^{r+1}.
$$
Therefore, we obtain a lower bound on the difference
$$
\sum_{u\in\VertexSetRestricted{<\weightBoundHeavy}}w_u^{r+2}-\sum_{u\in\VertexSetRestricted{[\quotmom/2,\weightBoundHeavy)}}w_u^{r+2}\ge \frac{3}{4}\quotmom\sum_{u\in\VertexSetRestricted{<\quotmom/2}}w_u^{r+1}.
$$
However, at the same time this difference satisfies
$$
\sum_{u\in\VertexSetRestricted{<\weightBoundHeavy}}w_u^{r+2}-\sum_{u\in\VertexSetRestricted{[\quotmom/2,\weightBoundHeavy)}}w_u^{r+2}\stackrel{\quotmom\le \weightBoundHeavy}{=}\sum_{u\in\VertexSetRestricted{<\quotmom/2}}w_u^{r+2}\le\frac{1}{2}\quotmom\sum_{u\in\VertexSetRestricted{<\quotmom/2}}w_u^{r+1},
$$
a clear contradiction, since $\quotmom\sum_{u\in\VertexSetRestricted{<\quotmom/2}}w_u^{r+1}> 0$ by~\eqref{def:quotmom} and~\eqref{eq:sum_of_small_vertices_large}, and since $\breedingGround\neq\emptyset$.
\end{proof}

\subsubsection{Main argument}
Now we show that until almost every vertex with weight at least $\weightBoundHeavy/(\log{\scaledTotalWeightSparseExpectation})^{1/4}$ becomes infected whp in every step of the process, the weight of the infected set increases significantly. 
For simplicity of notation we define the \emph{scaled total weight} $\scaledTotalWeightSparse{t}$ 
by
\begin{equation}\label{def:scaledTotalWeight}
\scaledTotalWeightSparse{t}=\scaledTotalWeightSparse{t}(n):=\frac{2\totalWeight{\infecCoupling{t}}}{\candidateThresholdSparse n}.
\end{equation}
Note that by Claim~\ref{claim:iniSparse} and~\eqref{def:scaledTotalWeightSparseExpectation} we have whp $\scaledTotalWeightSparse{0}=2\totalWeight{\infecCoupling{0}}/(\candidateThresholdSparse n)\ge \scaledTotalWeightSparseExpectation$ (this is why there is a factor $2$ in the definition), and therefore these scaled total weights satisfy 
\begin{equation}\label{eq:scaledTotalWeightProperties}
1\ll\infectionRate/\candidateThresholdSparse=\scaledTotalWeightSparseExpectation \le\scaledTotalWeightSparse{0}\le\scaledTotalWeightSparse{1}\le\dots\, .
\end{equation}

\begin{lemma}\label{indstep}
Let $t\ge0$, $r\ge 2$ be integers and assume that
\begin{align}
\totalWeight{\infecCoupling{t}}&\ge \frac{1}{2} \scaledTotalWeightSparseExpectation^{r-1}\exp \left(-2\sqrt{\log 
\scaledTotalWeightSparseExpectation} \right) 
\totalWeight{\infecCoupling{t-1}},\label{eq:increasingWeightRepeated}\\
\maxWeight{\infecCoupling{t}}&\le \frac{1}{2\log \scaledTotalWeightSparseExpectation} 
\totalWeight{\infecCoupling{t}}, \label{eq:fewHeavyInfectedVerticesRepeated} 
\end{align}
and
\begin{align}
\totalWeight{\infecCoupling{t}}&\leq (W/
\weightBoundHeavy)\sqrt{\log \scaledTotalWeightSparseExpectation}.
\label{eq:smallTotalWeightBreedingGround} 
\end{align}
Then conditional on the random sets $\infecCoupling{t}$ and $\infecCoupling{t-1}$ with probability at least $1-\scaledTotalWeightSparse{t}^{1-r}$ we have 
\begin{equation}\label{eq:goodEvent1}
\totalWeight{\infecCoupling{t+1}}\geq \frac{1}{2} \scaledTotalWeightSparse{t}^{r-1}\totalWeight{\infecCoupling{t}}\exp 
\left(-2\sqrt{\log \scaledTotalWeightSparseExpectation} \right)
\end{equation}
and
\begin{equation}\label{eq:goodEvent2}
\maxWeight{\infecCoupling{t+1}}\leq  \scaledTotalWeightSparse{t}^{r-1}\totalWeight{\infecCoupling{t}}\exp \left(-2\sqrt{\log{\scaledTotalWeightSparseExpectation}} \right) \big/\log{\scaledTotalWeightSparseExpectation}.
\end{equation}
\end{lemma}
\begin{proof}
We first calculate the total weight of vertices that become infected in step $t+1$. Afterwards, we use a second moment argument to show that it is actually concentrated around its expectation, thereby proving the first statement. Then we use a first moment argument to guarantee that no vertices of too large weight became infected whp.

Let $t\ge 0$ be an integer and assume that $\infecCoupling{t}$ and $\infecCoupling{t-1}$ have been realised. Furthermore, recall that we consider the process in which only vertices within the breeding ground $\breedingGround$ become infected (apart from those being infected initially). Moreover, observe that Conditions~\eqref{eq:increasingWeightRepeated},~\eqref{eq:fewHeavyInfectedVerticesRepeated}, and~\eqref{eq:smallTotalWeightBreedingGround} 
imply that their counterparts,~\eqref{eq:increasingWeight},~\eqref{eq:fewHeavyInfectedVertices}, and~\eqref{eq:smallTotalWeight}, respectively, are satisfied.
Also, Condition~\eqref{eq:fewHeavyInfectedVertices2} holds since any two vertices $u, u' \in S$ have weight at most $\weightBoundHeavy$ and by Claim~\ref{claim:weightBoundHeavyProperties}, the product of their weights satisfies $w_u w_{u'} \leq \frac{4}{9}W$. 
Thereby, Lemma~\ref{infecprobl} is applicable for all vertices $u\in\breedingGround$, i.e.\ we have
$$
\PP[u\in \infecCoupling{t+1}|\infecCoupling{t},\infecCoupling{t-1}]\geq \exp \left(- 1.9 \sqrt{\log{\scaledTotalWeightSparseExpectation}} \right) \frac{w_u^{r} \, \totalWeight{\infecCoupling{t}}^r}{r! W^r}.
$$
Consequently
\begin{align*}
\EE\left[\totalWeight{\infecCoupling{t+1}}\cond\infecCoupling{t},\infecCoupling{t-1}\right]&\geq \exp \left(- 1.9\sqrt{\log{\scaledTotalWeightSparseExpectation}} \right)\sum_{u \in \breedingGround}\frac{w_u^{r+1} \totalWeight{\infecCoupling{t}}^r}{r! W^r}-\totalWeight{\infecCoupling{t}}, 
\end{align*}
where the last term is a (crude) upper bound on the contribution of the vertices in $\breedingGround$ which are already infected at time $t$. Now observe that 
\begin{equation}\label{eq:auxLowerBound}
\frac{\totalWeight{\infecCoupling{t}}^{r-1}}{W^r}\sum_{u\in \breedingGround}w_u^{r+1}\stackrel{L.\ref{boundonS}}{\geq} \frac{\totalWeight{\infecCoupling{t}}^{r-1}}{8 W^r}\sum_{u\in\VertexSetRestricted{<\weightBoundHeavy}}w_u^{r+1}\stackrel{\eqref{eq:candidateThresholdSparse}}{=}\frac{\totalWeight{\infecCoupling{t}}^{r-1}}{8\left(\candidateThresholdSparse W\right)^{r-1}}\stackrel{\eqref{def:scaledTotalWeight}}{=}\frac{\scaledTotalWeightSparse{t}^{r-1}}{8}~\left(\frac{1}{2\lambda}\right)^{r-1},
\end{equation}
where in particular the right-hand side grows polynomially in $\scaledTotalWeightSparse{t}$. Hence, because $\scaledTotalWeightSparseExpectation\le \scaledTotalWeightSparse{t}$ and thus also $\exp(-1.9\sqrt{\log\scaledTotalWeightSparseExpectation})\scaledTotalWeightSparse{t}^{r-1}\to\infty$, we obtain
\begin{align}
\EE\left[\totalWeight{\infecCoupling{t+1}}\cond\infecCoupling{t},\infecCoupling{t-1}\right]&\geq \exp 
\left(-1.9\sqrt{\log{\scaledTotalWeightSparseExpectation}} +O(1) \right) \frac{\totalWeight{\infecCoupling{t}}^{r}}{W^r}\sum_{u\in \breedingGround}w_u^{r+1}\nonumber \\
&\stackrel{\eqref{eq:auxLowerBound}}{\ge}\exp \left(-1.9\sqrt{\log{\scaledTotalWeightSparseExpectation}} +O(1) \right)\scaledTotalWeightSparse{t}^{r-1}\totalWeight{\infecCoupling{t}} \nonumber\\
&\ge\exp \left(-2\sqrt{\log{\scaledTotalWeightSparseExpectation}} \right)\scaledTotalWeightSparse{t}^{r-1}\totalWeight{\infecCoupling{t}}\label{eq:totalWeightExpectation}
\end{align}
for any sufficiently large $n$, because $\scaledTotalWeightSparseExpectation\to\infty$.

Next we want to apply Chebyshev's inequality, so we need to provide an upper bound on the variance of $\totalWeight{\infecCoupling{t+1}}$. Because infections 
(at time $t+1$) take place independently for all vertices in $\breedingGround\setminus\totalWeight{\infecCoupling{t}}$, we have 
\begin{align}
\VV\left[\totalWeight{\infecCoupling{t+1}}\cond \infecCoupling{t},\infecCoupling{t-1}\right]&\le \sum_{u\in \breedingGround\setminus\infecCoupling{t}}w_u^2\,\PP\left[u\in\infecCoupling{t+1}\cond\infecCoupling{t},\infecCoupling{t-1}\right]\nonumber\\
&\stackrel{L.\ref{infecprobu}}{\le}\sum_{u\in\breedingGround}w_u^{r+2}\left(\totalWeight{\infecCoupling{t}}/W\right)^r\nonumber\\
&\stackrel{L.\ref{lem:conc}}{\le}2^{r+7}\candidateThresholdSparse^{2-r}W^2\left(\totalWeight{\infecCoupling{t}}/W\right)^r\nonumber\\
&\stackrel{\eqref{def:scaledTotalWeight}, W \geq n}{\le}2^{r+7}\totalWeight{\infecCoupling{t}}^2\scaledTotalWeightSparse{t}^{r-2}.\label{eq:totalWeightVariance}
\end{align}
By Chebyshev's inequality we thus obtain
\begin{align*}
\PP[\totalWeight{\infecCoupling{t+1}}\leq \EE[\totalWeight{\infecCoupling{t+1}}]/2]&\stackrel{\eqref{eq:totalWeightExpectation},\eqref{eq:totalWeightVariance}}{\le} \frac{2^{r+9}\totalWeight{\infecCoupling{t}}^2\scaledTotalWeightSparse{t}^{r-2}}{\left(\exp 
\left(-2\sqrt{\log{\scaledTotalWeightSparseExpectation}}\right)\scaledTotalWeightSparse{t}^{r-1}\totalWeight{\infecCoupling{t}}\right)^2}\\
&\leq \scaledTotalWeightSparse{t}^{-r}\exp \left(5\sqrt{\log{\scaledTotalWeightSparseExpectation}}\right),
\end{align*}
and the first statement follows since $\scaledTotalWeightSparseExpectation\le\scaledTotalWeightSparse{t}$.

For the second part of the statement we define 
$$
\zeta:=\scaledTotalWeightSparse{t}^{r-1}\totalWeight{\infecCoupling{t}}\exp \left(-2\sqrt{\log{\scaledTotalWeightSparseExpectation}} \right)/\log{\scaledTotalWeightSparseExpectation}
$$
 and note that by Lemma~\ref{infecprobu} the expected number of vertices in $\breedingGround$ of weight at least $\zeta$ becoming infected at time $t+1$ is at most
\begin{align*}
\sum_{u\in\VertexSetRestricted{[\zeta,\weightBoundHeavy)}} w_u^r\frac{\totalWeight{\infecCoupling{t}}^r}{W^r} &\stackrel{P.\ref{prop:restrmom},\vartheta=r}{\leq} 4r\totalWeight{\infecCoupling{t}}^r \zeta^{-r}\le\left(\scaledTotalWeightSparse{t}^{1-r}\exp \left(3\sqrt{\log{\scaledTotalWeightSparseExpectation}} \right) \log{\scaledTotalWeightSparseExpectation}\right)^r.
\end{align*}
Because $\scaledTotalWeightSparseExpectation\le\scaledTotalWeightSparse{t}$ the second statement follows from Markov's inequality. 

In fact, both error terms are sufficiently small so that by a union bound, both statements hold simultaneously with probability at least 
$1-\scaledTotalWeightSparse{t}^{1-r}$.
\end{proof}
Note that~\eqref{eq:goodEvent1} and~\eqref{eq:goodEvent2} imply that Conditions~\eqref{eq:increasingWeightRepeated},~\eqref{eq:fewHeavyInfectedVerticesRepeated} hold for $t+1$. 
We apply Lemma \ref{indstep} repeatedly in order to show that the total weight of the infected vertices becomes large enough so that every vertex of large enough weight becomes infected in the next step whp.

\begin{lemma}\label{largeinfset}
For any integer $r\ge 2$, there exists an integer $T\ge 0$ such that whp 
\begin{align*}
\totalWeight{\infecCoupling{T}}&>(W/\weightBoundHeavy)\sqrt{\log \scaledTotalWeightSparseExpectation}.
\end{align*} 
\end{lemma}
\begin{proof}
We show that in every step $\tau\ge0$ there are two possibilities: either this condition is satisfied, i.e.\ we can \emph{stop} and use $T=\tau$, or else we can apply Lemma~\ref{indstep} once more; note that this ensures that $\totalWeight{\infecCoupling{\tau+1}}=\omega(\totalWeight{\infecCoupling{\tau}})$ and $\maxWeight{\infecCoupling{\tau+1}}=o(\totalWeight{\infecCoupling{\tau+1}})$ for the next step. In this spirit, we denote the event that we stopped by time $\tau\ge0$ by $\stopEvent{\tau}$ (and by $\neg\stopEvent{\tau}$ its complement), and define a family of \emph{good events} 
$$
\goodEvent{\tau}:=\left\{\forall 0\le t\le\tau \colon \eqref{eq:goodEvent1}\text{ and }\eqref{eq:goodEvent2}
\text{ hold} 
\right\}
$$ 
for $\tau\ge 0$. The core of the proof is formed by the following recursive argument, whose proof we will postpone for a moment.
\begin{claim}\label{claim:goodEventInduction}
If $\neg\stopEvent{0}$, then 
\begin{equation}\label{eq:baseCase}
\PP\left[\goodEvent{0}\cond\infecCoupling{0}\right]\ge 1-o(1)-\scaledTotalWeightSparse{0}^{1-r}.
\end{equation}
Similarly, for any $\tau\ge 1$,  
\begin{equation}\label{eq:inductionStep}
\PP\left[\goodEvent{\tau}\cond \goodEvent{\tau-1},\neg\stopEvent{\tau},\infecCoupling{\tau},\infecCoupling{\tau-1}\right]\ge 1-\scaledTotalWeightSparse{\tau}^{1-r}.
\end{equation}
\end{claim}
Now assume that Claim~\ref{claim:goodEventInduction} holds. Then we observe that by definition of the scaled total weights 
$$
\frac{\scaledTotalWeightSparse{t}}{\scaledTotalWeightSparse{t-1}}\stackrel{\eqref{def:scaledTotalWeight}}{=}\frac{\totalWeight{\infecCoupling{t}}}{\totalWeight{\infecCoupling{t-1}}}\stackrel{\goodEvent{t-1}}{\ge}\frac{1}{2}\scaledTotalWeightSparse{t-1}^{r-1}\exp(-2\sqrt{\log{\scaledTotalWeightSparseExpectation}}) \ge
\frac{1}{2}\scaledTotalWeightSparseExpectation^{r-1}\exp(-2\sqrt{\log{\scaledTotalWeightSparseExpectation}})
\to \infty
$$ 
for any $t\ge 1$, and thus by a union bound
$$
\PP\left[\goodEvent{\tau}\cond \neg\stopEvent{\tau},\infecCoupling{\tau},\dots,\infecCoupling{0}\right]\ge 1-o(1)-\sum_{t=0}^{\infty}\scaledTotalWeightSparse{t}^{1-r}=1-o(1).
$$
Since we have $\totalWeight{\infecCoupling{t}}\le W$ for any time $t\ge 0$, the recursion must end in finite time, however this can only happen because there exists a $T\ge 0$ such that $\stopEvent{T}$ holds, proving the statement of the lemma.

It remains to prove Claim~\ref{claim:goodEventInduction}. By Claim~\ref{claim:iniSparse} we initially have $\totalWeight{\infecCoupling{0}}\ge n\infectionRate/2 \to 
\infty$ and $\maxWeight{\infecCoupling{0}}\leq 2\totalWeight{\infecCoupling{0}}\scaledTotalWeightSparseExpectation^{-1/2}$.

Lemma~\ref{indstep} is applicable for $t=0$
 and thus $\goodEvent{0}$ holds with probability at least $1-o(1)-\scaledTotalWeightSparse{0}^{1-r}$. Otherwise we must have $\totalWeight{\infecCoupling{0}}> (W/\weightBoundHeavy) \sqrt{\log{\scaledTotalWeightSparseExpectation}}$, and therefore $\stopEvent{0}$ holds. Both cases together prove~\eqref{eq:baseCase}.

The induction step is proven analogously: let $\tau\ge 1$ and assume $\goodEvent{\tau-1}$ holds, then if additionally $\totalWeight{\infecCoupling{\tau}}\leq (W/\weightBoundHeavy) \sqrt{\log{\scaledTotalWeightSparseExpectation}}$, then 
 Lemma~\ref{indstep} with $t=\tau$ 
implies that with sufficiently high probability $\goodEvent{\tau}$ holds. Otherwise we have $\totalWeight{\infecCoupling{\tau}}> (W/\weightBoundHeavy) \sqrt{\log{\scaledTotalWeightSparseExpectation}}$, and therefore $\stopEvent{\tau}$ holds. Both cases together prove the induction step~\eqref{eq:inductionStep}.
\end{proof}

Next we construct a nucleus (cf. Definition~\ref{definition:kernel}) for the sparse process. We consider the \emph{candidate nucleus} $\VertexSetRestricted{\ge\weightBoundKernel}$ given by the weight-bound
\begin{equation}\label{def:kernelSparse}
\weightBoundKernel:=
\begin{cases}
W/w_{n-r+1}&\text{if }\weightBoundHeavy(\log{\scaledTotalWeightSparseExpectation})^{-1/4}\eta> W/w_{n-r+1},\\
\weightBoundHeavy(\log{\scaledTotalWeightSparseExpectation})^{-1/4}&\text{if }\weightBoundHeavy(\log{\scaledTotalWeightSparseExpectation})^{-1/4}\eta\le W/w_{n-r+1},
\end{cases}
\end{equation}
where $\eta=\eta(n)$ is an arbitrarily slowly growing function, satisfying $\eta\to\infty$ but $(\log\scaledTotalWeightSparseExpectation)^{-1/4}\eta=o(1)$.\footnote{Such an $\eta$ exists since $\scaledTotalWeightSparseExpectation\to\infty$ by~\eqref{eq:scaledTotalWeightProperties}.} We do so by determining a subset $\mathcal{U}_K\subseteq\VertexSetRestricted{\ge \weightBoundKernel}\cap \infec{F}$ of weight 
\begin{equation}\label{eq:witnessNucleus}
\totalWeight{\mathcal{U}_K}\ge (1-o(1))\totalWeightRestricted{\ge\weightBoundKernel}.
\end{equation} 
It will be crucial that this construction requires us to only expose a specific subset of edge indicator variables. 
\begin{theorem}\label{thm:kernelSparse}
Under the assumptions of Theorem~\ref{threshold} suppose that $\infectionRate\gg\candidateThresholdSparse$. Then whp the following two statements hold: 
\begin{itemize}
\item  There exists a set $\mathcal{U}_K\subseteq(\VertexSetRestricted{\ge \weightBoundKernel}\cap \infec{F})$ (where $\VertexSetRestricted{\ge\weightBoundKernel}$ is defined as in~\eqref{def:kernelSparse}) which satisfies~\eqref{eq:witnessNucleus}, contains $\{n-r+1,\dots,n\}$, and is constructed by only exposing edge-indicator random variables corresponding to edges in $\breedingGround'\times (\infecCoupling{0}\cup \breedingGround'\cup\VertexSetRestricted{\ge\weightBoundKernel})$.
\item The weight-bound $\weightBoundKernel$ satisfies $\weightBoundKernel\le \min\left\{w_{n-r+1},W/w_{n-r+1}\right\}$ and $\weightBoundKernel=o(\sqrt{W})$.
\end{itemize} 
In particular, this means that whp $\VertexSetRestricted{\ge\weightBoundKernel}$ is a nucleus. 
\end{theorem}
\begin{proof}
For this proof we abbreviate $\weightBoundHeavy':=\weightBoundHeavy(\log{\scaledTotalWeightSparseExpectation})^{-1/4}$ for convenience of notation. We first show that the two conditions on the weight-bound $\weightBoundKernel$ (defined in~\eqref{def:kernelSparse}) are satisfied. Because $w_{n-r+1}=\alpha w_n$ 
and $\weightBoundHeavy\le w_n+1$ (cf.~\eqref{eq:defHeavy}) we have 
\begin{equation}\label{eq:superHeavyVertices}
\weightBoundHeavy'\eta=\weightBoundHeavy(\log\scaledTotalWeightSparseExpectation)^{-1/4}\eta=o(\weightBoundHeavy)=o(w_{n-r+1}),
\end{equation}
since $(\log\scaledTotalWeightSparseExpectation)^{-1/4}\eta=o(1)$. Now if $\weightBoundKernel=\weightBoundHeavy'$, then we have 
 $$
\weightBoundKernel= \weightBoundHeavy'\stackrel{\eqref{def:kernelSparse}}{\le} W/(\eta w_{n-r+1})=o(W/w_{n-r+1})\quad\text{and}\quad\weightBoundKernel=\weightBoundHeavy'\stackrel{\eqref{eq:superHeavyVertices}}{=}o(w_{n-r+1}).
$$
Otherwise if $\weightBoundKernel=W/w_{n-r+1}$, then by \eqref{def:kernelSparse} we obtain
 $$
 \weightBoundKernel=W/w_{n-r+1} < \weightBoundHeavy'\eta\stackrel{\eqref{eq:superHeavyVertices}}{=}o(w_{n-r+1}).
 $$
 Thus in both cases the first property of $\weightBoundKernel$ is satisfied. Furthermore, the previous argument also showed $\weightBoundKernel=o(\weightBoundHeavy)$ and since $\weightBoundHeavy=O(\sqrt{W})$ by Claim~\ref{claim:weightBoundHeavyProperties} the second property follows. 

It remains to construct the set $\mathcal{U}_K$, or, in other words, show that a significant proportion of the candidate nucleus becomes infected eventually. For this let $T\ge0$ be an integer satisfying  $\totalWeight{\infecCoupling{T}}\ge (W/\weightBoundHeavy)\sqrt{\log \scaledTotalWeightSparseExpectation}$ as in Lemma~\ref{largeinfset} and condition on this (high probability) event. In particular we have
\begin{equation}\label{eq:conditionLargeWeight}
\totalWeight{\infecCoupling{T}}\ge (W/\weightBoundHeavy)\sqrt{\log \scaledTotalWeightSparseExpectation}\ge (W/w_u)\sqrt[4]{\log\scaledTotalWeightSparseExpectation},
\end{equation}
for any $w_u\ge \weightBoundHeavy'$.

First assume that we have $\weightBoundKernel=W/w_{n-r+1}$. 
We consider the $r$ vertices of largest weight and observe that
$$
w_n\ge\dots\ge w_{n-r+1}\stackrel{\eqref{eq:superHeavyVertices}}{\gg} \weightBoundHeavy'\eta\gg\weightBoundHeavy'
$$
and thus the assertion in~\eqref{eq:conditionLargeWeight} is true for each of them. Therefore the second statement of Lemma~\ref{infecprobl} is applicable (with $t=T$ and $\xi=\scaledTotalWeightSparseExpectation\to \infty$) showing that whp $\{n-r+1,\dots,n\}\subseteq\infec{T+1}$ (using only edges between $\{n-r+1,\dots,n\}$ and $(\infecCoupling{T}\setminus\infecCoupling{T-1})\subseteq \breedingGround'$). 
Moreover, 
we certainly have $\VertexSetRestricted{\ge\weightBoundKernel}=\VertexSetRestricted{\ge W/w_{n-r+1}}\subseteq\infec{T+2}$, because every vertex in this set is connected to each vertex in $\{n-r+1,\dots,n\}$ with probability~$1$. Thus we set $\mathcal{U}_K:=\VertexSetRestricted{\ge \weightBoundKernel}$, which obviously satisfies~\eqref{eq:witnessNucleus}. Since the process $\{\infecCoupling{t}\}$ exposes only edges in $\breedingGround'\times (\infecCoupling{0}\cup\breedingGround')$ and the last step only depends on edges in $\breedingGround'\times \VertexSetRestricted{\ge\weightBoundKernel}$ the statement follows in this case.

Otherwise, we have $\weightBoundKernel=\weightBoundHeavy'$ and we need to consider the vertices in $\VertexSetRestricted{\ge\weightBoundHeavy'}\setminus \infecCoupling{T}$. We observe that~\eqref{eq:conditionLargeWeight} shows that the weight of  any vertex $u\in \VertexSetRestricted{\ge\weightBoundHeavy'}\setminus \infecCoupling{T}$ is sufficiently large to apply the second statement of Lemma~\ref{infecprobl} (with $t=T$ and $\xi=\scaledTotalWeightSparseExpectation\to\infty$). Hence, whp each of them has at least $r$ neighbours in $\infecCoupling{T}\setminus\infecCoupling{T-1}$, and this events depend only on mutually disjoint sets of edges. So now define $\mathcal{U}_K:=\{u\in \VertexSetRestricted{\ge\weightBoundHeavy'}\;|\; u\in \infecCoupling{T}\text{ or } |\neighbourhood{u}\cap(\infecCoupling{T}\setminus\infecCoupling{T-1})|\ge r\}$, and note that $\totalWeight{\mathcal{U}_K}=(1-o(1))\totalWeightRestricted{\ge\weightBoundHeavy'}$.


Consequently, since the weight of each vertex is at most $w_n$, McDiarmid's inequality (Theorem~\ref{mcdct}) is applicable (with $X_u=\Ind{u\in \mathcal{U}_K}$, $M=w_n$, and $s=\sqrt{w_n\totalWeightRestricted{\ge\weightBoundHeavy'}}$) yielding
\begin{align}\label{eq:errrorProb}
\PP\left[\totalWeight{\mathcal{U}_K}\leq(1-o(1))\totalWeightRestricted{\ge\weightBoundHeavy'}-\sqrt{w_n\totalWeightRestricted{\ge\weightBoundHeavy'}}\, \right]&\nonumber\\
&\hspace{-5cm}\leq\exp\left(-\frac{w_n\totalWeightRestricted{\ge\weightBoundHeavy'}}{2(\totalWeightRestricted{\ge\weightBoundHeavy'}(1-o(1))+w_n\sqrt{w_n\totalWeightRestricted{\ge\weightBoundHeavy'}}/3)}\right).
\end{align}
Observe that by the assumption on the size-biased distribution in Theorem~\ref{threshold} we have
\begin{equation}\label{eq:totalWeightKernel}
\totalWeightRestricted{\ge\weightBoundHeavy'}=W\sum_{u\in\VertexSetRestricted{\ge\weightBoundHeavy'}}\frac{w_u}{W}=W\PP[\wbrv\ge\weightBoundHeavy']\ge \frac{CW}{\weightBoundHeavy'}.
\end{equation}
Furthermore, since $\weightBoundKernel=\weightBoundHeavy'$ we have $\weightBoundHeavy'\eta\le W/w_{n-r+1}$ by~\eqref{def:kernelSparse}, and this in turn shows that $W/\weightBoundHeavy'=\omega(w_n)$, because $w_{n-r+1}=\alpha w_n$ and $\eta\to\infty$. Combined with~\eqref{eq:totalWeightKernel} this provides $\totalWeightRestricted{\ge\weightBoundHeavy'}=\omega(w_n)$, and consequently the right-hand side of~\eqref{eq:errrorProb} is $o(1)$. Because $\mathcal{U}_K\subseteq(\VertexSetRestricted{\ge\weightBoundHeavy'}\cap\infec{F})$, we thus obtain whp $\totalInfectedWeightRestricted{\ge\weightBoundKernel}=(1-o(1))\totalWeightRestricted{\ge\weightBoundKernel}$ also in this case. Because the process $\{\infecCoupling{t}\}$ exposes only edges in $\breedingGround'\times (\infecCoupling{0}\cup\breedingGround')$ and the last step only depends on edges in $\breedingGround'\times \VertexSetRestricted{\ge\weightBoundKernel}$ this completes the proof.
\end{proof}

\subsection{Dense process}\label{dense}

In this section we want to show that if the candidate threshold $\candidateThresholdDense$ for the dense process---as defined in~\eqref{eq:candidateThresholdSparse}---exists, in other words,
\begin{equation}\label{eq:candidateThresholdDenseExists}
\weightBoundHeavy\le w_n,
\end{equation}
 and the initial infection rate $0\le \infectionRate=\infectionRate(n)\le 1$ satisfies
\begin{align*}
 \infectionRate\gg \candidateThresholdDense,
\end{align*}
then there exists a nucleus $\VertexSetRestricted{\ge\weightBoundKernel}$. Throughout this section we assume that these two conditions hold, and introduce the parameter
\begin{equation}\label{def:scaledTotalWeightDenseExpectation}
\scaledTotalWeightDenseExpectation:=\infectionRate/\candidateThresholdDense\to\infty.
\end{equation}
Furthermore we observe that the candidate threshold for the dense process satisfies
\begin{equation*}
\candidateThresholdDense W\stackrel{\eqref{eq:candidateThresholdDense}}{=}\left(\frac{W^r}{\sum_{u\in\VertexSetRestricted{\ge\weightBoundHeavy}}w_u^r}\right)^{1/r}\ge 1,
\end{equation*}
and thus, since $W=\weightConstant n$ for some constant $\weightConstant > 0$, we have
\begin{equation}\label{eq:candidateThresholdDenseLowerBound}
n\infectionRate=(W/\weightConstant)\scaledTotalWeightDenseExpectation\candidateThresholdDense\ge \scaledTotalWeightDenseExpectation/\weightConstant\to\infty.
\end{equation}

The setup for the dense process is much simpler than that for the sparse process. We use the set \begin{equation}\label{eq:breedingGroundDense}
\breedingGround:=\VertexSetRestricted{\ge\weightBoundHeavy}
\end{equation}
 as breeding ground and define the weight-bound for the initial infection by
\begin{equation}\label{def:weightBoundInitialInfectionDense}
\weightBoundInitialInfection:=\min\{\weightBoundHeavy,\sqrt{n\infectionRate}\}.
\end{equation}

Similarly as in Section~\ref{sparse}, we establish some basic properties of $\totalWeight{\infecCoupling{0}}$, which for instance guarantee that Lemma~\ref{infecprobl} is applicable at time $t=0$.
\begin{claim}\label{claim:iniDense}
Whp $n\infectionRate/2\le\totalWeight{\infecCoupling{0}}< (\weightConstant+1)n\infectionRate$ and 
$\maxWeight{\infecCoupling{0}}=o(\totalWeight{\infecCoupling{0}})$.

\end{claim}
\begin{proof}
In order to apply Lemma~\ref{ini} it suffices to show that $\weightBoundHeavy\to\infty$, since we already showed in~\eqref{eq:candidateThresholdDenseLowerBound} that $n\infectionRate\to\infty$. To do so, we observe that $\weightBoundHeavy=O(1)$ (over a subsequence)
would imply 
$$
\left|\VertexSetRestricted{\ge\weightBoundHeavy}\right|\stackrel{\eqref{eq:defHeavy}}{\ge} \left(\frac{W}{4\weightBoundHeavy^2}\right)^r=\Omega(W^r)=\omega(n),
$$
 yielding a contradiction since there are only $n$ vertices in total.  Hence Lemma~\ref{ini} is applicable and we have whp
$$
\totalWeight{\infecCoupling{0}}\stackrel{L.\ref{ini}}{\ge}n\infectionRate/2\gg\sqrt{n\infectionRate}\stackrel{\eqref{def:weightBoundInitialInfectionDense}}{\ge} \weightBoundInitialInfection\ge \maxWeight{\infecCoupling{0}}, 
$$
since only vertices of weight at most $\weightBoundInitialInfection$ are infected initially. On the other hand, we have 
$$
\VV[\totalWeight{\infecCoupling{0}}]=\sum_{u\in\VertexSetRestricted{<\weightBoundInitialInfection}} 
\left( w_u^2\infectionRate-(w_u\infectionRate)^2\right) \le \sum_{u\in\VertexSetRestricted{<\weightBoundInitialInfection}}
w_u^2\infectionRate \le \weightBoundInitialInfection \infectionRate W.
$$
Theorem~\ref{mcdct} implies 
$$
\PP[\totalWeight{\infecCoupling{0}}\ge (\weightConstant+1)n\infectionRate]\le\exp\left(-\frac{(n\infectionRate)^2}{3\weightBoundInitialInfection\infectionRate W}\right)\stackrel{\eqref{def:weightBoundInitialInfectionDense}}{\le}\exp(-\sqrt{n\infectionRate}/(3\weightConstant))\stackrel{\eqref{eq:candidateThresholdDenseLowerBound}}{=}o(1),
$$
i.e.,\ the claimed upper bound on $\totalWeight{\infecCoupling{0}}$ holds whp. 
\end{proof}

The crucial observation for the dense process is the following consequence of Theorem~\ref{jltv}, asserting that once a substantial number of heavy vertices become infected, all of them will be infected eventually.
\begin{lemma}\label{lem:JLTV}
Assume that there is a function $a = a(n)$ such that $a \to \infty$
for which $|\infecCoupling{1}\cap \VertexSetRestricted{\ge\weightBoundHeavy}|\geq a$. 
Conditional on this, whp there exists a time $T\ge 1$ such that $\VertexSetRestricted{\ge\weightBoundHeavy}\subseteq\infec{T}$.
\end{lemma}
\begin{proof}
Consider the random graph induced by the vertex set $\VertexSetRestricted{\ge\weightBoundHeavy}$, and note that it stochastically dominates the binomial random graph $G(n',p')$, with 
$$
n':=|\VertexSetRestricted{\ge\weightBoundHeavy}|\ge|\infecCoupling{1}\cap\VertexSetRestricted{\ge\weightBoundHeavy}|\geq a \to\infty
$$
and
$$
p':=\weightBoundHeavy^2/W\stackrel{\eqref{eq:defHeavy}}{\ge} |\VertexSetRestricted{\ge\weightBoundHeavy}|^{-1/r}/4=(n')^{-1/r}/4.
 $$
 Now we consider bootstrap process with parameter $r$ on $G(n',p')$ where the set of initially infected vertices is $\infecCoupling{1}\cap\VertexSetRestricted{\ge\weightBoundHeavy}$. So $|\infecCoupling{1}\cap\VertexSetRestricted{\ge\weightBoundHeavy}|\ge a \to\infty$. Hence Theorem~\ref{jltv} is applicable. Since all vertices which become infected by this process within $\tau\ge0$ steps are contained in $\infec{1+\tau}$, we obtain
 $$
 \PP[\exists T\ge 1\colon\VertexSetRestricted{\ge\weightBoundHeavy}\subseteq\infec{T}]\to1,
 $$
 as $n'\to\infty$ and the claim follows since $n'\to\infty$ as $n\to\infty$.
\end{proof}
Next we show that a significant number of heavy vertices become infected whp if either one of the following additional assumptions holds:
\begin{align}
 (\weightConstant+1)\infectionRate w_n&\le (\log \scaledTotalWeightDenseExpectation)^{1/2};\label{cond:smallInfectionRate}\\
w_{n-r+1}\le W^{1/2}(\log \scaledTotalWeightDenseExpectation)^{1/16}\quad&\wedge\quad \weightBoundHeavy\le W^{1/2}(\log \scaledTotalWeightDenseExpectation)^{-1/8}.\label{cond:largeInfectionRate}
\end{align}

Before we make this statement precise, we perform some calculations motivating the first assumption. It allows us to distinguish the two regimes of Lemma~\ref{infecprobl} (which is applicable at time $t=0$ by Claim~\ref{claim:iniDense}): if~\eqref{cond:smallInfectionRate} holds, then since $W=\weightConstant n$, for some $\weightConstant\ge1$, (cf.~\eqref{Wcondition}) we have
\begin{equation}\label{eq:upperBoundInitialInfectedWeight}
\totalWeight{\infecCoupling{0}}\stackrel{C.\ref{claim:iniDense}}{<}(\weightConstant+1)\infectionRate W
\stackrel{\eqref{cond:smallInfectionRate}}{\le}(W/w_n)(\log\scaledTotalWeightDenseExpectation)^{1/2},
\end{equation}
and thus we also obtain
$$
\maxWeight{\infecCoupling{0}}\stackrel{\eqref{def:weightBoundInitialInfectionDense}}{\leq} \sqrt{n\infectionRate}\stackrel{C.\ref{claim:iniDense}}{\le}\frac{2\totalWeight{\infecCoupling{0}}}{\sqrt{n\infectionRate}}\stackrel{\eqref{eq:candidateThresholdDenseLowerBound}}
{\le}\frac{2\totalWeight{\infecCoupling{0}} \weightConstant^{1/2}}{\scaledTotalWeightDenseExpectation^{1/2}}\stackrel{\eqref{eq:upperBoundInitialInfectedWeight}}{\le}
\frac{2\weightConstant^{1/2}W(\log\scaledTotalWeightDenseExpectation)^{1/2}}{w_n \scaledTotalWeightDenseExpectation^{1/2}}\stackrel{\eqref{def:scaledTotalWeightDenseExpectation}}{\le}
\frac{W}{4w_n},
$$
for any $n$ large enough. 
This together with Claim~\ref{claim:iniDense} implies that the first assertion of Lemma~\ref{infecprobl} holds (with 
$\xi=\scaledTotalWeightDenseExpectation$) for all vertices $u\in \VertexSetRestricted{\ge\weightBoundHeavy}$ (other vertices are not considered in the restricted process), i.e.,\ we have
$$
\PP[u\in \infecCoupling{1}|\infecCoupling{0}]\geq \frac{w_u^{r}\totalWeight{\infecCoupling{0}}^r}{r! W^r} 
\exp\left(-1.9\sqrt{\log{\scaledTotalWeightDenseExpectation}}\right).
$$
Because $\totalWeight{\infecCoupling{0}}/W\ge \infectionRate/(2\weightConstant)=\scaledTotalWeightDenseExpectation\candidateThresholdDense/(2\weightConstant)$, this implies 
\begin{equation}\label{eq:smallInfectionProbability}
\PP[u\in \infecCoupling{1}|\infecCoupling{0}]\stackrel{\eqref{def:scaledTotalWeightDenseExpectation}}{\gg} w_u^r\candidateThresholdDense^r\stackrel{\eqref{eq:candidateThresholdDense}}{=}\frac{w_u^r}{\sum_{v\in\VertexSetRestricted{\ge\weightBoundHeavy}}w_v^r}.
\end{equation}

On the other hand, if~\eqref{cond:smallInfectionRate} does not hold, i.e.,\ we have $(\weightConstant+1)\infectionRate w_n> (\log \scaledTotalWeightDenseExpectation)^{1/2}$, then we obtain 
$$
\totalWeight{\infecCoupling{0}}\stackrel{C.~\ref{claim:iniDense}}{\ge}n\infectionRate/2>\frac{W(\log \scaledTotalWeightDenseExpectation)^{1/2}}{2(\weightConstant+1)\weightConstant w_n}\stackrel{\eqref{def:scaledTotalWeightDenseExpectation}}{\gg}(W/w_u)(\log \scaledTotalWeightDenseExpectation)^{1/4},
$$
for any vertex $u$ whose weight satisfies $w_u\ge w_n(\log \scaledTotalWeightDenseExpectation)^{-1/4}$.  Thus the second assertion of Lemma~\ref{infecprobl} (with $\xi=\scaledTotalWeightDenseExpectation$) holds for any such vertex showing that 
\begin{equation}\label{eq:highInfectionProbability}
\PP\left[u\in\infecCoupling{1}\cond \infecCoupling{0}\right]\ge 1-o(1/\sqrt{\log \scaledTotalWeightDenseExpectation}).
\end{equation}

Using these two observations we prove the following result.
\begin{lemma}\label{lem:denseHeavyVertices}
If either~\eqref{cond:smallInfectionRate} or~\eqref{cond:largeInfectionRate} holds, then there exists a function $a = a(n)$ 
such that $a \to \infty$ for which 
whp we have $|\infecCoupling{1}\cap \VertexSetRestricted{\ge\weightBoundHeavy}|\geq a$. 
\end{lemma}
\begin{proof}
The idea of this proof is to show that whp a significant number of heavy vertices becomes infected, and therefore by Lemma~\ref{lem:JLTV} all heavy vertices become infected. The proof is split into two cases. 

{\bf Case I.} Assume that~\eqref{cond:smallInfectionRate} holds, and therefore also~\eqref{eq:smallInfectionProbability}. Summing~\eqref{eq:smallInfectionProbability} over all $u\in\VertexSetRestricted{\ge\weightBoundHeavy}$ we have that $\EE\left[|\infecCoupling{1}\cap \VertexSetRestricted{\ge\weightBoundHeavy}|\cond\infecCoupling{0}\right]\gg 1$. Consequently, as $|\infecCoupling{1}\cap\VertexSetRestricted{\ge\weightBoundHeavy}|$ is the sum of independent Bernoulli random variables, the Chernoff bound (Theorem~\ref{Chernoff}) yields that whp $|\infecCoupling{1}\cap\VertexSetRestricted{\ge\weightBoundHeavy}|\gg 1$, completing Case~I.\medskip

{\bf Case II.} Now assume that~\eqref{cond:largeInfectionRate} holds, but~\eqref{cond:smallInfectionRate} does not hold which implies that~\eqref{eq:highInfectionProbability} holds. We consider vertices of weight at least $w':=\sqrt{W}(\log\scaledTotalWeightDenseExpectation)^{-1/8}$, and observe that by~\eqref{cond:largeInfectionRate} all of them are heavy, i.e.,\ $w'\ge\weightBoundHeavy$. Moreover by the assumptions of Theorem~\ref{threshold} the number of these vertices satisfies
$$
|\VertexSetRestricted{\ge w'}|\ge \frac{W\PP[\wbrv\ge w']}{w_n}\ge\frac{\alpha CW}{w_{n-r+1} w'}\stackrel{\eqref{cond:largeInfectionRate}}{=} \Omega((\log\scaledTotalWeightDenseExpectation)^{1/16})\to\infty.
$$
Now note that any such vertex $u$ satisfies
$$
w_u\ge w'\stackrel{\eqref{cond:largeInfectionRate}}{\ge} \alpha w_n(\log \scaledTotalWeightDenseExpectation)^{-3/16}\gg w_n(\log \scaledTotalWeightDenseExpectation)^{-1/4},
$$
and thus by~\eqref{eq:highInfectionProbability} it becomes infected at time $t=1$ with probability at least 
$1-o(1/\sqrt{\log\scaledTotalWeightDenseExpectation})$. Applying a union bound to a sufficiently small but growing number of these vertices implies that
there is a function $a \to \infty$ such that whp 
$$
|\infecCoupling{1}\cap \VertexSetRestricted{\ge\weightBoundHeavy}|\ge|\infecCoupling{1}\cap \VertexSetRestricted{\ge w'}|\geq a,
$$
completing the proof of the lemma. 
\end{proof}

With this preparation we will now construct a nucleus (cf. Definition~\ref{definition:kernel}) for the dense process. We consider the \emph{candidate nucleus} $\VertexSetRestricted{\ge\weightBoundKernel}$ given by the weight-bound
\begin{equation}\label{def:kernelDense}
\weightBoundKernel:=
\begin{cases}
W/w_{n-r+1}&\text{if }w_{n-r+1}>W^{1/2}(\log \scaledTotalWeightDenseExpectation)^{1/16},\\
\weightBoundHeavy&\text{if }\eqref{cond:largeInfectionRate}\text{ holds}.
\end{cases}
\end{equation}
Note that this does not always define a weight-bound $\weightBoundKernel$, however we will demonstrate in Section~\ref{sec:proofKernel} that this suffices to guarantee the existence of a nucleus. 
\begin{theorem}\label{thm:kernelDense}
Under the assumptions of Theorem~\ref{threshold}, suppose also that $\infectionRate\gg\candidateThresholdDense$. Then whp $\VertexSetRestricted{\ge\weightBoundKernel}$ as defined in~\eqref{def:kernelDense} is a nucleus with weight-bound $\weightBoundKernel$  that satisfies $\weightBoundKernel\le \min\left\{w_{n-r+1},W/w_{n-r+1}\right\}$ and $\weightBoundKernel=o(\sqrt{W})$. Moreover, the nucleus is 
completely infected whp and no edges in $G\left[\VertexSetRestricted{\le\weightBoundInitialInfection}\right]$ have been exposed yet.  
\end{theorem}
\begin{proof}
We start by showing that the weight-bound $\weightBoundKernel$ is sufficiently small. Recall that $\scaledTotalWeightDenseExpectation\to\infty$ by~\eqref{def:scaledTotalWeightDenseExpectation}. Now if $\weightBoundKernel=W/w_{n-r+1}$, then 
$$
\weightBoundKernel=W/w_{n-r+1}<W^{1/2}(\log \scaledTotalWeightDenseExpectation)^{-1/16}=o(\sqrt{W})=o(w_{n-r+1}),
$$
and both conditions are satisfied. On the other hand, if $\weightBoundKernel=\weightBoundHeavy$, then we have
$$
\weightBoundKernel=\weightBoundHeavy\le W^{1/2}(\log \scaledTotalWeightDenseExpectation)^{-1/8}\le W^{1/2}(\log \scaledTotalWeightDenseExpectation)^{-1/16}\le W/w_{n-r+1}.
$$
Furthermore, by \eqref{eq:defHeavy} the number of heavy vertices satisfies 
$$
\left|\VertexSetRestricted{\ge\weightBoundHeavy}\right|\ge \left(\frac{W}{4\weightBoundHeavy^2}\right)^r\ge 4^{-r}(\log \scaledTotalWeightDenseExpectation)^{r/4}\to\infty,
$$
and thus, in particular, we also obtain
$$
\weightBoundKernel=\weightBoundHeavy\le w_{n-r+1}.
$$
Therefore the weight-bound $\weightBoundKernel$ is sufficiently small in this case as well.

We will prove that whp all vertices of the candidate nucleus becomes infected eventually. Again, we first assume that $\weightBoundKernel=W/w_{n-r+1}$. In this case we start by showing that there is a $T\ge0$ at which whp the $r$ vertices of largest weight become infected, i.e.,\ 
\begin{equation}\label{eq:rInfectedVertices}
\{n-r+1,\dots,n\}\subseteq \infec{T}.
\end{equation} 
 
 If~\eqref{cond:smallInfectionRate} holds, then Lemma~\ref{lem:denseHeavyVertices} implies in particular that the number of heavy vertices is $\omega(1)$ and by Lemma~\ref{lem:JLTV} they all become infected by some time $T\ge0$, hence~\eqref{eq:rInfectedVertices} holds. On the other hand, if~\eqref{cond:smallInfectionRate} does not hold, then~\eqref{eq:highInfectionProbability} holds for each of the $r$ vertices of largest weight (since $w_{n-r+1}=\alpha w_n$, for some constant $\alpha>0$), and thus a union bound shows that whp $\{n-r+1,\dots,n\}\subseteq\infecCoupling{1}\subseteq \infec{1}$, i.e.,~\eqref{eq:rInfectedVertices} holds.
 
 Now consider any vertex $u\in\VertexSetRestricted{\ge W/w_{n-r+1}}$ and note that $u$ is connected to each vertex in $\{n-r+1,\dots, n\}$ with probability $1$ (cf.~\eqref{eq:pijCL}). Consequently, we have $\VertexSetRestricted{\ge\weightBoundKernel}=\VertexSetRestricted{\ge W/w_{n-r+1}}\subseteq \infec{T+1}$ implying $\totalInfectedWeightRestricted{\ge\weightBoundKernel}=\totalWeightRestricted{\ge\weightBoundKernel}$, in other words $\VertexSetRestricted{\ge\weightBoundKernel}$ is a nucleus.\medskip

Now suppose that $\weightBoundKernel=\weightBoundHeavy$, i.e.,~\eqref{cond:largeInfectionRate} holds. Then Lemmas~\ref{lem:denseHeavyVertices} and~\ref{lem:JLTV} imply that whp all heavy vertices become infected eventually and thus $\totalInfectedWeightRestricted{\ge\weightBoundKernel}=\totalWeightRestricted{\ge\weightBoundKernel}$, completing the proof.   
\end{proof}

We conclude this section by proving that if we cannot apply Theorem~\ref{thm:kernelDense}, then the premises of Theorem~\ref{thm:kernelSparse} are met, justifying the previous case distinction.
\begin{claim}\label{close}
Assume that $w_{n-r+1}\le\sqrt{W}(\log\scaledTotalWeightDenseExpectation)^{1/16}$ but $\weightBoundHeavy>\sqrt{W}(\log\scaledTotalWeightDenseExpectation)^{-1/8}$.
Then we have 
$$
\candidateThresholdDense(\log\scaledTotalWeightDenseExpectation)^{2r}\ge \candidateThresholdSparse,
$$
and in particular
$$
\infectionRate=\omega(\candidateThresholdSparse).
$$
\end{claim}
\begin{proof}
First set $w':=\sqrt{W}(\log\scaledTotalWeightDenseExpectation)^{-1/8}<\weightBoundHeavy$ and note that
\begin{equation}\label{eq:upperboundonheavyvertices}
|\VertexSetRestricted{\ge\weightBoundHeavy}|\le|\VertexSetRestricted{\ge w'}|\stackrel{\eqref{eq:defHeavy}}{<}\left(\frac{W}{4(w')^2}\right)^r=4^{-r}(\log \scaledTotalWeightDenseExpectation)^{r/4},
\end{equation}
and consequently
\begin{equation}\label{eq:lowerboundcandidated}
\candidateThresholdDense
\stackrel{\eqref{eq:candidateThresholdDense}}{=}\left(\frac{1}{\sum_{u\in\VertexSetRestricted{\ge\weightBoundHeavy}}w_u^r}\right)^{1/r}\ge\left(\frac{1}{|\VertexSetRestricted{\ge\weightBoundHeavy}|w_n^r}\right)^{1/r}
\stackrel{\eqref{eq:upperboundonheavyvertices}}{=}\Omega\left(W^{-1/2}(\log \scaledTotalWeightDenseExpectation)^{-5/16}\right),
\end{equation}
since $w_{n-r+1}=\alpha w_n$, for some constant $\alpha>0$ by the assumptions of Theorem~\ref{threshold}.
Therefore 
\begin{align*}
	\totalWeightRestricted{\ge\weightBoundHeavy}&\stackrel{\eqref{eq:upperboundonheavyvertices}}{\leq} 4^{-r}(\log \scaledTotalWeightDenseExpectation)^{r/4}w_n=O\left(\sqrt{W}(\log \scaledTotalWeightDenseExpectation)^{r/4}(\log \scaledTotalWeightDenseExpectation)^{1/16}\right)\\
	&=o\left(\sqrt{W}(\log \scaledTotalWeightDenseExpectation)^{r/2}\right)
\end{align*}
and thus $\mathbb{P}[\wbrv \ge\weightBoundHeavy]=o\left((\log \scaledTotalWeightDenseExpectation)^{r/2}/\sqrt{W}\right)$.
By Proposition~\ref{prop:restrmom} we have 
\begin{align*}
	\sum_{u\in \VertexSetRestricted{<\weightBoundHeavy}}w_u^{r+1}&\geq rW \left(\int_{C_1}^{\weightBoundHeavy}\mathbb{P}[\weightBoundHeavy>\wbrv\geq \mu]\mu^{r-1}d\mu\right)\\
	&\geq rW \left(\int_{C_1}^{w'/(\log\scaledTotalWeightDenseExpectation)^{r/2}}\left(\frac{C}{\mu}-o\left(\frac{(\log \scaledTotalWeightDenseExpectation)^{r/2}}{\sqrt{W}}\right)\right)\mu^{r-1}d\mu\right)\\
	&\stackrel{w'= o(\sqrt{W})}{=}\Omega\left(\frac{W (w')^{r-1}}{(\log\scaledTotalWeightDenseExpectation)^{(r-1)r/2}}\right).
\end{align*}


Combining \eqref{eq:lowerboundcandidated} with this we obtain the first statement. Moreover, this then implies
$$
\infectionRate\stackrel{\eqref{def:scaledTotalWeightDenseExpectation}}{=}\scaledTotalWeightDenseExpectation \candidateThresholdDense\ge \frac{\scaledTotalWeightDenseExpectation}{(\log \scaledTotalWeightDenseExpectation)^{2r}}\candidateThresholdSparse=\omega(\candidateThresholdSparse),
$$
i.e.,\ the second statement holds as well.
\end{proof}

\subsection{The existence of a nucleus: proof of Theorem~\ref{thm:kernel}}\label{sec:proofKernel}
In the preceding sections we gave several constructions for a nucleus, and we will now demonstrate how these results can be combined into the proof of Theorem~\ref{thm:kernel}. It is structured in three steps: at the beginning of each step we check whether some condition is satisfied, if it is, then whp one of the preceding results guarantees the existence of a nucleus; if not, we gained some information and proceed to the next step. 
\begin{proof}[Proof of Theorem~\ref{thm:kernel}]
We intend to use the subsubsequence principle (cf.  for example~\cite{JLR}), which states that in order to show that a certain property holds whp, it suffices to show that for every sequence of natural numbers there exists a subsequence along which the property holds whp. So let $N\subseteq\mathbb{N}$ be an arbitrary (infinite) sequence of natural numbers satisfying the conditions of Theorem~\ref{threshold}. Then it must contain an (infinite) subsequence $N_0\subseteq N$ satisfying either
\begin{equation}\label{cond:denseDoesNotExist}
\weightBoundHeavy> w_n\qquad\text{and}\qquad\infectionRate\gg\candidateThresholdSparse,
\end{equation}
or 
\begin{equation}\label{cond:denseExists}
\weightBoundHeavy\le w_n\qquad\text{and}\qquad\infectionRate\gg\min\{\candidateThresholdSparse,\candidateThresholdDense\}.
\end{equation}

We thus distinguish two cases.

\noindent
{\bf Case I.} If along $N_0$ we have
$$
\weightBoundHeavy>w_n,
$$
then by~\eqref{cond:denseDoesNotExist} we have $\infectionRate\gg\candidateThresholdSparse$, and furthermore Corollary~\ref{cor:candidatethresholdsSmall} implies $\candidateThresholdSparse=o(1)$. Therefore we consider the sparse process and note that Theorem~\ref{thm:kernelSparse} is applicable, proving that whp there exists a nucleus $\VertexSetRestricted{\ge\weightBoundKernel}$. In this case we call $N_0$ \emph{sparse}.\medskip

\noindent
{\bf Case II.} 
Assume now that $\weightBoundHeavy\le w_n$ along $N_0$, and therefore  also $\infectionRate\gg\min\{\candidateThresholdSparse,
\candidateThresholdDense\}$ by~\eqref{cond:denseExists}. 
Similarly as before, Corollary~\ref{cor:candidatethresholdsSmall} provides $\min\{\candidateThresholdSparse,
\candidateThresholdDense\}=o(1).$
We split $N_0$ into two subsequences: let $N_1 = \{n \in N_0 \ : \ p_d\geq p_s\}$ and $N_2 =\{n \in N_0 \ : \ p_d < p_s \} = N_0 \setminus N_1$. 

If $N_1$ is infinite, then we have
$$
\infectionRate\gg\candidateThresholdSparse
$$ 
along $N_1$. In this case, we consider the sparse process and observe that since $\candidateThresholdSparse=o(1)$ 
Theorem~\ref{thm:kernelSparse} is applicable, proving that whp there exists a nucleus $\VertexSetRestricted{\ge\weightBoundKernel}$. Once again we call $N_0$ \emph{sparse}. \medskip

Otherwise $N_2$ must be infinite, and we have 
 \begin{equation}\label{contradiction}
\infectionRate \gg\candidateThresholdDense
 \end{equation}
 along $N_2$.
 Now, if
$$
w_{n-r+1}>\sqrt{W}(\log \scaledTotalWeightDenseExpectation)^{1/16},
$$
along an infinite subsequence of $N_2$, then we consider the dense process. It follows from Theorem~\ref{thm:kernelDense} that along this subsequence whp there exists a nucleus $\VertexSetRestricted{\ge\weightBoundKernel}$. In this case we call $N_0$ \emph{dense}. 
\medskip

Thus, we may assume that there exists an infinite subsequence $N_3$ of $N_2$ along which we have $w_{n-r+1}\le\sqrt{W}(\log \scaledTotalWeightDenseExpectation)^{1/16}$. We consider two more cases.
If there exists an infinite subsequence of $N_3$ along which we have 
$$
\weightBoundHeavy\le \sqrt{W}(\log\scaledTotalWeightDenseExpectation)^{-1/8},
$$
then we also consider the dense process and apply Theorem~\ref{thm:kernelDense}, thus showing that whp there exists a nucleus $\VertexSetRestricted{\ge\weightBoundKernel}$. Again we call $N_0$ \emph{dense}.
\medskip

Therefore we finally assume that $N_3$ contains an infinite subsequence $N_4$ along which $\weightBoundHeavy>\sqrt{W}(\log\scaledTotalWeightDenseExpectation)^{-1/8}$ is satisfied. But this tells us that the requirements of Claim~\ref{close} are met along $N_4$. Hence it would imply that 
over $N_4$ we also have $p_0 \gg p_s$. Thereby, Theorem~\ref{thm:kernelSparse} is again applicable, proving that whp there exists a nucleus $\VertexSetRestricted{\ge\weightBoundKernel}$. In this final case we call $N_0$ \emph{sparse}.
 
By the subsubsequence principle the statement of Theorem~\ref{thm:kernel} holds along $\mathbb{N}$.
\end{proof}

Inspired by the case distinction in the above proof of Theorem~\ref{thm:kernel}, we define another partition of the set all infinite sequences of natural numbers, which we will use in Section~\ref{outbreak}.
\begin{definition}\label{def:partition}
Let $\mathcal{N}_s:=\{N\subseteq\mathbb{N}\colon |N|=\infty\text{ and }N\text{ is sparse}\}$ and $\mathcal{N}_d:=\{N\subseteq \mathbb{N}\colon |N|=\infty\text{ and }N\text{ is dense}\}$.
\end{definition}

\section{Outbreak}\label{outbreak}

In the previous section, we proved that in the supercritical regime there exists a nucleus which gets infected almost entirely. In this section, we show that once this happens, whp it also causes an outbreak.
\begin{theorem}\label{thm:outbreak}
Suppose that the premises of Theorem~\ref{threshold} hold. If $\infectionRate\gg \threshold$, then whp  
there is an outbreak, that is, we have
 $$
 |\infec{F}|=\Theta(n).
 $$
\end{theorem}

Again we use the subsubsequence principle to prove Theorem~\ref{thm:outbreak}. So we fix an arbitrary infinite sequence of natural numbers $N_0$ and show that it contains a subsequence along which the probability of an outbreak is $1-o(1)$. To this end we use the distinction between $N_0$ being sparse or dense provided by Definition~\ref{def:partition}.

In both cases, the proof is based on partitioning the vertex set $\VertexSetRestricted{\ge C_1}$, where the constant $C_1>0$ is the one from Theorem~\ref{threshold}, into several \emph{layers}, each corresponding to a subinterval of $[C_1,\infty)$, and then showing that the infection spreads from layer to layer, starting from a completely or an 
almost completely infected nucleus.  

More precisely, assume that most of the vertices in each of the, say $i$, heaviest layers are infected eventually. We prove that also most of the vertices in the $(i+1)$-st layer must become infected eventually. Some of the details depend sensitively on the relation of various error-terms. We begin with the definitions of these in the next subsection.

\subsection{Setup}

Recall that by the assumptions of Theorem~\ref{threshold} there are constants $\alpha>0$ and $C$ satisfying 
\begin{equation}\label{cond:powerlawThresholdConstant}
C\ge 64 r(\min\{\alpha,1/2\})^{-3}
\end{equation}
and $C_1>0$ such that $w_{n-r+1}=\alpha w_n$ and for any $C_1\le x\le w_n$ we have
\begin{equation}\label{cond:powerlawThreshold}
\PP\left[\wbrv\ge x\right]\ge \frac{C}{x}.
\end{equation}
Next abbreviate
\begin{equation}\label{def:powerlawThresholdConstantAltered}
C':=C\min\{\alpha,1/2\}/2,
\end{equation}
and use this constant to recursively define a non-increasing sequence of layer weight-bounds. 

More specifically, we set $\weightBoundLayer{1}:=\min\{\weightBoundKernel,W/w_n\}$ and  
\begin{equation} \label{def:weightBoundLayer}
\weightBoundLayer{i+1} := \frac{C'}{\PPo{\wbrv \geq \weightBoundLayer{i}}}
\end{equation}
for any $i\ge 1$ as long as 
\begin{equation}\label{def:weightBoundLayerLast}
\weightBoundLayer{i}\ge 2\max\{C_1,\weightConstant \},
\end{equation}
where $\weightConstant\ge1$ is the constant in~\eqref{Wcondition}. We denote the maximal such $i$ by $\lastLayer$ and observe that for each $1\le i\le\lastLayer$ we obtain 
$$
\weightBoundLayer{i+1}\le \weightBoundLayer{i}\frac{C'}{C}< \weightBoundLayer{i}.
$$
This implies, in particular, that
\begin{equation}\label{eq:weightBoundLayerProperties}
\weightBoundLayer{i+1}\le \weightBoundLayer{1}\left(\frac{C'}{C}\right)^{i}.
\end{equation}
Moreover, for convenience of notation we also set $\weightBoundLayer{0}:=\weightBoundKernel$ and note that we have 
\begin{equation}\label{eq:weighBoundLayerNonIncreasing}
\weightBoundLayer{0}\ge\weightBoundLayer{1}>\weightBoundLayer{2}>\dots>
\weightBoundLayer{\lastLayer+1}>0.
\end{equation}
Furthermore, recall that the weight-bound $\weightBoundLayer{0}$ satisfies
\begin{equation}\label{eq:kernelWeightBoundProp}
\weightBoundLayer{0}\le\min\{w_{n-r+1},W/w_{n-r+1}\}\qquad\text{and}\qquad\weightBoundLayer{0}=o(\sqrt{W}).
\end{equation}

 If $N_0$ is sparse, then we consider a subsequence $N_s$ which witnesses this fact, i.e., Theorem~\ref{thm:kernelSparse} is applicable along $N_s$. In this case, we set $\breedingGround'$ as in Section~\ref{sparse} (and recall that $\breedingGround\subseteq\breedingGround'$).  On the other hand, if $N_0$ is dense, then we consider a subsequence $N_d$ which witnesses this fact, i.e., Theorem~\ref{thm:kernelDense} is applicable along $N_d$, and define $\breedingGround':=\emptyset$.  For the remainder of this section all asymptotic statements will be with respect to $N_s$ or $N_d$, respectively. 
 
 Then for any $1\le i\le\lastLayer$ such that $\weightBoundLayer{i}<\weightBoundLayer{i-1}$ we call the set
 \begin{equation}
 \Layer{i}:= \VertexSetRestricted{[\weightBoundLayer{i},\weightBoundLayer{i-1})} \setminus \breedingGround' 
 \end{equation} 
 the \emph{$i$-th layer}, and in case 
 $\weightBoundLayer{i}=\weightBoundLayer{i-1}$ the $i$-th layer is the empty set. We also set 
 $\Layer{0}:=\VertexSetRestricted{\ge \weightBoundKernel}$. We let $\Kernel{i} :=\cup_{0\leq j\leq i} \Layer{i}$ be the union of the first $i+1$ layers 
and let $\WeightLayers{i}:= \totalWeight{\Kernel{i}}$ be its total weight.

The following claim will be useful later in our argument. 
\begin{claim} \label{clm:weightlayer_low}
For any $0\leq i \leq \lastLayer$ we have 
$$\WeightLayers{i}\geq \frac{1}{4} \totalWeight{\VertexSetRestricted{\ge \weightBoundLayer{i}}}. $$
\end{claim}
\begin{proof}
For $N_0\in \mathcal{N}_d$ we have $\breedingGround'=\emptyset$, whereby $\WeightLayers{i}= \totalWeight{\VertexSetRestricted{\ge \weightBoundLayer{i}}},$ and hence the statement holds trivially.

So now assume $N_0\in\mathcal{N}_s$. Recall that we consider all vertices in $\VertexSet$ to be ordered by weight. By the definition of $\breedingGround'$ the set $\Kernel{i}$ contains at least every other vertex of $\VertexSetRestricted{\ge \weightBoundLayer{i}}$ together with all vertices in
$\VertexSetRestricted{\ge \weightBoundKernel}$.  Therefore, for every vertex in  $\Kernel{i}\setminus \Kernel{0}$ its preceding vertex could belong to $\breedingGround'$. Thus, roughly speaking, $\WeightLayers{i}$ is approximately at least half of $\totalWeight{\VertexSetRestricted{\ge \weightBoundLayer{i}}}$.  More precisely, if we exclude the first vertex of $\Kernel{i}$ and the last vertex of 
$\VertexSetRestricted{[\weightBoundLayer{i},\weightBoundKernel)}$ which 
may not be in $\Kernel{i}$, we obtain the lower bound
$$\WeightLayers{i}\geq \frac{1}{2} \left( \totalWeight{\VertexSetRestricted{\ge \weightBoundLayer{i}}} - 2\weightBoundKernel \right)=\frac{1}{2} \left( W\PP[\wbrv\ge\weightBoundLayer{i}] - 2\weightBoundKernel \right),$$
as the two vertices we exclude both have weight at most $\weightBoundKernel$. 
Now, we would like to show that 
$$ W\PP[\wbrv\ge\weightBoundLayer{i}] - 2\weightBoundKernel \ge \frac{1}{2} W\PP[\wbrv\ge\weightBoundLayer{i}] $$
or, equivalently,  
\begin{equation}\label{eq:intermediate} \frac{1}{2} W\PP[\wbrv\ge\weightBoundLayer{i}]  \ge 2\weightBoundKernel. \end{equation}
Recall that by~\eqref{cond:powerlawThreshold} we have
$$ \frac{1}{2} W\PP[\wbrv\ge\weightBoundLayer{i}]  \geq \frac{1}{2}  \frac{CW}{\weightBoundLayer{i}}. $$
Furthermore \eqref{eq:weighBoundLayerNonIncreasing} implies $\weightBoundLayer{i}\leq \weightBoundKernel$ and by Theorem~\ref{thm:kernelSparse} we have $\weightBoundKernel=o(\sqrt{W})$. Hence $\weightBoundKernel =o(W/\weightBoundLayer{i})$ and inequality \eqref{eq:intermediate} follows.
\end{proof}

 Now, for some subset $U \subset V_{\geq \weightBoundKernel}$, we consider the restricted process $\postKernelProcess{t}{U}$ on $G[\Kernel{\lastLayer}]$ with initially infected set $\postKernelProcess{0}{U}:=U$. 
 In fact we will use either $U=\VertexSetRestricted{\ge \weightBoundKernel}$ (for $N_0\in\mathcal{N}_d$) or $U=\mathcal{U}_K$ (for $N_0\in\mathcal{N}_s$).
 
%

Finally, for any $i\ge0$ we introduce two error-terms: we set 
$$ 
\delta_i := \frac{1}{4}\left( \frac{C'}{C} \right)^i\qquad\text{and\qquad}\eps_i := \sum_{j=0}^i \delta_j
$$
and note that this implies
\begin{equation}\label{eq:propertiesEps}
\frac{1}{4}=\eps_0\le \eps_1\le\dots\le \eps_{\lastLayer+1} \leq \sum_{j=0}^\infty  \delta_j = \frac{1}{4} \sum_{j=0}^\infty \left(\frac{C'}{C} \right)^j = \frac{1}{4} \cdot\frac{1}{1-C'/C} 
\stackrel{\eqref{def:powerlawThresholdConstantAltered}}{\le} 1/2,
\end{equation}
and also, for any $1\le i < \lastLayer$, the following estimate (which will be used later)
\begin{equation}\label{eq:propertiesDelta}
\delta_{i+1}\weightBoundLayer{i}^{-2}\stackrel{\eqref{eq:weightBoundLayerProperties}}{\ge}\frac{1}{4} \weightBoundLayer{1}^{-2}\left( \frac{C'}{C} \right)^{3-i} \stackrel{\eqref{def:powerlawThresholdConstantAltered}}{\ge} 2^{i-7}\weightBoundLayer{1}^{-2}\alpha^2.
\end{equation}
Using these we define, for any $0\le i\le\lastLayer$, the event
\begin{equation*} 
\inductionLayerEvent{i}:=\left\{\InfectedWeightKernel{i} \geq (1 - \eps_i )\WeightLayers{i} \right\},
\end{equation*}
i.e.,\ $\inductionLayerEvent{i}$ asserts that the union of the nucleus and the first $i$ layers got infected almost entirely in the first $i$ steps. Hence, the goal will be to prove that $\inductionLayerEvent{\lastLayer}$ holds whp. We do so inductively in Section~\ref{sec:outbreakLargeLayers}. Depending on the total weight of a layer (in comparison to the union of the nucleus and the previous layers) the induction steps can be straightforward or rather complicated. 

More precisely, for any $1\le i\le \lastLayer$ we say that the $i$-th layer has property $\largeLayerEvent{i}$ if 
\begin{equation*} 
\WeightLayers{i} \ge (1+\delta_{i})\WeightLayers{i-1}.
\end{equation*}
We will prove in Section~\ref{sec:outbreakLargeLayers} that if this property holds, then with sufficiently high probability a large enough fraction of the total weight in the $i$-th layer becomes infected eventually.

\subsection{Infecting large layers}\label{sec:outbreakLargeLayers}

First we show that if the $i+1$-st layer has property $\largeLayerEvent{i+1}$, then conditional on $\inductionLayerEvent{i}$ the expected number of eventually infected neighbours of a vertex in the $i$-th layer is not too small. More precisely, for $0\le i <\lastLayer$ and a vertex $v\in \Layer{i+1}$ we denote the number of its neighbours in $\postKernelProcess{i}{U}\cap \Kernel{i}$ by 
$$
\infdeg{v}{i}:=\sum_{u\in \postKernelProcess{i}{U}\cap\Kernel{i}}\Ind{\{u,v\}\in E}.
$$
\begin{claim}\label{claim:largeLayerDegree}
Let $0\le i < \lastLayer$. If $\largeLayerEvent{i+1}$ holds, then for every $v \in \Layer{i+1} \setminus \infecCoupling{0}$ we have 
$\EE\left[\infdeg{v}{i}\cond \inductionLayerEvent{i}\right]\ge(1-\eps_i)C'/4$.
\end{claim}
\begin{proof}
Let $0\le i < \lastLayer$ and let $v\in\Layer{i+1} \setminus \infecCoupling{0}$. 
Assume that the event $\inductionLayerEvent{i}$ has been realised. 
Observe that each one of the indicator random variables satisfies 
$$
\PP\left[\Ind{\{u,v\}\in E}=1\cond \inductionLayerEvent{i}\right]=\min\left\{\frac{w_vw_u}{W},1\right\}\geq \min\left\{\frac{\weightBoundLayer{i+1}w_u}{W},1\right\}=\frac{\weightBoundLayer{i+1}w_u}{W}
$$ 
since $\weightBoundLayer{i+1}\le\weightBoundLayer{1}\le W/w_n$.
Thus we have
\begin{align} \label{eq:exp_degree}
\EE\left[\infdeg{v}{i}\cond \inductionLayerEvent{i}\right]& \geq \sum_{u \in \postKernelProcess{i}{U}\cap\Kernel{i}}\frac{\weightBoundLayer{i+1} w_u}{W} =\frac{\weightBoundLayer{i+1}\InfectedWeightKernel{i}}{W}
\stackrel{\inductionLayerEvent{i}}{\ge}\frac{(1- \eps_i ) \weightBoundLayer{i+1}\WeightLayers{i} }{W}.
\end{align}
By Claim~\ref{clm:weightlayer_low} and~\eqref{cond:powerlawThreshold} we have
\begin{equation} \label{eq:kernelweight_low}
\WeightLayers{i}\geq 
\frac{1}{4} \frac{CW}{\weightBoundLayer{i}}.
\end{equation}
So this implies that 
\begin{equation} \label{eq:case_full_I}
\EE\left[\infdeg{v}{i}\cond \inductionLayerEvent{i}\right]\geq 
(1-\eps_i)\frac{C}{4}\frac{\weightBoundLayer{i+1}}{\weightBoundLayer{i}}.
\end{equation}

Now, for any $i\ge1$ the claim follows directly by plugging~\eqref{def:weightBoundLayer} into the right-hand side
of~\eqref{eq:case_full_I}. 

Thus, it remains to consider the case $i=0$. First observe that  $W/w_n\geq \weightBoundLayer{0}$ would imply $\weightBoundLayer{1}=\weightBoundLayer{0}$ and thus the first layer would be empty contradicting property $\largeLayerEvent{1}$. Hence we may also assume that $\weightBoundLayer{1}=W/w_n$, and consequently 
$$
\EE\left[\infdeg{v}{ 0}\cond \inductionLayerEvent{0}\right]\geq (1- \eps_0)W\PP[\wbrv\geq \weightBoundLayer{0}] /w_n \stackrel{\eqref{cond:powerlawThreshold}}{\ge} (1- \eps_0) \frac{CW}{\weightBoundLayer{0}w_n}.
$$
But recall that
$\weightBoundLayer{0} = \weightBoundKernel \leq W/w_{n-r+1}=W/(\alpha w_n)$, which, in turn, implies that 
$$
\EE\left[\infdeg{v}{0}\cond \inductionLayerEvent{0}\right]\geq (1-\eps_0)\alpha C\stackrel{\eqref{def:powerlawThresholdConstantAltered}}{\ge} (1-\varepsilon_0) C',
$$
proving the claim also in this case, completing the proof.
\end{proof}
Claim \ref{claim:largeLayerDegree} allows us to compute a suitable lower bound on the expected total infected 
weight within the $(i+1)$-st layer. More precisely, for $0\le i < \lastLayer$ and a vertex $v\in \Layer{i+1}$ we consider 
the random variables
$$
\widehat{X}_{i,v}:=w_v\Ind{\infdeg{v}{ i}\ge r}
$$
and note that these satisfy $0\le\widehat{X}_{i,v}\le w_v$. 
Note that 
\begin{equation*} 
\InfectedWeightLayers{i+1} = \sum_{v \in \Layer{i+1} \setminus \infecCoupling{0}} \widehat{X}_{i,v} + \sum_{v \in \Layer{i+1} \cap \infecCoupling{0}} w_v. 
\end{equation*}
For each $v \in \Layer{i+1} \cap \infecCoupling{0}$ we define the random variable $\widehat{X}_{i,v} = w_v I_v$, where $I_v$ is the indicator random variable such that 
$\PP\left[ I_v=1\right] = \PP\left[ \sum_{u \in \postKernelProcess{i}{U} \cap \Kernel{i}} I_{u,v} \ge r\right]$, where $\{ I_{u,v} \}_{v \in \Layer{i+1}, u \in \Kernel{i}}$ is a collection of independent 
indicator random variables with $I_{u,v} = 1$ with probability given by~\eqref{eq:pijCL}.

Hence, setting
$$
\widehat{X}_{i+1}:=\sum_{v\in\Layer{i+1} }\widehat{X}_{i,v}
$$
conditional on $\inductionLayerEvent{i}$ the random variable  
$\InfectedWeightLayers{i+1}$ stochastically dominates $\widehat{X}_{i+1}.$ 
In the next two lemmas, we will use the stochastic domination in order to deduce that whp $\InfectedWeightLayers{i+1}$ is a significant proportion of $\totalWeight{\Layer{i+1}}$.  

\begin{claim}\label{claim:largeLayerExpectation} For any $0\le i < \lastLayer$, we have
$\EE\left[\widehat{X}_{i+1}\cond \inductionLayerEvent{i}\right]\ge(1-\eps_i^2) \totalWeight{\Layer{i+1}}$.
\end{claim}
\begin{proof}
First consider a vertex $v\in \Layer{i+1} \backslash \infecCoupling{0}$. Since  $\infdeg{v}{ i}$ 
is a sum of independent indicator random variables, we have 
$$
 \VV\left[\infdeg{v}{ i} \cond \inductionLayerEvent{i}\right] \leq \EE\left[\infdeg{v}{ i}
 \cond \inductionLayerEvent{i}\right].
 $$
Consequently, we obtain from Chebyshev's inequality that
\begin{align}
\PP\left[\infdeg{v}{ i} < r \cond \inductionLayerEvent{i}\right] & \leq 
\frac{\EE\left[ \infdeg{v}{ i}\cond \inductionLayerEvent{i}\right]}{\left( \EE\left[ \infdeg{v}{i} \cond 
\inductionLayerEvent{i}\right] -r\right)^2} \nonumber\\
& \stackrel{C.\ref{claim:largeLayerDegree}}{\le} \frac{4}{(1-\eps_i)C'}~ \frac{1}{\left( 1 - 4r/((1-\eps_i)C')\right)^2} \nonumber\\
&\stackrel{\eqref{eq:propertiesEps}}{\le} 
\frac{8}{C'}~ \frac{1}{\left( 1 - 8r/C'\right)^2}.\label{eq:largeLayerProbabilityNonInfection}
\end{align}

Now observe that from~\eqref{def:powerlawThresholdConstantAltered} and~\eqref{cond:powerlawThresholdConstant} we obtain the following lower bounds
\begin{equation*} 
C'\ge 64 r(\min\{\alpha,1/2\})^{-2}\ge 2^7 r.
\end{equation*}
The second inequality implies $(1-(8r/C'))^{-2}\le 2$ and also the following upper bound on the right-hand side in~\eqref{eq:largeLayerProbabilityNonInfection}
$$
\frac{16}{C'}\le \frac{1}{8r}\stackrel{r\ge2}{\le}\eps_0^2.
$$
Hence, since $\eps_0 < \eps_i$, for $i >0$, we have $\frac{16}{C'} \le \eps_i^2$. 
We therefore obtain 
$$
\PP\left[\infdeg{v}{ i}\ge r\cond \inductionLayerEvent{i}\right]\ge 1-\eps_i^2,
$$
for any $0\le i < \lastLayer$.
The random variable $\sum_{u \in \postKernelProcess{i}{U} \cap \Kernel{i}} I_{u,v}$ also satisfies 
Claim~\ref{claim:largeLayerDegree} and therefore the above argument also holds there. 
By summing up over all vertices $v \in\Layer{i+1}$ the statement follows.
\end{proof}
Next we extend Claim ~\ref{claim:largeLayerExpectation} and show that the probability that the total infected weight in the $(i+1)$-st layer is not large enough is sufficiently small.
\begin{lemma}\label{lem:largeLayer}
For any $0\le i < \lastLayer$ for which property $\largeLayerEvent{i+1}$ holds, we have 
$$
\PP\left[\InfectedWeightLayers{i+1}\le(1-\eps_i )\totalWeight{\Layer{i+1}} \cond 
\inductionLayerEvent{i}\right] \leq 
\exp \left(-2^{i-6} W\weightBoundLayer{1}^{-2} \right).
$$
\end{lemma} 
\begin{proof}
Let $0\le i < \lastLayer$ and recall that $\widehat{X}_{i+1}$ is a sum of independent random variables $\widehat{X}_{i,v}$ satisfying $0\le\widehat{X}_{i,v}\le w_v$ for $v\in\Layer{i+1}$. 
Moreover, note that Claim~\ref{claim:largeLayerExpectation} implies
$$
(1- \eps_i)\totalWeight{\Layer{i+1}} \le 
\EE\left[\widehat{X}_{i+1}\cond \inductionLayerEvent{i}\right]-
\eps_i (1-\eps_i) \totalWeight{\Layer{i+1}}.
$$
Therefore, using the above stochastic domination and applying the Azuma-Hoeffding inequality (Theorem~\ref{AzumaHoeffding}) to $\widehat{X}_{i+1}$ we obtain
\begin{eqnarray} \label{eq:largeLayerAzumaHoeffding}
\lefteqn{\PP \left[ \InfectedWeightLayers{i+1}\le(1-\eps_i )\totalWeight{\Layer{i+1}} \cond 
\inductionLayerEvent{i} \right] \leq} \nonumber \\
& &\PP\left[\widehat{X}_{i+1}\le(1-\eps_i )\totalWeight{\Layer{i+1}} \cond \inductionLayerEvent{i}\right] 
\leq \exp \left(- \frac{\eps_i^2 (1 - \eps_i )^2}{2}~\frac{\totalWeight{\Layer{i+1}}^2}{\sum_{u\in\Layer{i+1}} w_u^2} \right).
\end{eqnarray} 

We proceed by bounding the argument of the exponential function on the right-hand side from below by splitting it into three factors. First note that we have
$$
\frac{\totalWeight{\Layer{i+1}}}{\sum_{u\in\Layer{i+1}} w_u^2}\geq  \weightBoundLayer{i}^{-1}
$$
and also
$$
\totalWeight{\Layer{i+1}}=\WeightLayers{i+1}-\WeightLayers{i} \stackrel{\largeLayerEvent{i+1}}{\geq}
\delta_{i+1} \WeightLayers{i} 
\stackrel{\eqref{eq:kernelweight_low}}{\geq} 
\frac{\delta_{i+1}}{4}\weightBoundLayer{i}^{-1}CW. $$
Multiplying these two factors we obtain
\begin{equation}\label{eq:largeLayerProbabilityBound}
\frac{\totalWeight{\Layer{i+1}}^2}{\sum_{u\in\Layer{i+1}} w_u^2}\ge \frac{\delta_{i+1}}{4}\weightBoundLayer{i}^{-2}CW\stackrel{\eqref{eq:propertiesDelta}}{\ge}2^{i-9}\weightBoundLayer{1}^{-2}CW\alpha^2\ge 2^{i+1}\weightBoundLayer{1}^{-2}W\alpha^2,
\end{equation}
because $C\ge 2^{10}$ by~\eqref{cond:powerlawThresholdConstant}. Consequently it remains to bound the last factor
$$
\eps_i^2(1-\eps_i)^2/2\stackrel{\eqref{eq:propertiesEps}}{\ge} 2^{-7},
$$
and since $\InfectedWeightLayers{i+1}\ge \widehat{X}_{i+1}$ this yields
$$
\PP\left[\InfectedWeightLayers{i+1}\leq (1-\eps_i )\totalWeight{\Layer{i+1}} \cond \inductionLayerEvent{i}\right] \stackrel{\eqref{eq:largeLayerAzumaHoeffding}, \eqref{eq:largeLayerProbabilityBound}}{\le}
\exp \left(-2^{i-6}\alpha^2 W\weightBoundLayer{1}^{-2} \right),
$$
as desired.
\end{proof}

We will use the above construction to show inductively the following.
\begin{lemma}\label{lem:almostCompleteInfection}
If $U\subseteq \VertexSetRestricted{\ge \weightBoundKernel}$ satisfies
\begin{equation}\label{eq:kernelAlmostInfection}
\totalWeight{U} \ge (1-o(1))\wcore{0},
\end{equation}
then $\inductionLayerEvent{\lastLayer}$ holds whp. 
\end{lemma}
\begin{proof}
We proceed by induction on $i$ and show that $\inductionLayerEvent{i}$ holds for all $0\le i\le \lastLayer$ with sufficiently high probability. The statement for the base case $i=0$ holds since $\VertexSetRestricted{\ge \weightBoundKernel}$ is a nucleus and thus $\inductionLayerEvent{0}$ holds (with probability 1) by~\eqref{eq:kernelAlmostInfection}. 

So now assume that $0\le i < \lastLayer$ and $\inductionLayerEvent{i}$ holds. Additionally suppose that 
\begin{equation*}
\WeightLayers{i+1} < (1+\delta_{i+1}) \WeightLayers{i},
\end{equation*} 
i.e.,\ the property $\largeLayerEvent{i+1}$ does not hold. Then deterministically we have
\begin{align*}
\InfectedWeightKernel{i+1} \stackrel{\eqref{eq:weighBoundLayerNonIncreasing}}{\ge} \InfectedWeightKernel{i} \stackrel{\inductionLayerEvent{i}}{\geq}  (1- \eps_i)\WeightLayers{i} \stackrel{\neg\largeLayerEvent{i+1}}{>}&
\frac{1-\eps_i}{1+\delta_{i+1}}\WeightLayers{i+1} \\
&\hspace{1cm}> (1 -\eps_{i+1})\WeightLayers{i+1},
\end{align*}
implying $\inductionLayerEvent{i+1}$. 

Otherwise, property $\largeLayerEvent{i+1}$ holds, and so Lemma~\ref{lem:largeLayer} is applicable, showing that 
$$
\InfectedWeightLayers{i+1}>(1-\eps_{i})\totalWeight{\Layer{i+1}}
$$
with (conditional) probability at least $1-\exp(-2^{i-6}\alpha^2W\weightBoundLayer{1}^{-2})$. However, this implies  
\begin{align*}
\InfectedWeightKernel{i+1} = \InfectedWeightLayers{i+1} + \InfectedWeightKernel{i} >& (1-\eps_i) (\totalWeight{\Layer{i+1}} + \WeightLayers{i})\\
&\hspace{1cm} \geq (1 - \eps_{i+1}) \WeightLayers{i+1},
\end{align*}
and $\inductionLayerEvent{i+1}$ follows since $\eps_{i+1}\ge \eps_i$.

Therefore, $\inductionLayerEvent{\lastLayer}$ holds with probability at least
$$
(1-o(1))\prod_{i=0}^{\lastLayer-1}\left(1-\exp\left(-2^{i-6}\alpha^2W\weightBoundLayer{1}^{-2}\right)\right)\ge 1-o(1)-\sum_{i=1}^{\infty} 
\exp\left(-W\weightBoundLayer{1}^{-2} 2^{i-6}\alpha^2\right),
$$
by a union bound. Because $\weightBoundLayer{1}\le\weightBoundLayer{0}=\weightBoundKernel=o(\sqrt{W})$, by~\eqref{eq:kernelWeightBoundProp}, the right-hand side is $1-o(1)$, i.e.,\ whp we have 
$\InfectedWeightKernel{\lastLayer}\ge(1-\eps_{\lastLayer})\WeightLayers{\lastLayer}.$ 
\end{proof}

\begin{remark}\label{rem:almostCompleteInfection}
Observe that in the proofs of Claims~\ref{claim:largeLayerDegree} and~\ref{claim:largeLayerExpectation} and Lemmas~\ref{lem:largeLayer} and~\ref{lem:almostCompleteInfection} we only exposed edge-indicator random variables corresponding to edges in the set $\Kernel{\lastLayer}\times\bigcup_{i=1}^{\lastLayer}(\Layer{i}\setminus \infecCoupling{0})$.
\end{remark}

Now, observe that we have the (deterministic) lower bound 
\begin{equation}\label{eq:lastLayerLowerBound}
\frac{1}{2}\WeightLayers{\lastLayer}\stackrel{C. \ref{clm:weightlayer_low}}{\ge} \frac{1}{8} 
\totalWeight{\VertexSetRestricted{\ge \weightBoundLayer{\lastLayer}}} = \frac{W}{8}  \PP[\wbrv\ge\weightBoundLayer{\lastLayer}] =
\frac{W}{8}  \frac{C'}{\weightBoundLayer{\lastLayer +1}}  
\stackrel{\eqref{def:weightBoundLayer}}{\geq} \frac{C'W}{16\max\{C_1,\weightConstant \}} .
\end{equation}
In other words, we have already proven that the total weight of all eventually infected vertices is at least a constant fraction of the total weight $W$. It remains to show that this guarantees that whp a constant fraction of all vertices become infected eventually, i.e., there is an outbreak.

\subsection{Witnessing the outbreak}
To witness the outbreak, we consider the vertices in $\VertexSetRestricted{<\weightBoundLayer{\lastLayer}}\setminus\infecCoupling{0}$ of which there are at least 
\begin{equation}\label{eq:manyLightVertices}
n-|\VertexSetRestricted{\ge\weightBoundLayer{\lastLayer}}|-|\infecCoupling{0}|\ge n- \frac{W}{\weightBoundLayer{\lastLayer}}-o(n)\stackrel{\eqref{def:weightBoundLayerLast},\eqref{Wcondition}}{\ge}\left(1-\frac{\weightConstant}{2\max\{C_1,\weightConstant,1\}}-o(1)\right)n\ge \frac{1}{3}n.
\end{equation}
For any vertex $u \in \VertexSetRestricted{<\weightBoundLayer{\lastLayer}}\setminus\infecCoupling{0}$,
we will consider the random variables 
$$
\widehat{Y}_u:=\Ind{\infdeg{u}{ \lastLayer}\ge r}
$$ 
and denote their sum by 
\begin{equation}\label{eq:totalnumberofinfections}
\widehat{Y}:=\sum_{u\in\VertexSetRestricted{<\weightBoundLayer{\lastLayer}}\setminus\infecCoupling{0}}\widehat{Y}_u.
\end{equation}
\begin{lemma}\label{lem:outbreak}
There exists a constant $\gamma>0$ such that for any $U\subseteq \VertexSetRestricted{\ge\weightBoundKernel}$ satisfying~\eqref{eq:kernelAlmostInfection} and 
such that $\{n-r+1,\ldots, n\} \subseteq U$, conditional on $\inductionLayerEvent{\lastLayer}$, whp we have
$$
\widehat{Y}>\gamma n.
$$
\end{lemma}
\begin{proof}

As the $\widehat{Y}_u$s are independent, 
we will deduce this by applying the Chernoff bound (Theorem~\ref{Chernoff}). 

First of all observe that if we replace $u$ by a vertex $u_0$ of weight $w_{u_0}:=1\le w_u$, then we have
\begin{equation}\label{eq:outbreakWeightOneSuffices}
\PP\left[\widehat{Y}_u=1\cond\inductionLayerEvent{\lastLayer}\right] \ge 
\PP\left[\infdeg{u}{ \lastLayer} \ge r\cond\inductionLayerEvent{\lastLayer}\right]
\ge \PP\left[\infdeg{u_0}{ \lastLayer}=r\cond\inductionLayerEvent{\lastLayer}\right].
\end{equation}
Because $w_{u_0}=1$ and also 
\begin{equation}\label{eq:condUpperMaxWeight}
w_n\stackrel{\eqref{eq:manyLightVertices}, \ w_1\ge 1}{\le} W-n/3\stackrel{\eqref{Wcondition}}{=}(1- 1/(3\weightConstant))W,
\end{equation}
 we can drop the minimum in~\eqref{eq:pijCL}, and thus the above probability can be computed as
\begin{align}
\PP\left[\infdeg{u_0}{ \lastLayer}=r\cond\inductionLayerEvent{\lastLayer}\right]=\sum_{\mathcal{R}\in
\binom{\postKernelProcess{\lastLayer}{U}\cap\Kernel{\lastLayer}}{r}}\prod_{v\in\mathcal{R}}\frac{w_v}{W}
\prod_{v'\in(\postKernelProcess{\lastLayer}{U}\cap\Kernel{\lastLayer}) \setminus\mathcal{R}}\left(1-\frac{w_v'}{W}\right).\label{eq:outbreakInfectionProbability}
\end{align}

Because $1-x\ge \exp(-x/(1-x))$, for any $x<1$, for the innermost product we obtain
\begin{align}\nonumber
\prod_{v'\in(\postKernelProcess{\lastLayer}{U}\cap\Kernel{\lastLayer}) \setminus\mathcal{R}}\left(1-\frac{w_{v'}}{W}\right)&\ge \exp\left(-\frac{\sum_{v'\in\Kernel{\lastLayer}}w_{v'}}{W\left(1-\frac{w_n}{W}\right)}\right)\\
&\ge\exp\left(-1\Big/\left(1-\frac{w_n}{W}\right)\right)\stackrel{\eqref{eq:condUpperMaxWeight}}{\ge} \exp(-3\weightConstant), \label{eq:outbreakInnerProduct}
\end{align}
independently of $u$. Moreover, we have
\begin{equation}\label{eq:outbreakMainTerm}
\sum_{\mathcal{R}\in\binom{\postKernelProcess{\lastLayer}{U}\cap\Kernel{\lastLayer}}{r}}\prod_{v\in\mathcal{R}}w_v\ge \left.\left((\postKernelProcess{\lastLayer}{U}\cap\Kernel{\lastLayer})^r-(\postKernelProcess{\lastLayer}{U}\cap\Kernel{\lastLayer})^{r-2}\sum_{v\in\postKernelProcess{\lastLayer}{U}\cap\Kernel{\lastLayer}}w_v^2\right)\right/r!
\end{equation}
since the sum on the left-hand side ranges only over monomials corresponding to $r$ distinct vertices. Furthermore, we have 
\begin{equation}\label{eq:outbreakSumOfSquares}
\sum_{v\in\postKernelProcess{\lastLayer}{U}\cap\Kernel{\lastLayer}}w_v^2\le w_n \InfectedWeightKernel{\lastLayer},
\end{equation}
and thus in order for the right-hand side of~\eqref{eq:outbreakInfectionProbability} to be at least a (small) positive constant, it suffices to show that 
\begin{equation}\label{eq:outbreakMaxWeightSmall}
w_n\le (1-\eta)\InfectedWeightKernel{\lastLayer},
\end{equation}
for $\eta:=1-1/(1+(r-1)\alpha)$ (and note that $0 < \eta < 1$). 

We distinguish two cases, first assume that 
$$
\WeightLayers{0}-\totalWeight{\{n-r+1,\dots,n\}}< w_{n-r+1}/2.
$$
Then we observe that $w_{n-r+1}\ge\weightBoundLayer{0}\ge\weightBoundLayer{\lastLayer}$ by~\eqref{eq:weighBoundLayerNonIncreasing}  and~\eqref{eq:kernelWeightBoundProp}. 
Since $w_n\ge\ldots\ge w_{n-r+1}\ge \weightBoundLayer{0}$ and the $r$ vertices of largest weight are infected, i.e.,\ $\{n-r+1,\dots,n\}\subseteq U$,  we obtain
$$
\InfectedWeightKernel{\lastLayer}\ge\totalWeight{U} \ge \totalWeight{\{n-r+1,\dots,n\}}\ge (1+(r-1)\alpha)w_n,
$$
implying~\eqref{eq:outbreakMaxWeightSmall}. 

Otherwise we have
$$
\WeightLayers{0}\ge\totalWeight{\{n-r+1,\dots,n\}}+w_{n-r+1}/2\ge (1+(r-1/2)\alpha)w_n,
$$
and~\eqref{eq:kernelAlmostInfection} implies
$$
\InfectedWeightKernel{\lastLayer}\ge \totalWeight{U} \ge (1-o(1))(1+(r-1/2)\alpha)w_n\ge (1+(r-1)\alpha)w_n,
$$
for any sufficiently large $n$. Hence~\eqref{eq:outbreakMaxWeightSmall} also holds in this case.

Combining the bounds~\eqref{eq:outbreakInnerProduct},~\eqref{eq:outbreakMainTerm},~\eqref{eq:outbreakSumOfSquares}, and~\eqref{eq:outbreakMaxWeightSmall}, it follows from~\eqref{eq:outbreakWeightOneSuffices} and~\eqref{eq:outbreakInfectionProbability} that 
\begin{align*}
\PP\left[\widehat{Y}_u=1\cond \inductionLayerEvent{\lastLayer}\right]&\ge\exp(-3\weightConstant) \eta \frac{(\InfectedWeightKernel{\lastLayer})^r}{r! W^r}\\
&\stackrel{\eqref{eq:lastLayerLowerBound}}{\ge}\exp(-3\weightConstant) \eta\left.\left(\frac{C'}{16\max\{C_1,\weightConstant \}}\right)^r\right/r!, 
\end{align*}

where we define the right-hand side to be $4\gamma$, which is positive and independent of $U$. Thus by~\eqref{eq:manyLightVertices} and ~\eqref{eq:totalnumberofinfections} we have 
$$
 \frac{4\gamma n}{3}\le\EE\left[\widehat{Y}\cond\inductionLayerEvent{\lastLayer}\right]\le n,
$$
and hence applying the Chernoff bound (Theorem~\ref{Chernoff}) on $\widehat{Y}$ yields
\begin{equation*}
\PP\left[\widehat{Y}\le \gamma n\cond\inductionLayerEvent{\lastLayer}\right]\le \exp\left(-\frac{\left(\gamma n/3\right)^2}{2n}\right)=o(1).\qedhere
\end{equation*}
\end{proof}

\begin{remark}\label{rem:outbreak}
Observe that in the proof of Lemma~\ref{lem:outbreak} we only exposed edge-indicator random variables corresponding to edges in $(\VertexSetRestricted{<\weightBoundLayer{\lastLayer}}\setminus \infecCoupling{0})\times \Kernel{\lastLayer}$.
\end{remark}

Now, we combine the above two lemmas and prove Theorem~\ref{thm:outbreak}
\begin{proof}[Proof of Theorem~\ref{thm:outbreak}]
Let $\KernelDense$ be the event of Theorem~\ref{thm:kernelDense}, and  $\KernelSparse$ denote the event of Theorem~\ref{thm:kernelSparse}. For a given set $U\subseteq\VertexSetRestricted{\ge \weightBoundKernel}$ satisfying~\eqref{eq:kernelAlmostInfection} let $\LayersEvent{U}$ be the event that the random variable $\widehat{Y}>\gamma n$, where $\gamma>0$ is as in Lemma~\ref{lem:outbreak}. Now Lemmas~\ref{lem:almostCompleteInfection} and~\ref{lem:outbreak} imply that if $U$ satisfies~\eqref{eq:kernelAlmostInfection} and $\{n-r+1,\ldots, n\} \subseteq U$, then 
$\PP[\LayersEvent{U}]=1-o(1)$.

Recall that we are aiming to use the subsubsequence principle. Hence, consider an infinite sequence $N_0\subseteq\mathbb{N}$.

If $N_0\in\mathcal{N}_d$, then whp (over a subsequence of $N_0$) the event $\KernelDense$ is realised, i.e., we have complete 
infection of the nucleus $V_{\ge \weightBoundKernel}$. Using $U=\VertexSetRestricted{\ge\weightBoundKernel}$ 
and $\{w_{n-r+1},\ldots, w_n \} \subseteq{U}$, a union bound implies
$$
\PP[\KernelDense \cap \LayersEvent{\VertexSetRestricted{\geq \weightBoundKernel}}]= 1-o(1).
$$
The definition of $\widehat{Y}$ implies that if the event $\KernelDense \cap \LayersEvent{\VertexSetRestricted{\geq \weightBoundKernel}}$ is realised, then every vertex contributing to $\widehat{Y}$ will eventually be infected and thus we have $|\infec{F}|\ge \widehat{Y}>\gamma n$, in other words, there is an outbreak.

On the other hand, let $N_0\in\mathcal{N}_s.$ Then, whp (over a subsequence of $N_0$) the event $\KernelSparse$ is realised, i.e., we have shown almost complete
infection of the nucleus $\VertexSetRestricted{\ge \weightBoundKernel}$ through the breeding ground 
$\breedingGround \subseteq \breedingGround'$. 
Let $\mathcal{U}_K \subseteq \VertexSetRestricted{\ge \weightBoundKernel}$ be the random subset of 
$\VertexSetRestricted{\ge \weightBoundKernel}$ as in Theorem~\ref{thm:kernelSparse}.

Next we observe that the sets $\infecCoupling{0}\cup \breedingGround' \cup  \Kernel{0}$ and
 $(\cup_{i=1}^{\lastLayer} \Layer{i}) \cup \VertexSetRestricted{< \psi_i*})\setminus \infecCoupling{0}$
 are disjoint. Moreover, note that $\{ \mathcal{U}_K = U\}$ depends only on edges in $\breedingGround'\times (\infecCoupling{0}\cup \breedingGround' \cup \Kernel{0})$, whereas $\LayersEvent{U}$ depends on edges in $\Kernel{\lastLayer}\times ((\cup_{i=1}^{\lastLayer} \Layer{i}) \cup \VertexSetRestricted{< \psi_i*})\setminus \infecCoupling{0})$ (cf.\ Remarks~\ref{rem:almostCompleteInfection} and~\ref{rem:outbreak}). 
 Therefore, the two events are independent.
Thus, we have 
\begin{eqnarray*} 
\PP [\KernelSparse \cap \LayersEvent{\mathcal{U}_K}]  
&\geq & \sum\limits_{\substack{U \subseteq \VertexSetRestricted{\ge \weightBoundKernel} : 
\\ \totalWeight{U} \ge \frac{1}{2}\wcore{0}} }
\PP[\LayersEvent{U}\cap \{\mathcal{U}_K = U\}] \\
&= & \sum\limits_{\substack{U \subseteq \VertexSetRestricted{\ge \weightBoundKernel} : 
\\ \totalWeight{U} \ge \frac{1}{2}\wcore{0}} }
\PP[\LayersEvent{U}] \PP [\mathcal{U}_K = U]  \\
&=& (1-o(1))  \sum\limits_{\substack{U \subseteq \VertexSetRestricted{\ge \weightBoundKernel} : 
\\ \totalWeight{U} \ge \frac{1}{2}\wcore{0}} }\PP [\mathcal{U}_K = U]  = 1-o(1).
\end{eqnarray*}
Similarly to the previous case, if the event $\KernelSparse \cap \LayersEvent{\mathcal{U}_K}$ is realised, then every vertex contributing to $\widehat{Y}$ will eventually be infected and thus $|\infec{F}|\ge \widehat{Y}>\gamma n$, i.e.\, there is an outbreak
completing the proof of Theorem~\ref{thm:outbreak}.
\end{proof}

\section{Subcritical Regime: No Linear Outbreak}\label{subcritical}

The goal of this section is to show that in the subcritical regime whp there is no outbreak.
\begin{theorem}\label{thm:nooutbreak}
Under the assumptions of Theorem~\ref{threshold}, if additionally either
$$
\weightBoundHeavy\le w_n\qquad\text{and}\qquad\infectionRate\ll\min\{\candidateThresholdSparse,\candidateThresholdDense\}$$
or 
$$
\weightBoundHeavy> w_n\qquad\text{and}\qquad\infectionRate\ll\candidateThresholdSparse,
$$
 then there is no outbreak, that is, \ whp we have
 $$
 |\infec{F}|=o(n).
 $$
\end{theorem}
We will use the parameter 
\begin{equation}\label{def:scaledTotalWeightSubcriticalExpectationInverse}
\scaledTotalWeightSubcriticalExpectationInverse:=
\begin{cases}
\min\{\candidateThresholdSparse,\candidateThresholdDense\}/\infectionRate& \text{if  } \weightBoundHeavy\le w_n,\\
\candidateThresholdSparse/\infectionRate&\text{if } \weightBoundHeavy>w_n.
\end{cases}
\end{equation}
Since we consider the subcritical regime, we have 
\begin{equation}\label{eq:scaledTotalWeigthSubcrititcalExpectationInverse}
\scaledTotalWeightSubcriticalExpectationInverse\to\infty.
\end{equation}

Instead of tracking the infection process we relate it to a branching process, 
motivated by the following observation: if vertex $v$ becomes infected at time $t>0$ it must have at least one neighbour $u$ which became infected at time $t-1$; actually one can show that typically it has exactly one such neighbour. Hence, we may consider the vertex $v$ a child of this unique neighbour $u$ (and in case there are several, choose the smallest amongst them). Then for each vertex the number of its children is a random variable, however they may have different distributions and be dependent on each other. 

We will show that the condition on $p_0$ implies that this process is subcritical and whp it dies out quickly, thereby proving that the 
total infected population remains small. 

Some of these arguments require us to work on the subgraph spanned by all non-heavy vertices, and then argue separately for the heavy vertices. We will show that whp no heavy vertex becomes infected during the process and thus the relevant part of the proof is to analyse the behaviour of bootstrap process on the subgraph spanned by the non-heavy vertices.

We run the bootstrap process on $G[\VertexSetRestricted{<\weightBoundHeavy}]$ in the usual way (cf.\ Section~\ref{sec:bootstrapPercolation}) and we denote the set of vertices that have become infected by time $t\ge0$ by $\infecNonHeavy{t}\subseteq\VertexSetRestricted{<\weightBoundHeavy}$. Let $\infecNonHeavy{F}$ denote the set of infected vertices at the end of the process. 

\subsection{A branching process approximation}\label{sec:branchingProcess}

In this section we will prove that the total infected weight of the $\infecNonHeavy{}$-process will not be significantly larger than the total weight of the initially infected vertices.   

\begin{lemma}\label{lem:cloneProcessSmall}
	Whp we have $\totalWeight{\infecNonHeavy{F}}=o\left(\scaledTotalWeightSubcriticalExpectationInverse W\infectionRate\right)$.
\end{lemma}
We will prove this lemma by coupling the evolution of the bootstrap process with a stochastic process 
that is reminiscent of a branching process, where the offspring distribution depends on the current 
state. 

For $t\geq 0$ let $\cI{t}:=\infecNonHeavy{t}\setminus \infecNonHeavy{t-1}$ denote the set of vertices that belong to the $t$-th generation and let 
$\I{t} = \totalWeight{\cI{t}}$. In other words, the set 
$\cI{t}$ consists of the vertices that become infected in step $t$. 
Thus, $\cI{0}= \infecNonHeavy{0}$ and $\infecNonHeavy{t}= \cup_{s=0}^t \cI{s}$. 

Now, let us condition on $\infecNonHeavy{t-1}$. For every vertex $y \in \cI{t-1}$, following the 
ordering of the vertices, we expose its 
neighbours in $\VertexSetRestricted{<\weightBoundHeavy} \setminus \infecNonHeavy{t-1}$. For every $v \in \VertexSetRestricted{<\weightBoundHeavy} \setminus \infecNonHeavy{t-1}$ 
adjacent to $y$, we expose whether or not there are at least $r-1$ other 
edges between $v$ and $\infecNonHeavy{t-1}$. If this is the case, then we include $v$ into $\cI{t}$ and, in particular, we 
include the vertex $v$ among the offspring of $y$ -- we write $v \in \mathcal{X}_y$, where $\mathcal{X}_y$ denotes the set of 
offspring of $y$. This leads to a partition of $\cI{t}$ into sets
$\mathcal{X}_y$ of \emph{children} for $y\in\cI{t-1}$, i.e.,\ we have
$$
\cI{t}=\dot{\bigcup}_{y\in\infecNonHeavy{t-1}\setminus\infecNonHeavy{t-2}}\mathcal{X}_y.
$$

Using the FKG inequality (Theorem \ref{FKG}), we will show the following lemma, which bounds the probability that $v\in \mathcal{X}_y$.
\begin{lemma} \label{lem:prob_offspring}
	For any $t\geq 0$, let $v\in 
	\VertexSetRestricted{\leq \weightBoundHeavy} \setminus \infecNonHeavy{t-1}$. We have 
	$$\PP\left[v \in \mathcal{X}_y \cond \infecNonHeavy{t-1}\right]\le 
	w_y \totalWeight{\infecNonHeavy{t-1}}^{r-1}\left(\frac{w_v}{W}\right)^r.$$
\end{lemma}

\begin{proof}
	We will condition on a realisation of $\infecNonHeavy{t-1}$.
	Consider the conditional space where $\cI{i} =S_i$, for $i=0,\ldots, t-1$, and $v \not \in \cup_{i=0}^{t-1}S_{i}=\infecNonHeavy{t-1}$. 
	In order for the event $v\in \mathcal{X}_y$ to hold we need the following three conditions to hold: 
	\begin{enumerate}[label=$(\roman*)$]
	\item $y$ is a neighbour of $v$; 
	\item $v$ has $r-1$ additional infected neighbours;
	\item there exists no $z \in \infecNonHeavy{t-1}\backslash \infecNonHeavy{t-2}$ with $z<y$ such that $v \in \mathcal{X}_z$.
	\end{enumerate}
	We write $\mathcal{D}(v)$ for the event that $\sum_{j=0}^{t-2} d_{S_j} (v) < r$
	, where $d_{S_j}(v)$ is the 
	degree of $v$ in $S_j$. 
	So, if we ignore the third condition, the conditional probability of the event $v\in \mathcal{X}_y$ can be bounded as:
	\begin{eqnarray*}\lefteqn{\PP\left[v \in \mathcal{X}_y \cond \cap_{i=0}^{t-1} \{ \cI{i} =S_i \}\right]} \\
		& \leq \PP\left[ \{y,v\}\in\EdgeSet, \ \exists y_{1},\ldots, y_{r-1} \in \cup_{i=0}^{t-1}S_i: \ \cap_{i=0}^{r-1}\left\{ \{y_i,v\}\in\EdgeSet \right\}\cond \mathcal{D}(v), \ \Gamma\right], \end{eqnarray*}
	where $\Gamma$ is some event which does not depend on edges incident to $v$. Thus, we can omit $\Gamma$ from the conditioning. 

	We bound the latter probability from above using the FKG inequality (Theorem~\ref{FKG}). Note that the event 
	$\{y,v\}\in\EdgeSet, \ \exists y_{1},\ldots, y_{r-1} \in \cup_{i=0}^{t-1}S_i: \  \cap_{i=0}^{t-1} \{ \{y_i,v\}\in\EdgeSet$ is non-decreasing, whereas 
	the event $\mathcal{D}(v)$ is non-increasing. Therefore, Theorem~\ref{FKG} implies that 
	\begin{eqnarray*} 
		\lefteqn{\PP\left[ \{y,v\}\in\EdgeSet, \ \exists y_{1},\ldots, y_{r-1} \in \cup_{i=0}^{t-1}S_i: \ \cap_{i=0}^{r-1} \{\{y_i,v\}\in\EdgeSet\} \cond \mathcal{D}(v)\right]}
		\\&\leq \PP\left[ \{y,v\}\in\EdgeSet, \ \exists y_{1},\ldots, y_{r-1} \in \cup_{i=0}^{t-1}S_i: \  \cap_{i=0}^{r-1} \{ \{y_i,v\}\in\EdgeSet\}\right] \\
		& \leq w_y w_v^{r} \left(\sum_{i=0}^{t-1}\totalWeight{S_i}\right)^{r-1} \left( \frac{1}{W}\right)^{r}.
	\end{eqnarray*}
	In other words, we have
	$$\PP\left[v\in\mathcal{X}_y \cond \cap_{i=0}^{t-1} \{\cI{i} =S_i\}\right] \leq  w_y \left(\sum_{i=0}^{t-1}\totalWeight{S_i}\right)^{r-1} \left( \frac{w_v}{W}\right)^{r}$$
	and the lemma follows.
\end{proof}

\begin{proof}[Proof of Lemma \ref{lem:cloneProcessSmall}]
We now provide a stochastic upper bound on $\I{t}$ using a process that is very similar to a branching process except the offspring distribution  depends on 
the history of the process. Moreover, the number of offspring of each individual in each generation are not independent. 

Consider the family of 
Bernoulli random variables $I_{y,v}(t)$, where $y\in\cI{t-1}$ and $v$ is any vertex, which  
satisfies $I_{y,v}(t)=1$ if and only if  $v\in \mathcal{X}_y$. 
Hence, by Lemma~\ref{lem:prob_offspring}
$$ \EE \left[I_{y,v}(t) \cond \infecNonHeavy{t-1}\right] \leq w_y \totalWeight{\infecNonHeavy{t-1}}^{r-1}\left(\frac{w_v}{W}\right)^r. $$
Given $\infecNonHeavy{t-1}$ we write
$$\I{t} = \sum_{y \in \cI{t-1}} \sum_{v \in \VertexSetRestricted{<\weightBoundHeavy} \setminus \infecNonHeavy{t-1}} I_{y,v}(t) w_v. $$ 
This implies that 
\begin{eqnarray}\label{eq:expectations_evol}
\EE \left[\I{t} \cond \infecNonHeavy{t-1}\right] 
& \leq \I{t-1} \totalWeight{\infecNonHeavy{t-1}}^{r-1}W^{-r}  \sum_{ v \in \VertexSetRestricted{<\weightBoundHeavy}} w_v^{r+1}.
\end{eqnarray}
Now, we introduce the stopping time $T$ which is the first step $t$ where either $\I{t}=0$ or $\totalWeight{\infecNonHeavy{t}} > 
\mu^{1/2} Wp_0$. Note that $T < \mu^{1/2}Wp_0 $, since if $\I{t}>0$, then in fact $\I{t} \geq 1$. 

Now, let $\hat{I}(t)$ be equal to $\I{t}$, if $t\leq T$ and equal to 0 otherwise. In other words, $\hat{I}(t) = \I{t} \Ind{t\leq T}$. 
Let 
$$\offspringWeightRatio(t) := \totalWeight{\infecNonHeavy{t-1}}^{r-1}W^{-r}  
\sum_{ v \in \VertexSetRestricted{<\weightBoundHeavy}} w_v^{r+1}.$$
If $t \leq T$, then 
\begin{equation}\label{eq:upperBoundOffspringWeightRatio}
 \offspringWeightRatio(t) \leq (\mu^{1/2}Wp_0)^{r-1} W^{-r} \sum_{ v \in \VertexSetRestricted{<\weightBoundHeavy}} w_v^{r+1} =:
\hat{\offspringWeightRatio}.
\end{equation}
and furthermore this implies
\begin{equation}\label{eq:exp_recursion}
\EE \left[\hat{I}(t) \cond \infecNonHeavy{t-1}\right] \leq \hat{\offspringWeightRatio} \hat{I}(t-1).
\end{equation}

Next we show that
\begin{equation} \label{eq:gamma_bound}
\hat{\offspringWeightRatio} \leq \mu^{-1/2}.
\end{equation}
Indeed, since $p_0\leq \mu^{-1}p_s$ we have that
\begin{align*}
	(Wp_0)^{r-1}& \leq \mu^{-(r-1)}W^{r-1} p_s^{r-1}\\
	&=\mu^{-(r-1)}W^{r-1}  \frac{W}{\sum_{ v \in \VertexSetRestricted{<\weightBoundHeavy}} w_v^{r+1}}\\
	&= \mu^{-(r-1)} \frac{W^r}{\sum_{ v \in \VertexSetRestricted{<\weightBoundHeavy}} w_v^{r+1}}. 
\end{align*}
This together with \eqref{eq:upperBoundOffspringWeightRatio} imply
$$ 
\hat{\offspringWeightRatio} \leq \mu^{-(r-1)/2},
$$
and (\ref{eq:gamma_bound}) follows since $r\ge 2$.

Therefore, taking expectations on both sides of (\ref{eq:exp_recursion}) we deduce that 
$$ \EE \left[\hat{I}(t) \right] \leq \hat{\offspringWeightRatio} \EE\left[\hat{I}(t-1)\right]. $$
Repeating this inequality, we obtain
$$	\EE \left[\hat{I}(t) \right] \leq \hat{\offspringWeightRatio}^t \EE\left[\hat{I}(0)\right] 
	= \hat{\offspringWeightRatio}^t \EE\left[\totalWeight{\infecNonHeavy{0}}\right].$$
Note that $\EE[\totalWeight{\infecNonHeavy{0}}] \leq Wp_0$ and thus
$$\EE \left[\hat{I}(t) \right] \leq \hat{\offspringWeightRatio}^t W p_0,$$
implying
\begin{equation} \label{eq:total_exp}
\EE\left[\totalWeight{\infecNonHeavy{T}}\right]=\EE \left[ \sum_{t=0}^T \hat{I}(t) \right] \leq \left( \sum_{t=0}^T \hat{\offspringWeightRatio}^t \right) 
W p_0  \leq 
\frac{1}{1-\hat{\offspringWeightRatio}} W p_0.
\end{equation}

Now, let $\mathcal{B}$ be the event $\sum_{t=0}^T \hat{I}(t) \geq \mu^{1/2}W p_0 $.
We have 
$$ \EE \left[ \Ind{\mathcal{B}} \sum_{t=0}^T \hat{I}(t) \right]  \geq \mu^{1/2}W p_0  \EE\left[\Ind{\mathcal{B}} \right]. $$
This together with (\ref{eq:total_exp}) imply that 
$$ \EE [\Ind{\mathcal{B}}] =O(\mu^{-1/2})=o(1).  $$ 
Therefore Markov's inequality implies that whp the process stops before the total weight of the infected vertices reaches $\mu^{1/2}Wp_0$, as desired.

\end{proof}

\subsection{No outbreak: proof of Theorem~\ref{thm:nooutbreak}}\label{sec:nooutbreak}

Now we consider the process on the whole vertex set $V$. 
Until this point we showed that if we restrict ourselves to the non-heavy vertices, then no linear outbreak occurs. We now have to take care of the heavy vertices. The first observation is that initially whp none of them are infected.
\begin{claim}\label{eq:largeWeightsInitial}
Whp $\mathcal{A}_0\cap\VertexSetRestricted{\ge\weightBoundHeavy}=\emptyset$.
\end{claim}
\begin{proof}
In the case $\weightBoundHeavy>w_n$ there is nothing to prove, hence assume $\weightBoundHeavy\leq w_n$. 
Recall that in this case $\candidateThresholdDense$ is well-defined and by its definition ~\eqref{eq:candidateThresholdDense} we have
\begin{equation}\label{eq:candidateThresholdDenseEstimate}
\candidateThresholdDense^{-r}=\sum_{u\in\VertexSetRestricted{\ge\weightBoundHeavy}}w_u^r\ge|\VertexSetRestricted{\ge\weightBoundHeavy}|\weightBoundHeavy^r\stackrel{\eqref{eq:defHeavy}}{\ge}\left(\frac{W}{4\weightBoundHeavy}\right)^r.
\end{equation}
Thus, since  $|\VertexSetRestricted{\ge\weightBoundHeavy}|\le W/\weightBoundHeavy$, we obtain
$$
\EE[|\infec{0}\cap\VertexSetRestricted{\ge\weightBoundHeavy}|]=|\VertexSetRestricted{\ge\weightBoundHeavy}|\infectionRate
\stackrel{\eqref{def:scaledTotalWeightSubcriticalExpectationInverse}}{\le}\frac{\candidateThresholdDense W}{\scaledTotalWeightSubcriticalExpectationInverse\weightBoundHeavy}
\stackrel{\eqref{eq:candidateThresholdDenseEstimate}}{\leq}\frac{4}{\scaledTotalWeightSubcriticalExpectationInverse}
\stackrel{\eqref{eq:scaledTotalWeigthSubcrititcalExpectationInverse}}{=}o(1).
$$
Therefore the claim follows from Markov's inequality.
\end{proof}

\begin{proof}[Proof of Theorem~\ref{thm:nooutbreak}]
Now we consider the unrestricted process as described in Section~\ref{sec:bootstrapPercolation}, but starting with $\infec{0}':=\infecNonHeavy{F}$ as the initial set of infected vertices. This defines a sequence of sets $\infec{\tau}'$ for $\tau\ge 0$ and a final set $\infec{F}'$. Now Claim~\ref{eq:largeWeightsInitial} implies that whp $\infec{0}\subseteq\infecNonHeavy{F}=\infec{0}'$ and thus 
\begin{equation}\label{eq:finalSets}
\infec{F}\subseteq\infec{F}'.
\end{equation}

Lemma~\ref{lem:cloneProcessSmall} implies that whp the total weight of $\infecNonHeavy{F}$ satisfies
\begin{align}\label{eq:defrostingTotalWeightUpperBound}
\totalWeight{\infec{0}'}=\totalWeight{\infecNonHeavy{F}}=o\left(\scaledTotalWeightSubcriticalExpectationInverse W\infectionRate\right).
\end{align}
Moreover, this shows that $\infecNonHeavy{F}$ contains only few vertices
\begin{equation}\label{eq:smallFinalSet}
|\infec{0}'|\le\totalWeight{\infec{0}'}/w_1\stackrel{w_1\ge 1,\eqref{def:scaledTotalWeightSubcriticalExpectationInverse},Cor.\ref{cor:candidatethresholdsSmall}}{=}o(W)\stackrel{\eqref{Wcondition}}{=}o(n).
\end{equation}

The last step is to show that whp 
\begin{equation}\label{eq:largeWeights}
\infec{1}'=\infec{0}',
\end{equation}
 because due to~\eqref{eq:finalSets} and~\eqref{eq:smallFinalSet} this will imply 
 $$
 |\infec{F}|\le|\infec{F}'|=|\infec{0}'|=o(n).
 $$

It remains to prove~\eqref{eq:largeWeights}. In other words, we have to show that whp in the next step none of the 
heavy vertices become infected. 
Once again in case $\weightBoundHeavy>w_n$ there is nothing to be shown, so assume $\weightBoundHeavy\le w_n$. For any vertex $v\in\VertexSetRestricted{\ge\weightBoundHeavy}$ we have 
$$
\PP\left[v\in\infec{1}'\cond \infec{0}'\right] 
\stackrel{L.\ref{infecprobu}}{\le}\frac{w_v^r\totalWeight{\infec{0}'}^r}{r!W^r}
$$
and note that only heavy vertices can become infected in this step. Thus, summing over all heavy vertices and using \eqref{eq:defrostingTotalWeightUpperBound}, we obtain
$$
\EE\left[\infec{1}'\setminus\infec{0}'\cond\infec{0}'\right]=o\left((\scaledTotalWeightSubcriticalExpectationInverse\infectionRate)^r\sum_{v\in\VertexSetRestricted{\ge\weightBoundHeavy}}w_v^r\right).
$$
Moreover we have
$$
\scaledTotalWeightSubcriticalExpectationInverse\infectionRate\stackrel{\eqref{def:scaledTotalWeightSubcriticalExpectationInverse}}{\le}\candidateThresholdDense\stackrel{\eqref{eq:candidateThresholdDense}}{=}\frac{1}{\sum_{u\in\VertexSetRestricted{\ge\weightBoundHeavy}}w_u^r},
$$
implying 
$$
\EE\left[\infec{1}'\setminus\infec{0}'\cond\infec{0}'\right]=o(1).
$$
Therefore Markov's inequality implies that whp~\eqref{eq:largeWeights} holds, and as argued previously this completes the proof of Theorem~\ref{thm:nooutbreak}.
\end{proof}

\section{Proof of main results}
\subsection{Proof of Theorem~\ref{threshold}}\label{sec:threshold}
Theorem~\ref{threshold} follows directly from Theorems~\ref{thm:kernel} and~\ref{thm:outbreak}. 
\subsection{Proof of Theorem~\ref{nothreshold}}\label{sec:nothreshold}
Recall that the assumption on the weight sequence in Theorem \ref{nothreshold} is that there exist constants $c < 1/30$, $c_1$ and a function $h=\omega(1)$ such that for  $c_1\leq f \leq h$ we have $\PP[\wbrv\geq f]\leq c/f$.

For the remainder of the section assume that $\infectionRate\geq h^{-1}$ and $\infectionRate=\omega(W^{-1/2})$. 
Clearly, if whp there is no outbreak for such a $\infectionRate$, then this is the case for any $\infectionRate =o(1)$. 
Moreover, note that if there is no linear outbreak for $r=2$, then there is no linear outbreak for $r>2$. So we may restrict ourselves to the $r=2$ case.

We aim to apply the branching process argument from Section \ref{subcritical} here as well. 
Similarly as before we could construct the branching process, but unlike the previous case the typical vertex would have multiple parents during the early stages of the process. We use Theorem \ref{mcdct} instead to track the total weight of the infected vertices during the early stages of the process. After a point, due to the rapid decrease in the sum of the weight of the vertices which become 
infected in a given step, a typical vertex which becomes infected in step $t-1$ has on average less than one 
child in the next generation (in the sense we discussed in section \ref{subcritical}), giving rise to a subcritical 
process which dies out quickly. In turn, this implies that the bootstrap process stops quickly.  

We modify the initial step of the process slightly which leads to a stochastic upper bound: in the initial step, we infect every vertex with weight at least $\infectionRate^{-1}$ and in addition we infect every vertex with weight less than $\infectionRate^{-1}$ with probability $\infectionRate$ independently. We denote this set 
by $\infeclast{0}$. More generally, for any $t\geq 0$ we let $\infeclast{t}$ be the set of infected vertices 
after the $t$-th step. As usual, we set $\infeclast{-1}:=\emptyset$.

Let
$$\sbwsmoment{k}:= W^{-1}\sum_{v\in\VertexSetRestricted{<\infectionRate^{-1}}}w_v^{k+1}.$$ 
Using Proposition~\ref{prop:restrmom}, we bound $\sbwsmoment{k}$ from above by
\begin{align} 
\sbwsmoment{k}
&\leq  k\int_{0}^{\infectionRate^{-1}}\mu^{k-1}\PP\left[\wbrv\geq \mu\right]\mathrm{d}\mu \nonumber\\
&\leq k\int_{c_1}^{\infectionRate^{-1}}\mu^{k-1}\PP\left[\wbrv\geq \mu\right]\mathrm{d}\mu +k\int_{0}^{c_1}\mu^{k-1}\mathrm{d}\mu \nonumber\\
&\leq \label{moments}\begin{cases}
c\ln\left(\infectionRate^{-1}\right)+c_1\, ,& k=1\\
c\frac{k}{k-1}\infectionRate^{-k+1}+c_1^k\, , &k\geq 2.
\end{cases}
\end{align}

\begin{lemma}\label{lem:initial_infected_weight}
	Whp $\totalWeight{\infeclast{0}}\leq (1+2c) W\infectionRate$.
\end{lemma}

\begin{proof}
	Since $\infectionRate^{-1}\leq h$ and $\infectionRate^{-1} \to \infty$, we have  $\sum_{u\in\VertexSetRestricted{\ge\infectionRate^{-1}}} w_u \leq cW/\infectionRate^{-1}=cW \infectionRate$. Clearly, $\EE[\totalWeight{\infeclast{0}\cap\VertexSetRestricted{<\infectionRate^{-1}}}]\leq W\infectionRate$. 
	Theorem \ref{mcdct} implies that for large $n$ we have
	\begin{align*}
	\PP\left[\totalWeight{\infeclast{0}\cap\VertexSetRestricted{<\infectionRate^{-1}}}\geq W\infectionRate+cW\infectionRate\right]&\leq \exp\left(-\frac{(cW\infectionRate)^2}{2(\sum_{u\in\VertexSetRestricted{\le\infectionRate^{-1}}}w_u^2 \infectionRate+\infectionRate^{-1}cW\infectionRate/3)}\right)\\
	&\stackrel{\eqref{moments}}{\leq} \exp\left(-\frac{(cW\infectionRate)^2}{2(c(1+o(1))W\infectionRate\ln{\infectionRate^{-1}}+\infectionRate^{-1}cW\infectionRate/3)}\right)\\
	&\leq \exp\left(-\frac{(cW\infectionRate)^2}{cW}\right)\\
	&\leq \exp\left(-cW\infectionRate^2\right)=o(1),
	\end{align*}
	because $\infectionRate^2=\omega(W^{-1})$.
\end{proof}

\begin{lemma}\label{lem:vnotinAt}
	If $v\not\in \infeclast{t}$, then
	$$\PP[v\in \infeclast{t+1}| \infeclast{t},\infeclast{t-1}]\leq \frac{w_v^2}{W^2}\frac{\totalWeight{\infeclast{t}}^2-\totalWeight{\infeclast{t-1}}^2}{2}.$$
\end{lemma}

\begin{proof}
	Given a vertex $v$, in order for $v$
	to become infected at step $t+1$ it must neighbour at least $2$ vertices in $\infeclast{t}$ and at least one of these vertices must be contained in $\infeclast{t}\setminus\infeclast{t-1}$. For a set of vertices $U$ let $\mathcal{E}_U$ be the event that for every $u\in U$ we have that the pair $(u,v)$ is an edge. If $v\not\in \infeclast{t}$, then we have
	$$\PP[v\in \infeclast{t+1}| \infeclast{t},\infeclast{t-1}]\leq \sum_{k=1}^2 \sum_{U\in \binom{\infeclast{t}\setminus \infeclast{t-1}}{k}}\sum_{U'\in\binom{\infeclast{t-1}}{2-k}} 
	\PP[\mathcal{E}_U,\mathcal{E}_{U'}|\infeclast{t},\infeclast{t-1}].$$
	The event $\mathcal{E}_U$ is independent of $\mathcal{E}_{U'},\infeclast{t},\infeclast{t-1}$, because it depends on edges that have not been exposed up to time $t$. 
	Now, the event $\mathcal{E}_{U'}$ is non-decreasing, whereas $\infeclast{t},\infeclast{t-1}$'s only connection to $\mathcal{E}_{U'}$ is the event that 
	$v$ has at most one neighbour in $\mathcal{A}_{t-1}$. But the latter is a non-increasing event. Hence, the FKG inequality (Theorem~\ref{FKG}) implies that 
	$\PP[\mathcal{E}_{U'}|\infeclast{t},\infeclast{t-1}]\leq \PP[\mathcal{E}_{U'}]$.
	Thus, we have
	$$\PP[\mathcal{E}_U,\mathcal{E}_{U'}|\infeclast{t},\infeclast{t-1}]\leq \PP[\mathcal{E}_U] \PP[\mathcal{E}_{U'}]=\frac{w_v^2}{W^2} \prod_{u\in U}w_u \prod_{u'\in U'}w_{u'}.$$
	Therefore
	\begin{align*}
	\PP[v\in \infeclast{t+1}\mid \infeclast{t},\infeclast{t-1}]
	&\leq \sum_{k=1}^2 \sum_{U\in \binom{\infeclast{t}\setminus \infeclast{t-1}}{k}}\sum_{U'\in\binom{\infeclast{t-1}}{2-k}} \frac{w_v^2}{W^2} \prod_{u\in U}w_u \prod_{u'\in U'}w_{u'}\\
	&\leq\frac{w_v^2}{W^2} \sum_{k=1}^2 \frac{\totalWeight{\infeclast{t}\setminus\infeclast{t-1}}^k}{k!} \frac{\totalWeight{\infeclast{t-1}}^{2-k}}{(2-k)!}\\
	&=\frac{w_v^2}{W^2} \frac{\totalWeight{\infeclast{t}}^2-\totalWeight{\infeclast{t-1}}^{2}}{2},\\
	\end{align*}
	where the last equality follows, by writing $\totalWeight{\infeclast{t}}= \totalWeight{\infeclast{t}\setminus\infeclast{t-1}} + \totalWeight{\infeclast{t-1}}$.
\end{proof}

Since $c<1/30$ we have that $3c(1+2c) < 1/8$.
Set $\beta=1/4$ and note that
\begin{equation}\label{betacond}
(1-2\beta)\beta=1/8 > 3c(1+2c).
\end{equation}

\begin{lemma}\label{indsubcrit}
	Assume that $\totalWeight{\infeclast{\tau}\setminus\infeclast{\tau-1}}\leq (1+2c)(2\beta)^{\tau}W\infectionRate$ for every $\tau\leq t$. 
	Conditional on $\infeclast{t},\infeclast{t-1}$, with probability at least
	$$1-\exp\left(-\frac{\beta(2\beta)^{t}(1+2c)W\infectionRate^2}{3}\right),$$
	we have 
	$$\totalWeight{\infeclast{t+1}\setminus\infeclast{t}}\leq (2\beta)^{t+1}(1+2c)W\infectionRate.$$ 
\end{lemma}

\begin{proof}
	From Lemma~\ref{lem:vnotinAt} we have  
	$$
	\EE\left[\totalWeight{\infeclast{t+1}\setminus\infeclast{t}}\cond \infeclast{t} , \infeclast{t-1} \right]\leq \sum_{v\in \VertexSetRestricted{<\infectionRate^{-1}}} \frac{w_v^{3}}{W^2}\frac{\totalWeight{\infeclast{t}}^2-\totalWeight{\infeclast{t-1}}^2}{2}.
	$$
	
	Clearly, for $n$ sufficiently large, by \eqref{moments} we have
	$$
	\sum_{v\in \VertexSetRestricted{<\infectionRate^{-1}}} w_v^{3}\leq (1+o(1)) 2 c W \infectionRate^{-1}$$
	and thus
	\begin{align*}
	\EE[\totalWeight{\infeclast{t+1}\setminus\infeclast{t}} \mid \infeclast{t}, \infeclast{t-1}] &\leq (1+o(1))2c\, (W\infectionRate)^{-1} \frac{\totalWeight{\infeclast{t}}^2-\totalWeight{\infeclast{t-1}}^2}{2}\\
	&\leq (1+o(1))2c\, (W\infectionRate)^{-1} (\totalWeight{\infeclast{t}}-\totalWeight{\infeclast{t-1}})\totalWeight{\infeclast{t}},
	\end{align*}
	because $\totalWeight{\infeclast{t-1}}+\totalWeight{\infeclast{t}}\leq 2 \totalWeight{\infeclast{t}}$.
	
	From the assumptions we have 
	\begin{align*}
	\EE[\totalWeight{\infeclast{t+1}\setminus \infeclast{t}}\mid \infeclast{t}, \infeclast{t-1}] & \leq (1+o(1))2c\, (W\infectionRate)^{-1} (\totalWeight{\infeclast{t}}-\totalWeight{\infeclast{t-1}})\totalWeight{\infeclast{t}}\\
	& \hspace{-8ex}\stackrel{\infeclast{-1} =\emptyset}{=} (1+o(1))2c\, (W\infectionRate)^{-1} (\totalWeight{\infeclast{t}}-\totalWeight{\infeclast{t-1}})\sum_{\tau=0}^t \totalWeight{\infeclast{\tau }\setminus \infeclast{\tau-1}}\\
	&\hspace{-8ex}\leq (1+o(1))2c\, (W\infectionRate)^{-1} (\totalWeight{\infeclast{t}}-\totalWeight{\infeclast{t-1}})\left((1+2c)\, W\infectionRate\sum_{\tau=0}^{t}(2\beta)^\tau\right)\\
	&\hspace{-8ex}\leq (\totalWeight{\infeclast{t}}-\totalWeight{\infeclast{t-1}})3c(1+2c)/(1-2\beta)\\
	&\hspace{-8ex}\stackrel{\eqref{betacond}}{\leq} \beta (\totalWeight{\infeclast{t}}-\totalWeight{\infeclast{t-1}})\\
	&\hspace{-8ex}\leq \beta (2\beta)^t(1+2c)W\infectionRate.
	\end{align*}
	
	A similar argument provides
	\begin{align*}
	\VV[\totalWeight{\infeclast{t+1}\setminus\infeclast{t}}]&\leq \sum_{v\in \VertexSetRestricted{<\infectionRate^{-1}}} \frac{w_v^{4}}{W^2}\frac{\totalWeight{\infeclast{t}}^2-\totalWeight{\infeclast{t-1}}^2}{2}
	\leq \beta (2\beta)^t(1+2c)W.
	\end{align*}
	
	Now all we have to show is concentration. We will use Theorem~\ref{mcdct} 
	for the sum of the weighted Bernoulli-distributed random variables, which gives 
	$\totalWeight{\infeclast{t+1}\setminus\infeclast{t}}$. Each summand is bounded from above by $\infectionRate^{-1}$. 
	Hence, we can take $M=\infectionRate^{-1}$ in Theorem~\ref{mcdct} and deduce that
	\begin{equation*}
	\begin{split}
	\lefteqn{\PP\left[\totalWeight{\infeclast{t+1}\setminus\infeclast{t}}\geq (2\beta)^{t+1}(1+2c)\, W\infectionRate  \mid \infeclast{t}, \infeclast{t-1}\right]}\\
	&\leq \exp\left(-\frac{(\beta(2\beta)^{t}(1+2c)W\infectionRate)^2}{2(\beta(2\beta)^{t}(1+2c)\, W+\infectionRate^{-1}\beta(2\beta)^{t}(1+2c)W\infectionRate/3)}\right)\\
	&\leq \exp\left(-\frac{(\beta(2\beta)^{t}(1+2c)W\infectionRate)^2}{3\beta(2\beta)^{t}(1+2c)W}\right)\\
	&\leq \exp\left(-\frac{\beta(2\beta)^{t}(1+2c)W\infectionRate^2}{3}\right). 
	\end{split}
	\end{equation*}
\end{proof}

\begin{proof}[Proof of Theorem \ref{nothreshold}]
	Fix a function $T:=T(n)$ satisfying $T\to\infty$, but $(2\beta)^{T}W\infectionRate^2=\omega(1)$. Note that such a function exists, because $W\infectionRate^2=\omega(1)$. Let $\event$ denote the event $\totalWeight{\infeclast{T-1}}\leq (1+2c)W\infectionRate/(1-2\beta)$ and $\totalWeight{\infeclast{T}\setminus \infeclast{T-1}}=o(W\infectionRate)$.
	
	Now, by Lemma $\ref{indsubcrit}$, with probability 
	$$1-O\left(\exp\left(-\frac{\beta(2\beta)^{T}(1+2c)W\infectionRate^2}{3}\right)\right)=1-o(1)$$
	we have $\totalWeight{\infeclast{\tau}\setminus \infeclast{\tau-1}}\leq (2\beta)^{\tau}(1+2c)W\infectionRate$ for every $\tau\leq T$ and thus $\event$ holds whp as well. From this point on we will condition on the event $\event$.

We will give a stochastic upper bound on the evolution of the bootstrap process after step $T$, 
using the branching process framework we introduced in Section~\ref{subcritical}. 

Recall that $\mathcal{X}_v$ denotes the children of $v$. Consider the family of Bernoulli random variables $I_{v,u}(t)$, where $v\in\cI{t-1}$ and $u$ is any vertex, which  satisfies $I_{v,u}(t)=1$ if and only if  $u\in \mathcal{X}_v$. 

Following the steps of Lemma~\ref{lem:prob_offspring}, one can show that 
$$ \PP\left[u\in \mathcal{X}_v \mid \infeclast{t-1},\event\right] \leq  w_v\totalWeight{\infeclast{t-1}}\left(\frac{w_u}{W}\right)^2.$$
Hence
\begin{equation*}
\EE \left[I_{v,u}(t) \mid \infeclast{t-1},\event\right] \leq  w_v\totalWeight{\infeclast{t-1}}\left(\frac{w_u}{W}\right)^2.
\end{equation*}

Let $S$ be the stopping time which is the first step $t$ after $T$ where either $\infeclast{t} = \infeclast{t-1}$ 
or $\totalWeight{\infeclast{t}} > 4W\infectionRate$. 
In other words, $S$ is the first time where either the process dies out or the total weight of the set $\infeclast{t}$ exceeds 
$4W\infectionRate$.  

Denote by $I_v(t)$ the total weight of the offspring of $v$, that is,
\begin{equation*}
I_v(t) = \sum_{u \in \VertexSetRestricted{<\infectionRate^{-1}} \setminus \infeclast{t-1}} I_{v,u}(t) w_u.
\end{equation*}
Now, we show that if $t \leq S$, then provided that $c$ is small enough, the expected value of $I_v (t)$ is smaller than 1. 
We have conditional on $\event$
\begin{eqnarray*} 
	\Ex{I_v(t) \cond \infeclast{t-1}} \leq  w_v\frac{ \totalWeight{\infeclast{t-1}}}{W^2} 
	\sum_{u \in \VertexSetRestricted{<\infectionRate^{-1}}} w_u^3
	\stackrel{(\ref{moments})}{\leq} 3 w_v c\, \totalWeight{\infeclast{t-1}} (W\infectionRate)^{-1}.
\end{eqnarray*}
Now, as long as $t\leq S$, we have $\totalWeight{\infeclast{t-1}} \leq 4 W\infectionRate$. 
We deduce that $$ \Ex{ I_v(t) \Ind{t\leq S} \cond \infeclast{t-1}, \event} \leq 12 c w_v.$$

Set $z=12 c$. Since $12 c < 1$ (from our assumption that $c< 1/30$), we have $z<1$, 
for any $n$ sufficiently large. 
Denote by $I(t)$ the total weight of the vertices in the $t$-th generation. Then this is equal to the sum of the random variables $I_v(t)$ over all $v$ which 
belong to the $(t-1)$-th generation. 
Therefore, we have
$$  \Ex{I(t)\Ind{t\leq S} \cond \infeclast{t-1}, \event} \leq z I(t-1)\Ind{t\leq S}.$$ 
Now, we set $\hat{I}(t)=I(t)\Ind{t\leq S}$. 
Taking expectations on both sides of the above inequality implies that for $t>T$
$$\Ex{\hat{I}(t) \cond \event} \leq z \Ex{\hat{I}(t-1) \cond \event} \leq  \cdots \leq z^{t-T} \EE\left[\hat{I}(T)\cond \event\right].$$ 
The random variable
$$ X_S = \sum_{t=T}^\infty \hat{I}(t)$$ 
is the total progeny (but without the first $T$ generations) until the stopping time $S$. 
We deduce that 
$$ \Ex{X_S \mid \event} \leq \EE\left[\hat{I}(T) \cond \event\right]\sum_{t=T}^\infty z^{t-T} = \frac{1}{1-z}\EE[\hat{I}(T)\mid \event]=o(W p_0),$$
as $\hat{I}(T)\leq I(T)=\totalWeight{\infeclast{T}\setminus \infeclast{T-1}}=o(W\infectionRate)$ when $\event$ holds.

Markov's inequality implies
$\PPo{X_S > Wp_0 \mid \event} = o(1)$. 
Recall that $\event$ also implies that $\totalWeight{\infeclast{T-1}}\leq 3W\infectionRate$. The result follows as $\event$ holds whp.

\end{proof}

\section{Concluding remarks}\label{sec:discussion}

To sum up, we have considered the evolution of the classical bootstrap percolation processes on a general class of inhomogeneous random graphs, 
which is known as the Chung-Lu model. In this model, the vertices are equipped with positive weights and each potential edge is present with 
probability proportional to the product of the weights. Essentially the typical properties of the resulting random graph are determined by the sequence of the weights and its asymptotic properties. 

We gave an approximate characterisation of those weight sequences for which the evolution of a bootstrap process exhibits a critical phenomenon. 
This is connected to the existence of a critical density of the initial set of infected vertices such that when the initial density crosses this value 
then an outbreak occurs whp even if the initial set is small. 
Our main finding has to do with the existence of constants $c$ and $C$ such that if $\PP[\wbrv\geq f]< c/f$, then no such critical value exists 
whereas if $\PP[\wbrv\geq f]> C/f$, then it does exist. 

The results we have shown assume that $W=\weightConstant n$ and the minimum weight is at least 1. However the results also hold under the more general assumption that $W= (1+o(1)) \weightConstant n$ and the minimal weight is bounded away from 0. However if we remove the condition that the minimum weight is at least 1, then the lower bound on $C$ in Theorem \ref{threshold} should also depend on this quantity.

As mentioned earlier the smaller of the two candidate thresholds gives the critical threshold. In the following example we demonstrate that either of the candidate thresholds can be the minimum.

\begin{example}\label{ex:example} 
	Fix $r=2$ and take a  weight sequence such that it contains $W^{1/9}$ vertices of weight $W^{7/12}$, $W^{2/3}/20$ vertices of weight $W^{1/3}$ and each of the remaining vertices has weight $1$. Note that $|\VertexSetRestricted{\ge W^{1/3}}|=(1+o(1))W^{2/3}/20< W^{2/3}/2^{4}=W^2/(2W^{1/3})^{4}$ and $|\VertexSetRestricted{\ge 1}|=n=o(W^2)$. In addition $|\VertexSetRestricted{\ge\sqrt{W}}|=W^{1/9}\geq 2^{-4}= W^2/(2\sqrt{W})^{4}$.  Therefore $W^{1/2}>\weightBoundHeavy>W^{1/3}$. Thus
	$$\sum_{u\in\VertexSetRestricted{<\weightBoundHeavy}}w_u^{3}=\Theta\left( W^{5/3} \right).$$
	Clearly $\sum_{u\in\VertexSetRestricted{\ge \weightBoundHeavy}} w_u^{2}=W^{23/18}$. This gives us 
	$$\frac{W}{\sum_{u\in\VertexSetRestricted{<\weightBoundHeavy}}w_u^3}=\Theta(W^{-2/3})=o(W^{-23/36})=o\left(\left(\frac{1}{\sum_{u\in\VertexSetRestricted{\ge \weightBoundHeavy}}w_u^2}\right)^{1/2}\right).$$
	Thus the minimum in Theorem~\ref{threshold} is achieved by the vertices of weight less than $\weightBoundHeavy$. 
	
	Now replace the $W^{1/9}$ vertices of weight $W^{7/12}$ by $W^{1/9}$ vertices of weight $W^{3/4}$. Note that this has no effect on the value of $\weightBoundHeavy$. However, we have 
	$$\left(\frac{1}{\sum_{u\in\VertexSetRestricted{\ge \weightBoundHeavy}}w_u^2}\right)^{1/2}=\Theta(W^{-29/36})=o(W^{-2/3}),$$
	so this time the minimum is achieved by the vertices of weight at least $\weightBoundHeavy$. 
\end{example}

Several open questions still remain. The first concerns the exact values of $c$ and $C$. In particular, is $c=C$? 
If not, then what happens in the case where the weight sequences satisfies $\PP[\wbrv\geq f]=c'/f$ for $c<c'<C$.

Furthermore, our analysis does not consider the case where the initial density has the same order of magnitude as the critical density. We believe 
that in this case an outbreak occurs with probability that is asymptotically bounded away from 0 and 1. If this is the case, it would be interesting to
know whether a limiting value exists for this probability and how it depends upon the parameters of the model. 

\subsection*{Acknowledgement} We would like to thank an anonymous referee for making helpful comments 
on the presentation of the paper.

\bibliographystyle{plain}
\bibliography{btsrp}

\end{document}